\documentclass[11pt, reqno]{amsart} 
\usepackage{amssymb,amscd,amsfonts,amsbsy}
\usepackage{latexsym}
\usepackage{exscale}
\usepackage{amsmath,amsthm,amsfonts}
\usepackage{mathrsfs}
\usepackage{xcolor} 
\usepackage[colorlinks=true,linkcolor=blue,citecolor=red,urlcolor=blue
]{hyperref} 
\usepackage{esint} 
\usepackage{graphicx}
\usepackage[figurewithin=none]{caption}

\usepackage{tikz}
\tikzstyle{thin}=[line width=1.5pt]
\tikzstyle{fat}=[line width=2pt]
\tikzstyle{heavier}=[line width=3pt]
\tikzstyle{ultrafat}=[line width=4pt]
\tikzstyle{ultrafat-pt}=[line width=7pt]
\definecolor{myred}{HTML}{E53935}
\definecolor{myblue}{HTML}{1E88E5}
\definecolor{mygreen}{HTML}{43A047}
\definecolor{myyellow}{HTML}{FDD835}
\definecolor{myorange}{HTML}{FB8C00}
\definecolor{mygold}{HTML}{F9A825}
\definecolor{mypurple}{HTML}{8E24AA}
\definecolor{mygray}{HTML}{BDBDBD}
\definecolor{mybrown}{HTML}{6D4C41}
\definecolor{mynavy}{HTML}{1A237E}
\definecolor{mypink}{HTML}{ffbfca}
\definecolor{myseagreen}{HTML}{26A69A}
\definecolor{myviolet}{HTML}{f07ef0}
\definecolor{mydarkblue}{HTML}{0D47A1}
\definecolor{mydarkcyan}{HTML}{E0FFFF}
\definecolor{darkgray}{rgb}{0.66, 0.66, 0.66}
\definecolor{mydarkgreen}{HTML}{1B5E20}
\definecolor{mydarkmagenta}{HTML}{AD1457}
\definecolor{mydarkorange}{HTML}{EF6C00}
\definecolor{lightblue}{rgb}{0.68, 0.85, 0.9}
\definecolor{lightcyan}{rgb}{0.88, 1.0, 1.0}
\definecolor{lightgray}{rgb}{0.83, 0.83, 0.83}
\definecolor{mylightgreen}{HTML}{81C784}
\definecolor{lightyellow}{rgb}{1.0, 1.0, 0.88}
\definecolor{myshadow}{rgb}{0.5, 0.5, 0.5}
\definecolor{pink}{rgb}{1.0, 0.75, 0.8}
\definecolor{violet}{rgb}{0.93, 0.51, 0.93}
\definecolor{myauxcolor}{RGB}{245, 255, 255}
\definecolor{mylightgreen}{RGB}{193, 225, 159}

\newcommand{\cL}{\mathcal{L}}

\usepackage[utf8]{inputenc}

\calclayout

\newcommand{\miT}{\mathrm{T}}

\usepackage{enumitem}
\newlist{myitemize}{itemize}{1}
\setlist[myitemize,1]{label=\textbullet,leftmargin=.75cm}

\allowdisplaybreaks

\theoremstyle{plain}
\newtheorem{theorem}{Theorem}[section]
\newtheorem{lemma}[theorem]{Lemma}
\newtheorem{corollary}[theorem]{Corollary}
\newtheorem{proposition}[theorem]{Proposition}

\theoremstyle{definition}
\newtheorem{definition}[theorem]{Definition}

\newtheorem*{assumption}{Assumption}

\theoremstyle{remark}
\newtheorem{remark}[theorem]{Remark}

\numberwithin{equation}{section} 

\newcommand{\mipie}[1]{{\let\thefootnote\relax\footnotetext{\hspace{-.55cm}#1}}}

\renewcommand{\emptyset}{\mbox{\textup{\O}}}

\newcommand{\abs}[1]{\left\lvert #1 \right\rvert}

\newcommand{\norm}[1]{\left\lVert #1 \right\rVert}

\def\XXint#1#2#3{{\setbox0=\hbox{$#1{#2#3}{\int}$ }
		\vcenter{\hbox{$#2#3$ }}\kern-.585\wd0}}

\def\diam{\operatorname{diam}}
\def\div{\operatorname{div}}
\def\dist{\operatorname{dist}}

\def\supp{\operatorname{supp}}

\def\Reg{{\operatorname{Reg}}}
\def\Sing{{\operatorname{Sing}}}
\def\Vol{{\operatorname{Vol}}}
\def\ECC{{\operatorname{ECC}}}

\def\F{\mathcal{F}}

\def\K{{\mathbb{K}}}
\def\tm{{\widetilde{m}}}
\def\N{\mathbb{N}}

\def\R{\mathbb{R}}
\def\Rn{{\mathbb{R}^n}}

\def\Z{\mathbb{Z}}

\newcommand{\pom}{{\partial \Omega}}

\title[Overdetermined boundary problems in large domains]{Positive solutions to general semilinear overdetermined boundary problems}

\author[A. Enciso]{Alberto Enciso}
\address{Alberto Enciso\\Instituto de Ciencias Matemáticas, Consejo Superior de Investigaciones Cien\-tíficas, 28049 Madrid, Spain}
\email{aenciso@icmat.es}

\author[P. Hidalgo-Palencia]{Pablo Hidalgo-Palencia}
\address{Pablo Hidalgo-Palencia\\Instituto de Ciencias Matemáticas, Consejo Superior de Investigaciones Científicas, 28049 Madrid, Spain, \& Departamento de Análisis Matemático y Matemática Aplicada, Facultad de Matemáticas, Universidad Complutense de Madrid, 28040 Madrid, Spain}
\email{pablo.hidalgo@icmat.es}

\author[X. Ros-Oton]{Xavier Ros-Oton}
\address{Xavier Ros-Oton \\ ICREA, Pg. Lluís Companys 23, 08010 Barcelona, Spain \& Universitat de Barcelona, Departament de
	Matemàtiques i Informàtica, Gran Via de les Corts Catalanes 585, 08007 Barcelona, Spain \& Centre de
	Recerca Matemàtica, Barcelona, Spain}
\email{xros@icrea.cat}

\date{\today}

\begin{document}
	
\begin{abstract}
We establish the existence of positive solutions to a general class of overdetermined semilinear elliptic boundary problems on suitable bounded open sets $\Omega\subset\Rn$. Specifically, for $n\leq 4$ and under mild technical hypotheses on the coefficients and the nonlinearity, we show that there exist open sets~$\Omega\subset\R^n$ with smooth boundary and of any prescribed volume where the overdetermined problem admits a positive solution. The proof builds on ideas of Alt and Caffarelli on variational problems for functions defined on a bounded region. In our case, we need to consider functions defined on the whole~$\Rn$, so the key challenge is to obtain uniform bounds for the minimizer and for the diameter of its support. Our methods extend to higher dimensions, although in this case the free boundary $\partial\Omega$ could have a singular set of codimension~5. The results are new even in the case of the Poisson equation $-\Delta v =g(x)$ with constant Neumann data. \end{abstract}
	
\maketitle

\tableofcontents


\section{Introduction}

The study of overdetermined boundary problems, that is, problems where
one prescribes both Dirichlet and Neumann data, has grown into a major
field of research. Although this topic had already been considered
in Lord Rayleigh's classic treatise~\cite{Rayleigh}, the breakthrough result was Serrin's symmetry result~\cite{S}. Under mild technical
hypotheses, he showed that if a positive
function satisfies a semilinear equation of the form
\[
-\Delta v = f(v)
\]
in a smooth bounded domain~$\Omega\subset\mathbb R^n$ with the boundary
conditions
\begin{equation*}
	v=0\quad\text{and}\quad |\nabla v|=c\qquad \text{on } \partial\Omega
\end{equation*}
for some constant $c$, then~$\Omega$ is a ball and~$v$ is radial.

Serrin's theorem is best understood as a rigidity result for overdetermined problems: indeed, balls naturally admit radial solutions, so Serrin's contribution in \cite{S} is to show that, even allowing the domain to vary, these are the only possible solutions. Serrin's theorem relies heavily on the underlying Euclidean symmetry of the problem. This naturally raises the question of what remains true when such symmetries are taken away. For instance, one may consider anisotropic or non-homogeneous operators and boundary conditions, or right hand sides $f$ that depend on position. In this case, in contrast with the classical setting, once symmetry is broken, there is in general no obvious candidate solution. Consequently, before addressing rigidity questions, one must first establish the existence of solutions to these overdetermined problems. This is indeed the main goal of the present paper.

There is a large body of literature generalizing Serrin's result. As a sample of the available results, and without attempting to be comprehensive, let us mention the development of alternative approaches to symmetry  that rely on
$P$-functions~\cite{GaLe,Kawohl} or Pohozaev-type integral identities~\cite{Brandolini,MP1,MP2}, quantitative estimates for the stability of these symmetry properties~\cite{ABR}, and extensions to a variety of related contexts: different constant-coefficient elliptic operators of variational form~\cite{Salani};   
degenerate elliptic equations such as the $p$-Laplace
equation~\cite{DPR99}; exterior~\cite{AB,Garo}, unbounded~\cite{Farina} or non-smooth
domains~\cite{FZ, Pr98}; and problems on the hyperbolic space and the hemisphere~\cite{KP98}. The connection between overdetermined problems on unbounded domains and constant mean curvature
surfaces has been studied in~\cite{Traizet,DP,Ros1}. Nontrivial
solutions have been shown to exist in a number of interesting situations, for instance in the case of sign-changing solutions~\cite{Ruiz-arxiv}, for periodic unbounded domains~\cite{RSW1,FMW2,FMW3}, and for problems that are either partially overdetermined or degenerate~\cite{Alessa,Fragala1,Fragala2,Farina2}, and have led to recent progress on various questions in fluid mechanics~\cite{R,ARMA,cpt} and spectral geometry~\cite{Fall-Minlend-Weth-main,DMJ}.

Even though these questions arise naturally in applied contexts~\cite{Toland,ARMA,cpt,Sirakov,Wahlen}, the study of overdetermined problems for position-dependent equations has attracted comparatively little attention. In two remarkable papers, Pacard and Sicbaldi~\cite{PS:AIF}
and Delay and Sicbaldi~\cite{DS:DCDS} proved the existence of smooth extremal
domains with small volume for the first eigenvalue of the
Laplacian in any compact Riemannian manifold. These results were extended in~\cite{Adv19,APDE21} to the case of fairly general overdetermined semilinear problems depending on a free parameter~$\lambda$. Roughly speaking, one can consider a nonlinear eigenvalue problem of the form
\begin{equation} \label{E.semilinearlambda}
	\begin{cases}
		-\div(A(x)\nabla v)  = \lambda f(x, v) \quad &\text{ in } \Omega\subset\R^n, \\
		v = 0 \quad \text{and}\quad |\nabla v|^2 = c\,{q(x)} &\text{ on } \pom.\end{cases}
\end{equation}
Here $q$ is a positive function, $c$ is a free constant as above, $A$ is a uniformly elliptic matrix-valued function, and $\lambda$ is an additional free constant. In all these results (namely \cite{PS:AIF,DS:DCDS,Adv19,APDE21}), the key idea is to study the high-frequency regime $\lambda\gg1$. After a careful asymptotic analysis, and under suitable technical assumptions, one can show that for all large enough~$\lambda$, the overdetermined problem~\eqref{E.semilinearlambda} admits a positive solution on a family of smooth domains constructed as suitably small perturbations of small balls, of radius of order $\lambda^{-1/2}$ and centered at specific points depending on~$f(\cdot,0)$. Thus these are perturbation results, asserting that, in local coordinates adapted to the matrix~$A$ around a judiciously chosen point, one can perturb a positive radial solution to the limiting problem to obtain a solution to~\eqref{E.semilinearlambda}. Similar results can be proven for overdetermined nonlinear eigenvalue problems on compact manifolds.

Our objective in this paper is to remove the additional freedom granted by the free parameter $\lambda$. That is, we aim to study the existence of positive solutions to fairly general classes of semilinear overdetermined boundary value problems that do {\em not}\/ necessarily have a free parameter, and to consider domains that are {\em not}\/ necessarily deformations of small balls. Since these conditions are crucially used in the aforementioned papers to frame the problem in the context of  perturbation theory, our results must therefore be based on an entirely different set of ideas, coming from the theory of elliptic free boundary problems.

\subsection{Main results}

We will present our results for positive solutions to the model semilinear problem
\begin{equation} \label{eq:fbp}
	\begin{cases}
		-\div(A(x)\nabla v) = f(x, v) \quad &\text{ in } \Omega\subset\R^n, \\
		v = 0 \quad \text{and}\quad \nabla v \, A(x)\, \nabla v^{\miT} = c\,{q(x)} &\text{ on } \pom,\end{cases}
\end{equation}
where the (conormal) Neumann boundary condition is understood in the viscosity sense.
Note we are using the natural Neumann boundary condition for this equation,
which reduces to $\partial_\nu v=-\sqrt{c\,q(x)}$ when $A$ is the identity matrix. One can also consider similar problems on compact manifolds, but for concreteness we will relegate the discussion to Section~\ref{S.manifolds} in the main text. In \eqref{eq:fbp}, the functions~$A$, $f$ and~$q$ are fixed, while the constant~$c$ is not specified a priori. We assume that these functions are {\em admissible and $x$-periodic}\/. Roughly speaking, this  means that $f(x,v)$, $A(x)$, and~$q(x)$ are sufficiently smooth and periodic in~$x$, that $f(x,v)$ is nonnegative and grows at most linearly in~$v$, that $q(x)$ is positive, and that the matrix $A(x)$ is uniformly elliptic. Details given in Section~\ref{sec:assumptions}.

 To control the size of the (possibly disconnected) open set~$\Omega$, we will fix a constant $m>0$ and restrict our attention to sets $\Omega\subset\R^n$ such that
\begin{equation}\label{E.volume}
	\Vol_q(\Omega) := \int_\Omega q(x)\, dx=m\,.
\end{equation}
Since $q$ is both upper bounded and bounded away from zero, this quantity is always comparable with the volume of~$\Omega$, and coincides with it in the classical case $q\equiv 1$.

Our main result can the be stated as follows. As we will discuss later in this section, the result is new even for the Poisson equation $-\Delta v=g(x)$ with constant Neumann data.

\begin{theorem}\label{th:main_low}
	Let $n\leq 4$. For any constant $m>0$ and any smooth, admissible, $x$-periodic functions~$q$, $A$ and~$f$, there exists a bounded open set $\Omega\subset\Rn$ with smooth boundary, satisfying the $q$-volume constraint~\eqref{E.volume}, and a constant $c>0$, for which the overdetermined boundary value problem~\eqref{eq:fbp} admits a positive solution $v\in C^\infty(\overline{\Omega})$.
\end{theorem}

To prove Theorem~\ref{th:main_low}, we follow a variational approach, motivated by the seminal work of Alt and Caffarelli \cite{AC}. To this end, let us define a function $F : \R^n \times [0, +\infty) \to \R$  by 
\begin{equation} \label{eq:F_f}
F(x, u) := \int_0^u f(x, t) \, dt\,,
\end{equation}
so that $f(x, u) = F'(x, u)$. Here and in what follows, $F'(x,u)$ denotes the derivative of $F(x,u)$ with respect to the second variable, $u$. In view of the constraint~\eqref{E.volume}, we aim to minimize the energy $\F_0(u) $, 
among the functions in the space
\begin{equation*}
	\K_{\leq m} := \big\{ u \in H^1(\Rn) : \; u \geq 0, \;\; \Vol_q(\{u>0\}) \leq m \big\},
\end{equation*}
with the hope that the minimizer actually belongs to the smaller space
\begin{equation*}
	\K_{=m} := \big\{ u \in \K_{\leq m} : \Vol_q(\{u>0\}) = m \big\}.
\end{equation*}
For $D \subset \Rn$, we are denoting by
\begin{equation}\label{E.cF0}
	\F_0(u, D)
	:=
	\int_D \nabla u (x) \, A(x) \, \nabla u(x)^{\miT} dx - 2 \int_D F(x, u(x)) \, dx
	\end{equation}
the energy functional associated with Equation~\eqref{eq:fbp}, and we use the shorthand notation $\F_0(u):= \F_0(u, \Rn)$.

Setting $\Omega:=\{u>0\}$ and $v:=u|_\Omega$, Theorem~\ref{th:main_low} is proven a consequence of the following result on the existence and regularity of minimizers to the above variational problem. For the definition of the Neumann boundary trace of~$u$ in the viscosity sense, we refer to Definition~\ref{def:viscosity} in the main text.

\begin{theorem} \label{th:main}
	Let $m > 0$ and assume that $F$, $A$ and $q$ are admissible and $x$-periodic in the sense of Definition~\ref{def:admissible}. Then there exists a Lipschitz continuous minimizer $u$ of $\F_0$ in $\K_{\leq m}$, which actually belongs to $\K_{=m}$. The set $\{u > 0\}$ is open and bounded, and the following equation is satisfied
	\begin{equation} \label{eq:bvp}
		\begin{cases}
			 -\div(A \nabla u) = F'(x, u) \quad & \text{in } \{u > 0\}, \\
			\nabla u \,{A} \,\nabla u^{\miT} = c\, {q}(x) \quad & \text{on } \partial \{u > 0\} \text{ in the viscosity sense}
		\end{cases}
	\end{equation}
	for some constant $c>0$. Furthermore, the free boundary can be decomposed as a disjoint union 
	$$
	\partial \{u > 0\} = \Reg(\{u > 0\}) \cup \Sing(\{u > 0\})\,,
	$$
	where: 
	\begin{enumerate}
		\item The regular part $\Reg(\{u > 0\})$ is a $C^{1, \alpha}$ manifold of dimension $n-1$ for some $\alpha > 0$, it is open within $\partial \{u > 0\}$, and on it the Neumann condition holds in the classical sense: $\nabla u \, A\, \nabla u^{\miT} = c\, q$. If $A,q$ and~$f$ are smooth, $\Reg(\{u > 0\})$ is smooth as well.
		\item The singular part $\Sing(\{u > 0\})$ is a closed set within $\partial \{u >0\}$ of Hausdorff dimension at most $n-5$. Moreover, $\Sing(\{u > 0\})$ is empty if $n \leq 4$, and consists at most of countably many points if $n = 5$.
	\end{enumerate}
\end{theorem}

As is well known, determining the optimal dimension of the singular part of a free boundary is a major open problem in the area, and at the moment $n-5$ is the best known bound. As can be seen from our proofs (see e.g. Proposition~\ref{prop:singular}), the dimension of the singular set in Theorem~\ref{th:main} coincides with this optimal dimension; and we have used $n-5$ in the statement simply for clarity in the exposition. We refer to~\cite[Section 1.4]{V} for an up to date discussion of the state of the art on the regularity of free boundaries. 

\subsection{Some comments about the proof and connection with previous work}

To prove Theorem~\ref{th:main}, we use ideas introduced by Alt and Caffarelli in \cite{AC}, and expanded on by many authors. Velichkov's recent book~\cite{V} provides an authoritative presentation of these methods, with a number of extensions and refinements.

The proof of Theorem~\ref{th:main} involves two main differences with the existing literature. These differences are both key for our objective of constructing smooth solutions to the overdetermined problem~\eqref{eq:fbp} in low dimensions, and essential in that they present challenges that cannot be overcome through straightforward modifications of previous results.

The first difference is that in \cite{AC} and most of the works based on it such as~\cite{V}, one considers the minimization problem for functions defined on a fixed regular bounded open set $D \subset \Rn$, and impose  some boundary conditions on $\partial D$. These conditions typically yields uniform Poincaré inequalities on any $\Omega \subset D$, as well as bounds for the solution $u$ in terms of the boundary values on $D$. And of course, every $\Omega \subset D$ is automatically bounded. 

In contrast, our minimization problem is posed on the whole $\Rn$. Proving that both the minimizer~$u$ and the set $\{u>0\}$ are bounded when the ambient space is unbounded, as in our case,  requires rather non-trivial arguments. In fact, a key step in our proof is to obtain  {\em uniform}\/ estimates for $\|u\|_{L^\infty(\Rn)}$ and for the diameter of the set~$\{u>0\}$ which one can effectively control in terms of bounds for the functions~$A$, $f$ and~$q$.

The second difference is related to the kind of functions~$f$ that we can consider\footnote{There are also differences in the kind of functions~$A$ and~$q$ one can consider, but we shall focus our discussion of the role of~$f$ because the effect of~$f$ on the minimization problem is usually much more dramatic than that of~$A$ and~$q$.} in the right hand side of~\eqref{eq:fbp}. Indeed, there are several previous papers using a variational approach to study problems for certain kind of functions~$f$. Most of them only deal with some concrete part of the minimization problem, like  the analysis of the improvement of flatness performed in~\cite{DS}. 

To the best of our knowledge, there are only two cases in which the minimization problem has been carried out completely. First, in~\cite{GS}, Gustafsson and Shahgholian consider the minimization problem on~$\R^n$ for the Poisson equation $-\Delta v= h(x)$. A key difference here is not only that the equation is linear, but also and most fundamentally the sign of~$h$. Indeed, these authors assume that $h\leq -c_0<0$ outside a fixed ball, which basically forces minimizers to concentrate around the origin, which is crucially used to show that minimizers have compact support. In fact, the key steps of their proof rely strongly on this choice of sign. In particular, they extensively use comparisons with radially decreasing rearragements, which decrease their energy because they concentrate the mass around the origin. These ideas cannot be adapted to handle any class of functions~$h(x)\geq0$, as we do in this paper. As is well known, overdetermined problems with a nonnegative right hand side are classical in the literature, starting with Serrin's breakthrough paper~\cite{S}, but we shall see that handling this sign condition will require much work.

The second case concerns the torsion problem, $-\Delta u =1$, in which the right hand side has the same sign as in our problem~\eqref{eq:fbp}. In this case, Lederman~\cite{L} carried out the full minimization program on the unbounded ambient space $\Rn \setminus H$, where $H$ is a fixed ``hole''. One then looks for minimizers among function whose Dirichlet trace on $\partial H$ is a fixed positive function. This boundary condition forces solutions to concentrate around $\partial H$, which is the key observation used to show the key property that that the support of minimizers is uniformly bounded. As the right hand side is constant, this property then is used to equivalently formulate the minimization problem on~$B\backslash H$ for some fixed ball~$B$, so one can effectively avoid working on an unbounded space. This is certainly not the case when one considers more general functions on the right hand side, as we do in this paper. Aguilera, Alt and Caffarelli~\cite{AAC} addressed a problem closely related to Lederman's using a penalized energy functional.

The admissibility and periodicity conditions we impose on the functions $f$, $A$ and~$q$ are easy to interpret in the context of the minimization program. First, we require that these functions are smooth and positive (which, in the case of the matrix-valued function~$A$, corresponds to a uniform ellipticity condition). Then we assume that $F$~grows at most quadratically in~$u$ (which is optimal, as shown in Appendix~\ref{S.appendix}), and that it is sufficiently positive somewhere so that nontrivial minimizers exist. Lastly, we assume that these functions are periodic in~$x$. This does not immediately imply that minimizers have uniformly bounded supports because the background space is still all of~$\Rn$ (it only does when the whole minimization program is carried out on the compact quotient space $(\R/2\pi\Z)^n$ instead, as we do when we consider overdetermined problem on compact manifolds in Section~\ref{S.manifolds}, and there the analysis is much simpler). But periodicity is crucially used to establish the uniform estimates for~$u$ and its support and, in fact, in Appendix~\ref{S.appendix} we show that without periodicity in~$x$, minimizers with uniformly bounded supports do not necessarily exist.

The paper is organized as follows. In Section~\ref{sec:assumptions} we state the admissibility and periodicity assumptions in our main theorems, introduce some notation, and prove some auxiliary estimates that we will use throughout the paper. Then we move on to prove some basic properties of minimizers in Section~\ref{sec:basic}. In Section~\ref{sec:first_variation} we consider a penalized functional, which is subsequently used to show that minimizers are Lipschitz continuous in Section~\ref{sec:lip}. In Section~\ref{sec:compact_support} we obtain bounds for the diameter of the support of minimizers. After analyzing the regularity of the free boundary in Section~\ref{sec:boundary}, the proof of Theorem~\ref{th:main} is presented in Section~\ref{sec:proof}. Section~\ref{S.manifolds} is devoted to an analogue of our main results for overdetermined problems on compact manifolds. To conclude, in Appendix~\ref{S.appendix} we show that, without the periodicity assumption, admissible overdetermined problems do not generally have minimizers with uniformly bounded supports.



\section{Assumptions and auxiliary results} \label{sec:assumptions}

Our goal in this section is to discuss the assumptions of the functions $f$, $A$ and~$q$ that we use, and to prove some basic estimates that will be used throughout the paper.

\subsection{The admissibility and periodicity assumptions}

Let us now state carefully the assumptions we need in Theorem~\ref{th:main}. As we anticipated in the Introduction, if these are relaxed, several of our key estimates can fail, as we show in Appendix~\ref{S.appendix}.

Before stating the assumptions, let us introduce some notation that we will use throughout the paper. We denote by $C$ constants whose value may be different in every appearance. They may depend on the dimension $n$, $m$ and the parameters appearing in Definition~\ref{def:admissible} below. We sometimes make dependencies explicit using brackets, e.g. $C(n, M')$. In turn, $X \lesssim Y$ is a equivalent notation to $X \leq C Y$ for $C$ depending $n$, $m$ and the parameters appearing in Definition~\ref{def:admissible}, and $X \lesssim_{M'} Y$ means $X \leq C(M') Y$. Similarly, $X \approx Y$ means $X \lesssim Y$ and $Y \lesssim X$. 

Euclidean balls of radius $r > 0$ are denoted $B_r(x)$ if centered at $x \in \Rn$, or simply $B_r$ if centered at the origin. We also denote $B^m$ the ball with volume $m$. The $n-1$ dimensional Hausdorff measure is denoted $\mathcal{H}^{n-1}$. Averages will be denoted by the sign $\fint$. Also, $\mathbf{1}_E$ denotes the indicator function of the set $E \subset \Rn$ and $\partial_\nu$ denotes the outer normal derivative. We will also denote by $\lambda_1$ the first Dirichlet eigenvalue of different sets.

We will make frequent use of matrix notation when multiplying vectors (which will be thought as row vectors) and matrices, mixing it with the notation $x\cdot y$ for the inner product in $\Rn$. In Sections~\ref{sec:first_variation} and \ref{sec:boundary}, it will also appear $\nabla A \cdot x$, where $A$ is a matrix: this is the matrix whose entries are $\nabla a_{ij} \cdot x$, which amounts to taking the gradient componentwise. With some abuse of notation, in this definition we denote by $\R^{n^2}$ the space of $n\times n$ matrices.

We are ready to state the admissibility and periodicity assumptions in Theorem~\ref{th:main}. Here we are introducing several constants that will appear later on in various steps of the proof.

\begin{definition}[Admissibility and periodicity conditions] \label{def:admissible}
Consider the functions appearing in~\eqref{eq:fbp}-\eqref{E.volume}-\eqref{E.cF0}, namely $f, F:\Rn\times\R\to\R$, $A:\Rn\to\R^{n^2}$ and $q:\Rn\to\R$.
\begin{enumerate}
	\item These functions are {\em admissible}\/ if:\smallskip
	\begin{enumerate}
	\item The functions~$f(x, u)$ and $F(x,u)$, where $F$ is defined as in \eqref{eq:F_f}, satisfy, for (almost) every $x\in\Rn$ and $u\geq0$,
	\begin{align}
		\tag{HF1} \label{hip:F_regular}
		\bullet\quad &F(x, u) \text{ is $C^{1,1}$ in both variables}, \\
		\tag{HF2} \label{hip:zero}
		\bullet\quad &F(x, 0) = 0\,,\\
		\bullet\quad &\tag{HF3} \label{hip:negative}
		F'(x, u) = f(x, u) \geq 0,\\
		\bullet\quad &\tag{HF4} \label{hip:F_quadratic}
		F(x, u) \leq N + bu^2,\\
		\tag{HF5} \label{hip:F'_F''}
		\bullet\quad &\begin{cases}
			F'(x, u) = f(x, u) \leq N' + M' u, 	\\ 
			F''(x, u) = f'(x, u) \leq M_2,
		\end{cases}\\[2mm]
		\bullet\quad &\tag{HF6} \label{hip:sol_not_zero}
		\begin{array}{c}
			\text{There exists } u_0 \in \K_{\leq m} \text{, i.e. } 0 \leq u_0 \in H^1(\Rn) \text{ and } 
			\\ \Vol_q(\{u>0\}) \leq m \text{ (see \eqref{E.volume}), such that } \F_0(u_0) < 0.
		\end{array}
	\end{align}
Here $0 \leq b < \lambda \, \frac{\lambda_1(B^m)}{2}$  and  $N, N', M', M_2 \geq 0$ are constants. We note that \eqref{hip:zero} is automatically satisfied because of \eqref{eq:F_f}, but it will useful to keep it in mind. \medskip

\item The function~$A(x)$ satisfies, for every $x,\xi\in\Rn$,
\begin{align}
	\bullet\quad &\tag{HA1} \label{hip:A_regular}
		A \in C^{1, 1}(\Rn),\\
			\bullet\quad &\tag{HA2} \label{hip:A_symmetric}
		\text{$A(x)$ is a symmetric matrix,}\\
			\bullet\quad &\tag{HA3} \label{hip:A_elliptic_bounded}
		\lambda \abs{\xi}^2\leq \xi A(x) \xi^{\miT} \leq \lambda^{-1} \abs{\xi}^2
\end{align}
for some constant $\lambda\in(0,1)$.\medskip

\item The function $q(x)$ satisfies, for every $x\in\Rn$,
\begin{align}
		\bullet\quad &\tag{Hq1} \label{hip:q_regular}
		q \in C^{1, 1}(\Rn),\\
			\bullet\quad &\tag{Hq2} \label{hip:q_positive}
		\underline q  \leq q(x)\leq   \overline q
\end{align}
for some positive constants $\underline q, \overline q$.\bigskip

	\end{enumerate}
	\item These functions are {\em $x$-periodic}\/ if there is some $T>0$ such that
	\begin{equation}\tag{HPer} \label{hip:periodic}
		F(x, u)=F(x+Te, u),\quad  A(x+Te)=A(x), \quad q(x+Te)=q(x)
	\end{equation}
	hold for every $e \in \Z^n$.
\end{enumerate}
	With some abuse, we will also say that the functions $f(x,v)$, $A(x)$, $q(x)$ appearing in~\eqref{eq:fbp} are {\em admissible and $x$-periodic}\/ when $F(x,v):=\int_0^vf(x,t)\,dt$, $A(x)$ and $q(x)$ (which are the functions appearing in the variational formulation) are admissible and periodic in the above sense.
\end{definition}

\begin{remark}
	The hypothesis \eqref{hip:sol_not_zero} is obviously satisfied for many reasonable functions $F(x, u)$. In fact, it is satisfied as soon as $F$ is positive somewhere in the sense that $F(x, u) \geq c_0 u$ for all $x$ a small ball $B_\varepsilon$, with $c_0, \varepsilon > 0$. Indeed, we can take $\varphi$ the first eigenfunction of the Laplacian in the ball $B_\varepsilon$, and compute, for $\tau >0$,
	\begin{equation*}
		\F_0(\tau \varphi)
		\leq 
		\lambda^{-1} \lambda_1(B_\varepsilon) \, \tau^2 \int_{B_\varepsilon} \varphi^2 \, dx 
		- 2 c_0 \tau \int_{B_\varepsilon} \varphi \, dx
		 .
	\end{equation*}
	Therefore, by the positivity of $\varphi$, $\F_0(\tau \varphi)<0$ provided that $\tau $ is sufficiently small. 
\end{remark}

\begin{remark}
	In Assumption~\eqref{hip:periodic}, one could have considered periodicity conditions defined by any other lattice on~$\Rn$. All the arguments in the paper work in this case too, but we have stated the results only in terms of square lattices for concreteness.
\end{remark}

Let us discuss these assumptions. 
First, in order not to worry about technicalities which are not the focus of this work, we assume \eqref{hip:F_regular}, \eqref{hip:A_regular} and \eqref{hip:q_regular}, where here and everywhere, $\nabla F$ and $F'$ respectively denote the derivatives of $F$ with respect to the first and second variables. The structural assumptions \eqref{hip:A_symmetric} and \eqref{hip:A_elliptic_bounded} are natural if we want to obtain an elliptic PDE for our solutions, and \eqref{hip:q_positive} is natural if one desires to obtain the overdetermined Neumann boundary condition. 
Then, we also make the natural assumption that~$u$ only contributes to the energy on its support, which is~\eqref{hip:zero}.
{As explained in the Introduction, the sign condition \eqref{hip:negative} is a key novelty, and will require much work.} The 
Assumptions \eqref{hip:F_quadratic}--\eqref{hip:sol_not_zero} and~\eqref{hip:periodic} will arise naturally in certain estimates, so we will comment on them later.

\subsection{An approximate mean value formula for inhomogeneous equations}

Let us now establish analogues of the  Harnack inequality and the mean value formula for the equation $\cL v=g$. Here and in what follows, we will apply the notation
\begin{equation}\label{E.L}
	\cL v:=-\div(A(x) \nabla v)
\end{equation}
for the elliptic operator in divergence form associated with the matrix-valued function~$A$.

\begin{lemma} \label{lem:harnack}
	Assume that $A$ is admissible in the sense of Definition~\ref{def:admissible}. Let $r \in (0, 1)$ and suppose that $u \in H^1(B_r)$ satisfies $u \geq 0$ on $\partial B_r$. For any $M \geq 0$, the following hold: 
	\begin{enumerate}
		\item If $\cL u \leq M$ in $B_r$, then 
		\begin{equation} \label{eq:harnack}
			u(x) 
			\lesssim_{n} \;
			\fint_{\partial B_r} u \, d\mathcal{H}^{n-1} + Mr^2, 
			\qquad 
			x \in B_{r/2}.
		\end{equation}
		\item If $\abs{\cL u} \leq M$ in $B_r$, then 
		\begin{equation} \label{eq:harnack_gradient}
			\abs{\nabla u(0)} 
			\lesssim_{n} \;
			\frac{1}{r} \fint_{\partial B_r} u \, d\mathcal{H}^{n-1} + Mr.
		\end{equation}
	\end{enumerate}
\end{lemma}
\begin{proof}
	Let us prove this in an elementary way for the Laplacian, and then sketch the differences to obtain the result for general $\cL$. 
	
	\subsubsection*{Proof for the Laplacian, $A \equiv I$.}
	To show \eqref{eq:harnack}, note that $v(x) := u(x) + M \abs{x}^2 / (2n)$ is subharmonic in $B_r$, i.e. $-\Delta v \leq 0$ in $B_r$. Therefore, by the maximum principle, $v \leq w$ in $B_r$, where $w$ is the harmonic extension of $v$ to $B_r$, that is, $w$ is the function satisfying $\Delta w = 0$ in $B_r$, and coincides with $v$ at $\partial B_r$. By the  Poisson formula,
	\begin{equation*}
		w(x)
		=
		r^{n-2} \fint_{\partial B_r} \frac{r^2 - \abs{x}^2}{\abs{y-x}^n} \left(u(y) + M \frac{\abs{y}^2}{2n} \right) \, d \mathcal{H}^{n-1}(y),
		\qquad 
		x \in B_r.
	\end{equation*}
	Therefore, if we take $x \in B_{r/2}$, a simple triangle inequality (recalling also $u \geq 0$ on $\partial B_r$) shows 
	\begin{equation} \label{eq:mean_value_2}
		v(x) 
		\leq 
		w(x)
		\leq 
		r^{n-2} \frac{r^2}{(r - r/2)^n} \left( \fint_{\partial B_r} u \, d\mathcal{H}^{n-1} + M \frac{r^2}{2n} \right),
	\end{equation}
	from which \eqref{eq:harnack} readily follows. 
	
	To obtain \eqref{eq:harnack_gradient} as a consequence of \eqref{eq:harnack}, we can do the following. 
	First, start by defining $\widehat{v}(x) := u(x) + \abs{\nabla u(0)} x_1$ (where $x = (x_1, x_2, \ldots, x_n)$ are just the coordinates of points in $\Rn$), and note that still $\abs{-\Delta \widehat{v}} \leq M$ in $B_r$. Applying \eqref{eq:harnack} to $\widehat{v} + \abs{\nabla u(0)} r$, which is non-negative on $\partial B_r$ because so is $u$, we obtain 
	\begin{equation*}
		\widehat{v}(x) + \abs{\nabla u(0)} r
		\lesssim 
		\fint_{\partial B_r} \big(\widehat{v} + \abs{\nabla u(0)} r\big) \, d\mathcal{H}^{n-1} + M r^2,
		\qquad 
		x \in B_{r/2}.
	\end{equation*} 
	Cancelling the terms $\abs{\nabla u(0)} r$ in both sides, and then recalling the definition of $\widehat{v}$, we get
	\begin{equation*}
		u(x) + \abs{\nabla u(0)} x_1 
		\lesssim 
		\fint_{\partial B_r} u \, d\mathcal{H}^{n-1} + M r^2,
		\qquad 
		x \in B_{r/2},
	\end{equation*} 
	because of the cancellation $\fint_{\partial B_r} x_1 \, d\mathcal{H}^{n-1}(x) = 0$. Then, note that $u \geq -Mr^2/(2n)$ in $B_r$: if we define $\widehat{w}(x) := u(x) - M \abs{x}^2 / (2n)$, it is superharmonic (i.e. $-\Delta \widehat{w} \geq 0$) in $B_r$, and $\widehat{w} \geq -M r^2/(2n)$ on $\partial B_r$ (because $u \geq 0$ on $\partial B_r$); whence we finish the minimum principle for superharmonic functions. Therefore, by taking any $x \in \partial {B_{r/4}}$ in the previous display, we obtain \eqref{eq:harnack_gradient}.
	
	\subsubsection*{Proof for general $\cL$.} 
	To generalize the above proof to general $\cL$, the first modification is to use $v(x) := u(x) + M\phi(x)$ (and the same for $\widehat{w}$ just changing signs) and $\widehat{v}(x) := u(x) + \abs{\nabla u(0)} \widehat{\phi}$, where $\phi$ and $\widehat{\phi}$ solve 
	\begin{equation*}
		\begin{cases}
		 \cL\phi = -1 \quad & \text{in } B_r \\
		 \phi = \dfrac{\abs{x}^2}{2n} = \dfrac{r^2}{2n} \quad & \text{on } \partial B_r,
		\end{cases}
		\qquad \text{and} \qquad 
		\begin{cases}
			\cL\widehat{\phi} = 0 \quad & \text{in } B_r \\
			\widehat{\phi} = x_1 \quad & \text{on } \partial B_r.
		\end{cases}
	\end{equation*}
	These are the natural analogues to $v$ and $\widehat{v}$ in this more general context. By rescaling $\phi$ to $B_1$ by $\phi_r(x) := \phi(rx)$, which satisfies $\cL_r \phi_r = r^2$ in $B_1$ (where $A_r(x) := A(rx)$, whence it has the same ellipticity constant as $A$) and $\phi_r = r^2/(2n)$ on $\partial B_1$, the a priori bounds for elliptic equations (e.g., \cite[Theorem 3.7]{GT}, where we use \eqref{hip:A_regular} and $r \leq 1$ to have uniform bounds on the moduli of continuity of the coefficients of $\cL_r$) 
	yield $\norm{\phi}_\infty \lesssim r^2$. This is used to deduce \eqref{eq:harnack} from the analogue estimate in our setting to \eqref{eq:mean_value_2}, and to bound $\widehat{w}$ from below.

	The second main point is to use a Poisson formula for $w$ (which in this case, solves $\cL w=0$ in $B_r$). It turns out that this can be done for fairly general operators, and for that, we may use the $\cL $-harmonic measure. Indeed, given $g \in C(\partial B_r)$, let us denote by $u_g$ the solution of the problem $\cL u_g = 0$ in $B_r$, $u_g =g$ on $\partial B_r$. Then, given $x \in B_r$, the map $g \mapsto u_g(x)$ is a linear and bounded (by the maximum principle, and also positive) operator from $C(\partial B_r)$ to $\R$. Therefore, by the Riesz Representation Theorem, it can be represented by a (positive) Radon measure $\omega_\cL ^x$, the so-called {\em $\cL $-harmonic measure}\/, or elliptic measure associated to $\cL $: 
	\begin{equation*}
		u_g(x) = \int_{\partial B_r} g(y) \, d\omega_\cL ^x(y), 
		\qquad 
		g \in C(\partial B_r).
	\end{equation*}

	It is known that if $A(x) \in C^\alpha$, then $\omega_\cL ^x$ is absolutely continuous with respect to $\mathcal{H}^{n-1}|_{\partial B_r}$ (at least in regular sets like the ball, see e.g.~\cite{MM})\footnote{
	The Hölder continuity of the coefficients is necessary, because if we only let $A(x) \in C(\overline{B_r})$, then there are counterexamples, as those in \cite{CFK, MM}.
	}. Therefore, we can represent $w$ in the proof above as
	\begin{equation*}
		w(x)
		=
		\int_{\partial B_r} w(y) \, d\omega_\cL ^x(y)
		=
		\int_{\partial B_r} w(y) \, \frac{d \omega_\cL ^x}{d\mathcal{H}^{n-1}} (y) \, d\mathcal{H}^{n-1}(y).
	\end{equation*}
	Given this representation, to be able to draw the same conclusions as in the proof for the Laplacian above, we just need to show that 
	\begin{equation} \label{eq:poisson_kernel_constant}
		\frac{d \omega_\cL ^x}{d\mathcal{H}^{n-1}} (y)
		\approx 
		1, 
		\qquad x \in B_{r/2}, y \in \partial B_r.
	\end{equation}
	
	To prove this, fix $x \in B_{r/2}$, $y \in \partial B_r$, and $\rho > 0$ small, see Figure~\ref{fig:cfms} for a sketch. Then, the estimate from \cite[Lemma 2.2]{CFMS} yields 
	\begin{equation*}
		\omega_\cL ^x(B_\rho(y) \cap \partial B_r)
		\approx 
		\rho^{n-2} \, G_\cL ((1-\rho) y, x), 
	\end{equation*}
	where $G_\cL (\cdot, \cdot)$ is the Green's function in $B_r$ for the operator $\cL $, i.e., the function satisfying $\cL  G_\cL (\cdot, x) = \delta_x$ in $B_r$ and $G_\cL  (\cdot, x) = 0$ on $\partial B_r$. Therefore, since near $\partial B_r$, $G_\cL $ is $\cL $-harmonic, and it vanishes identically at the boundary, Hopf's lemma (for a reference under the only assumption that $A(x) \in \cL ^\infty$, see \cite{Safonov}) yields some linear growth, namely $G_\cL ((1-\rho)y, x) \approx \rho$ for $\rho > 0$ small enough. Putting all the above together, we obtain 
	\begin{equation*}
		\frac{d \omega_\cL ^x}{d\mathcal{H}^{n-1}} (y)
		=
		\lim_{\rho \to 0} \,\frac{\omega_\cL ^x(B_\rho(y) \cap B_r)}{\mathcal{H}^{n-1}(B_\rho(y) \cap B_r)}
		\approx 
		\lim_{\rho \to 0} \,\frac{\rho^{n-2} G_\cL ((1-\rho) y, x)}{\rho^{n-1}}
		\approx 
		1,
	\end{equation*}
	which finally shows \eqref{eq:poisson_kernel_constant}.
	
	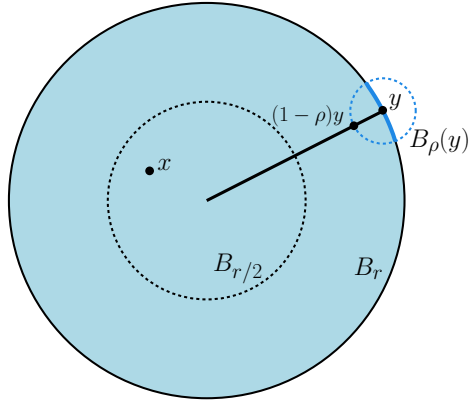
\begin{figure}[t!]
		\centering 
		\resizebox{0.4\textwidth}{!}{
			\begin{tikzpicture}[scale=1]
				\filldraw[fat, fill=lightblue] 
				(6.773, 6.773) circle[radius=6.773];
				\draw[fat, dashed] 
				(6.773, 6.773) circle[radius=3.387];
				\draw[heavier] (6.773, 6.773) -- (12.81, 9.844) -- (12.81, 9.844);
				\node[circle, fill, inner sep=3pt] at (4.818, 7.791) {};
				\node[anchor=center, font=\Huge] at (5.33, 8.002) {$x
					$};
				\node[anchor=base, font=\Huge] at (13.238, 10.05) {$y$};
				\node[anchor=center, font=\huge] at (10.232, 9.671) {$(1-\rho) y $};
				\node[anchor=center, font=\Huge] at (14.752, 8.793) {$B_\rho(y)$};
				\node[anchor=center, font=\Huge] at (12.314, 4.44) {$B_r$};
				\node[anchor=center, font=\Huge] at (7.84, 4.455) {$B_{r/2}$};
				\draw[myblue, fat, dashed] 
				(12.81, 9.844) circle[radius=1.12];
				\draw[myblue, ultrafat] (12.234, 10.805) .. controls (12.295, 10.697) and (12.327, 10.65) .. (12.364, 10.596) .. controls (12.401, 10.542) and (12.442, 10.479) .. (12.479, 10.421) .. controls (12.516, 10.364) and (12.549, 10.311) .. (12.579, 10.26) .. controls (12.61, 10.209) and (12.638, 10.161) .. (12.667, 10.109) .. controls (12.697, 10.057) and (12.727, 10.002) .. (12.761, 9.937) .. controls (12.795, 9.872) and (12.833, 9.798) .. (12.871, 9.718) .. controls (12.91, 9.637) and (12.949, 9.552) .. (12.987, 9.465) .. controls (13.025, 9.378) and (13.061, 9.289) .. (13.098, 9.192) .. controls (13.135, 9.095) and (13.174, 8.99) .. (13.23, 8.806);
				\node[circle, fill, inner sep=3pt] at (12.799, 9.867) {};
				\node[circle, fill, inner sep=3pt] at (11.812, 9.337) {};
			\end{tikzpicture}
		}
		\caption{To compute the Poisson kernel at a point $x \in B_{r/2}$, one may use the Green's function at points close to the boundary.}
		\label{fig:cfms}
	\end{figure}
	
	With these properties in hand, it is straightforward to follow the proof for the Laplacian, extending the result to a general $\cL $ as in~\eqref{E.L} with $A(x) \in C^\alpha$. The lemma then follows.
\end{proof}



\section{Some basic properties of minimizers} \label{sec:basic}

Over the next sections we will establish the key intermediate results that we will need to prove Theorem~\ref{th:main}. For this purpose, let us fix some $m > 0$.

As is well known, instead of minimizing $\F_0$ over $\K_{=m}$, it is more convenient to minimize it over the larger family $\K_{\leq m}$, which is better behaved from the perspective of the calculus of variations (indeed, it is a larger family which allows for easier constructions of competitors, and still one is often able to show that minimizers belong to $\K_{=m}$; see Lemma~\ref{lem:saturation}). However, an important first observation is that if one tries to apply the direct method in the calculus of variations (as e.g.\ in \cite[Proposition 11.1]{V}) to find a minimizer of $\F_0$ in $\K_{\leq m}$, one encounters a key difficulty: even with he admissibility and periodicity assumptions in Definition~\ref{def:admissible}, it may be the case that the mass of the minimizing sequences escapes to infinity. At this stage, we still have no information about the shapes that minimize $\F_0$, so it would be difficult to discard that the mass escapes to infinity. 

We will  control this difficulty a posteriori in Section~\ref{sec:proof}, once we have gathered enough information about possible minimizers. So for the moment, and until Section~\ref{sec:proof}, we will work {a priori}, under the assumption of the existence of a minimizer, and infer as much information about it as possible. Therefore, until Section~\ref{sec:proof} we will make the following

\begin{assumption}
	In Sections~\ref{sec:basic} to~\ref{sec:boundary}, we assume that:
	\begin{equation}
		\tag{H*} \label{hip:setting}
		\begin{cases}
			m>0,\\
			\text{$F, A$ and $q$ are admissible  in the sense of Definition~\ref{def:admissible}},\\
			\text{$u$ is a minimizer of $\F_0$ in $\K_{\leq m}$}.
		\end{cases}
	\end{equation}
\end{assumption}

We note that here we do not assume the periodicity condition in \eqref{hip:periodic}. We will use this in Section~\ref{sec:proof}, but for now it is not necessary.



A first property that a minimizer~$u$ must satisfy is the following. Recall that the operator~$\cL $ was introduced in~\eqref{E.L}. Note that in this lemma we show that if
$$\Omega_u:=\{u>0\}$$
is an open set (which we have not proved yet), then the minimizer~$u$ does satisfy the PDE we want:

\begin{lemma}[Euler--Lagrange (I)] \label{lem:EL}
	Assume \eqref{hip:setting}. Then
	$\cL  u \leq F'(x, u)$ in $\Rn$, in the weak sense. Furthermore, if $B_r(x_0) \subset \{u>0\}$ for certain $x_0 \in \Rn$ and $r > 0$, then $\cL  u = F'(x, u)$ in $B_r(x_0)$.  
\end{lemma}
\begin{proof} 
	Testing the minimizing property of $u$ against $v_\varepsilon := (u - \varepsilon \phi)_+$ for $\phi \in C^\infty_c(\Rn)$, we obtain (recalling that $A$ is symmetric, elliptic, and uniformly bounded, see \eqref{hip:A_elliptic_bounded} and \eqref{hip:A_symmetric})
	\begin{align*}
		0
		\leq 
		\F_0(v_\varepsilon) - \F_0(u)
		& \leq 
		\int_{\{u>0\}} \nabla (u - \varepsilon \phi) \, A \, \nabla (u - \varepsilon \phi)^{\miT} \, dx
		- \int_{\{u>0\}} \nabla u \, A \, \nabla u^{\miT} \, dx
		\\ & \qquad -
		2 \int_\Rn \big(F(x, v_\varepsilon) - F(x, u)\big) \, dx
		\\ & =
		-2\varepsilon \int_\Rn \nabla u \, A \, \nabla \phi^{\miT} \, dx
		+ 2 \int_\Rn F'(x, u) (u-v_\varepsilon) \, dx
		+ O(\varepsilon^2)
		\\ & \leq 
		-2\varepsilon \int_\Rn \nabla u \, A \, \nabla \phi^{\miT} \, dx
		+2 \int_\Rn F'(x, u) \, \varepsilon \phi \, dx
		+ O(\varepsilon^2).
	\end{align*}
	The third step is simply a Taylor expansion using \eqref{hip:F'_F''}, and in the last step we used $u - v_\varepsilon \leq \varepsilon \phi$ along with \eqref{hip:negative}. Dividing by $\varepsilon$ and letting $\varepsilon \to 0$ gives the first part of the result. 
	
	Furthermore, if $B_r(x_0) \subset \{u>0\}$ and $0 \leq \phi \in C^\infty_c(B_r(x_0))$, we have that $u - \varepsilon \phi\in \K_{\leq m}$ for all~$\varepsilon\in\R$ close enough to~0, so the above argument yields
	\begin{align*}
		0
		\leq 
		\F_0(u - \varepsilon \phi) - \F_0(u)
		& =
		-2\varepsilon \int_\Rn \nabla u \, A \, \nabla \phi^{\miT} \, dx
		+ 2 \varepsilon \int_\Rn F'(x, u) \,  \phi \, dx
		+ O(\varepsilon^2).
		\end{align*}
		Therefore, $\cL u=F'(x,u)$ in $B_r(x_0)$, as claimed.
\end{proof}

Our next goal is to show that $u$ is uniformly bounded (that is, $\|u\|_{L^\infty(\Rn)}<C$ for some~$C$ depending only on the admissibility constants, see Definition~\ref{def:admissible}).

\begin{proposition}[Boundedness of minimizers] \label{prop:bounded} 
	Assume \eqref{hip:setting}. Then $u$ is uniformly bounded.
\end{proposition}
\begin{proof}
	
	Let us distinguish cases depending on the profile of the fundamental solution in $\Rn$. 
	
	\subsubsection*{Case 1: $n>2$.} This case is more delicate, and we will follow a series of steps.
	
	\subsubsection*{Step 1: Newtonian potentials.}
	First, write $f(x) := F'(x, u(x))$ (one should {\em not}\/ mistake this function for $f(x,u)$ in~\eqref{eq:fbp}). By the linear bounds for $F'$ in \eqref{hip:zero}, \eqref{hip:negative}, \eqref{hip:F'_F''}, and the estimate on the size of $\Omega_u$ in \eqref{eq:mtilde},
	one easily estimates
	\begin{equation} \label{eq:f_u}
		\norm{f}_{L^r(\Rn)} \lesssim_r \norm{u}_{L^r(\Rn)} + 1, 
		\qquad 
		\text{for any } 1 \leq r \leq +\infty.
	\end{equation}
	Concretely, since $u \in H^1(\Rn)$, we have $u \in L^2(\Rn)$, so $f \in L^2(\Rn)$, too.
	
	On the other hand, we denote 
	\[w(x) := \int_\Rn \Gamma_\cL (x, y) \, f(y) \, dy
	\quad \text{and} \quad w_{-\Delta} (x) := \int_\Rn \Gamma_{-\Delta}(x, y) \, f(y) \, dy,
	\qquad x \in \Rn,\] where $\Gamma_\cL $ and $\Gamma_{-\Delta}$ are, respectively, the fundamental solutions of the operators $\cL $ and $-\Delta$ in $\Rn$ (see \cite[Chapter 2.C]{Folland} or \cite[Chapter 4]{GT} for $-\Delta$\footnote{We warn the reader that both references have a different sign convention to ours, for they work out the fundamental solution for $\Delta$ instead of $-\Delta$.}, and \cite[Section 6]{LSW} or \cite{HK} for general $\cL $ satisfying \eqref{hip:A_elliptic_bounded}). The fact that $w_{-\Delta}$ is well-defined (at least for a.e. $x \in \Rn$), by an absolute convergent integral, is shown in \cite[Theorem V.1.1]{Stein}\footnote{Note that $w_{-\Delta} = I_2f$ adopting the classical notation of Riesz potentials, as used in \cite{Stein}.}. In turn, $w$ is also well-defined (via an absolutely convergent integral) because $\Gamma_\cL  \approx \Gamma_{-\Delta}$ (see \cite[(7.9)]{LSW}). Actually, this last fact also implies, noting that $\Gamma_{\cL}, \Gamma_{-\Delta} > 0$ and $f \geq 0$ (by \eqref{hip:negative}),
	\begin{equation} \label{eq:w_comparable}
		w \approx w_{-\Delta}, \qquad \text{ in } \Rn.
	\end{equation}

	\subsubsection*{Step 2: A pointwise estimate.} We claim that it holds 
	\begin{equation} \label{eq:ptwise_uw}
		u(z) 
		\lesssim 
		w(z) + \norm{u}_{H^1(\Rn)} + 1, 
		\qquad 
		z \in \Rn.
	\end{equation}
	
	Indeed, since $\cL w = f$ in $B_1(z)$ (see again \cite{Folland, GT, LSW, HK}), it holds $\cL (u-w) \leq 0$ in $B_1(z)$ by Lemma~\ref{lem:EL}. Then, it also holds $\cL (u-w)^+ \leq 0$ in $B_1(z)$ (the pointwise maximum of $\cL$-subharmonic functions is again $\cL$-subharmonic, see e.g. \cite[Lemma 7.2]{HKM}). Therefore, applying Lemma~\ref{lem:harnack} for radii between $1/2$ and $1$, and then averaging, we get
	\begin{equation*}
		(u-w)^+(z)
		\!
		\lesssim 
		\fint_{B_1(z) \setminus B_{1/2}(z)} \!\!\!\!\!\! (u-w)^+ \, dx
		\leq
		\left( \fint_{B_1(z) \setminus B_{1/2}(z)} \!\!\!\!\!\!\!\! \abs{u-w}^{2^*} \!\! dx \right)^{1/2^*}
		\!\!\!\!\!
		\lesssim 
		\norm{u}_{L^{2^*}(\Rn)} + \norm{w}_{L^{2^*}(B_1(z))},
	\end{equation*}
	which jointly with \eqref{eq:w_comparable} and the boundedness of Riesz potentials \cite[Theorem V.1.1]{Stein} yield
	\begin{multline*}
		u(z)
		\leq 
		w(z) + (u-w)^+(z)
		\lesssim 
		w(z) + \norm{u}_{L^{2^*}(\Rn)} + \norm{w}_{L^{2^*}(B_1(z))}
		\\ \approx  
		w(z) + \norm{u}_{L^{2^*}(\Rn)} + \norm{w_{-\Delta}}_{L^{2^*}(B_1(z))}
		\lesssim 
		w(z) + \norm{u}_{L^{2^*}(\Rn)} + \norm{f}_{L^{2_*}(B_1(z))},
	\end{multline*}	
	where $1 < 2_* < 2$ satisfies $(2_*)^* = 2$. This easily yields \eqref{eq:ptwise_uw} after using \eqref{eq:f_u} and the fact that $\norm{u}_{L^{2_*}(\Rn)} \lesssim \norm{u}_{L^2(\Rn)}$ because $2_* < 2$ and \eqref{eq:mtilde}.
	
	\subsubsection*{Step 3: Self-improvement of integrability.} We claim that 
	\begin{equation} \label{eq:self_improvement}
		\norm{u}_{L^r(\Rn)} 
		\lesssim 
		\norm{u}_{L^{r_{**}}(\Rn)}
		+ \norm{u}_{H^1(\Rn)} + 1, 
		\qquad 
		\text{for any } \frac{n}{n-2} < r < +\infty.
	\end{equation}

	Indeed, fix one such $r$. Integrating \eqref{eq:ptwise_uw} and using the bound for $\abs{\Omega_u}$ in \eqref{eq:mtilde}, we obtain 
	\begin{equation*}
		\norm{u}_{L^r(\Rn)} 
		\lesssim 
		\norm{w}_{L^r(\Rn)} + \norm{u}_{H^1(\Rn)} + 1,
	\end{equation*}
	so to get \eqref{eq:self_improvement} it suffices to obtain a bound for $\norm{w}_{L^r(\Rn)}$. For that, just use \eqref{eq:w_comparable}, \cite[Theorem V.1.1]{Stein} and \eqref{eq:f_u}	
	\begin{equation*}
		\norm{w}_{L^r(\Rn)}
		\approx 
		\norm{w_{-\Delta}}_{L^r(\Rn)} 
		\lesssim 
		\norm{f}_{L^{r_{**}}(\Rn)}
		\lesssim 
		\norm{u}_{L^{r_{**}}(\Rn)} + 1, 
	\end{equation*}
	where the application of \cite[Theorem V.1.1]{Stein} is justified because $r > \frac{n}{n-2} = 1^{**}$ implies that $r_{**} > 1$. This finishes the proof of \eqref{eq:self_improvement}.
	
	Since $u \in H^1(\Rn)$, by the Sobolev embedding we have $u \in L^{2^*}(\Rn)$ (note that it already holds $2^* > \frac{n}{n-2}$, which makes \eqref{eq:self_improvement} available). Therefore, noting that $\norm{u}_{L^p(\Rn)} \lesssim \norm{u}_{L^q(\Rn)}$ for $p < q$ because of \eqref{eq:mtilde}, after iterating \eqref{eq:self_improvement} we obtain that $u \in L^r(\Rn)$ for any $1 < r < +\infty$, and the constants only depend on $r$, $\norm{u}_{H^1(\Rn)}$ and our admissibility conditions.
	
	
	\subsubsection*{Step 4: Estimate at the endpoint.} To achieve the $L^\infty$ bound, we will localize our estimates. For a fixed $z \in \Rn$ (in the sequel, it is important that we do not allow constants to depend on $z$), we consider the localized Newtonian potentials 
	\[\widehat{w}(x) := \int_\Rn \!\! \Gamma_\cL (x, y) \, f(y) \, \mathbf{1}_{B_1(z)}(y) \, dy
	\; , \;\; \widehat{w}_{-\Delta} (x) := \int_\Rn \!\! \Gamma_{-\Delta}(x, y) \, f(y) \, \mathbf{1}_{B_1(z)}(y) \, dy,
	\;\; x \in \Rn.\]
	These are again well-defined and satisfy $\cL \widehat{w} = f$ and $-\Delta \widehat{w}_{-\Delta} = f$ in $B_1(z)$. Thus, we can apply the same argument as in Step 2 to obtain 
	\[
	u(z) 
	\lesssim 
	\widehat{w}(z) + \norm{u}_{L^{2^*}(\Rn)} + \norm{\widehat{w}}_{L^{2^*}(B_1(z))}.
	\]
	Noting that $\abs{\widehat{w}} \leq \abs{w}$ (recalling $\Gamma_{\cL}, \Gamma_{-\Delta}, f \geq 0$), it follows that $\norm{\widehat{w}}_{L^{2^*}(B_1(z))} \leq \norm{w}_{L^{2^*}(B_1(z))}$, so that the very same computations from Step 2 end up giving us 
	\begin{equation} \label{eq:ptwise_w0}
		u(z) 
		\lesssim 
		\widehat{w}(z) + \norm{u}_{H^1(\Rn)} + 1.
	\end{equation}

	On the other hand, if $r$ is large (and finite), it holds by \cite[Theorem V.1.1]{Stein}, \eqref{eq:f_u} and \eqref{eq:mtilde}
	\begin{equation*}
		\norm{\widehat{w}_{-\Delta}}_{L^r(B_1(z))} 
		\lesssim 
		\norm{f}_{L^{r_{**}}(B_1(z))}
		\lesssim 
		\norm{u}_{L^{r_{**}}(B_1(z))} + 1
		\lesssim 
		\norm{u}_{L^r(\Rn)} + 1,
	\end{equation*}
	and \cite[Theorem 9.9]{GT} yields 
	\begin{equation*}
		\norm{\nabla^2 \widehat{w}_{-\Delta}}_{L^r(B_1(z))} 
		\lesssim 
		\norm{f}_{L^r(B_1(z))}
		\lesssim 
		\norm{u}_{L^r(\Rn)} + 1.
	\end{equation*}
	Putting both estimates together with the well-known fact of Sobolev spaces that one can interpolate the remaining intermediate derivatives (see e.g. \cite[Theorem 5.2]{AF}), we have 
	\begin{equation*}
		\norm{\widehat{w}_{-\Delta}}_{W^{2, r}(B_1(z))}
		\lesssim 
		\norm{u}_{L^r(\Rn)} + 1,
	\end{equation*}
	which together with the supercritical Sobolev embeddings (put, e.g., $r := 2n > n$) yields that 
	\begin{equation*}
		\norm{\widehat{w}_{-\Delta}}_{L^\infty(B_1(z))}
		\lesssim 
		\norm{u}_{L^{2n}(\Rn)} + 1.
	\end{equation*}
	Jointly with \eqref{eq:w_comparable} and \eqref{eq:ptwise_w0}, we have proved 
	\begin{equation*}
		u(z) 
		\lesssim 
		\norm{u}_{L^{2n}(\Rn)} + \norm{u}_{H^1(\Rn)} + 1, 
	\end{equation*} 
	and we already obtained uniform bounds for the first term in the second half of Step 3 before. One may readily check that the bounds do not depend on $z$ (all our estimates were translation invariant), which readily gives the desired result.
	
	\subsubsection*{Case 2: $n=2$.} This case is easier because of the fundamental solution exhibiting a logarithmic profile (see \cite[(4.1)]{GT}, \cite[Theorem 2.17]{Folland}). Define, similarly as above,
	\[\widetilde{w}(x) := \int_\Rn \!\! \Gamma_\cL (x, y) \, f(y) \, \mathbf{1}_{B_{1/4}}(y) \, dy
	\; , \;\; \widetilde{w}_{-\Delta} (x) := \int_\Rn \!\! \Gamma_{-\Delta}(x, y) \, f(y) \, \mathbf{1}_{B_{1/4}}(y) \, dy,
	\;\; x \in \Rn,\]
	where the reader may note that we now use the origin as our center, instead of an arbitrary $z$; but constants will be independent of this choice. Thus, as in Steps 2 and 4 before, we obtain 
	\begin{equation} \label{eq:dim_2}
		(u-\widetilde{w})^+(0)
		\lesssim 
		\left( \fint_{B_{1/4} \setminus B_{1/8}} \abs{u-\widetilde{w}}^2 \, dx \right)^{1/2} 
		\lesssim 
		\norm{u}_{L^2(\Rn)} + \norm{\widetilde{w}}_{L^2(B_{1/4})}.
	\end{equation}
	To estimate the last term, an easy computation recalling $\Gamma_\cL  \approx \Gamma_{-\Delta}$ yields, for $z \in B_{1/4}$,
	\begin{multline*}
		\abs{\widetilde{w}(z)}
		\approx 
		\abs{\widetilde{w}_{-\Delta}(z)}
		\leq
		\frac1{2\pi}
		\int_{B_{1/4}} \abs{\log|z-y|} \, |f(y)| \,dy
		\\ \lesssim
		\left( \int_{B_{1/2}(z)} \abs{\log|z-y|}^2 \, dy \right)^{1/2}
		\left( \int_{B_{1/4}} (1 + u(y))^2 \, dy \right)^{1/2}
		\lesssim 
		\norm{u}_{L^2(\Rn)} + 1,
	\end{multline*}
	where in the second-to-last estimate we have used Cauchy-Schwarz and the bounds \eqref{hip:negative} and \eqref{hip:F'_F''}, and we have estimated the first term in last step using polar coordinates. Both estimates together yield
	\[
		u(0)
		\leq 
		\widetilde{w}(0) + (u-\widetilde{w})^+(0)
		\lesssim 
		\norm{u}_{L^2(\Rn)} + 1,
	\]
	which already gives the bound for $\norm{u}_{L^\infty(\Rn)}$ because the constants do not depend on our balls being centered at the origin.
\end{proof}
	
Combining Proposition~\ref{prop:bounded} and \eqref{hip:F'_F''}, we infer that
\begin{equation} \label{eq:def_M_1}
	M_1 
	:=
	\sup_{x \in \Rn} \abs{F'(x, u(x))}
	\leq 
	N' + M' \norm{u}_{L^\infty(\Rn)}
	< +\infty
\end{equation}
is also uniformly controlled by the constants in Definition~\ref{def:admissible}.

\begin{corollary} \label{corol:Lap_bdd}
	Assume \eqref{hip:setting}. Then, with $M_1$ defined in \eqref{eq:def_M_1}, it holds
	\begin{equation*} 
		\cL  u \leq M_1 
		\qquad 
		\text{in } \Rn.
	\end{equation*} 
	Moreover, $u$ has a representative defined pointwise by 
	\begin{equation} \label{eq:u_pointwise}
		u(y) = \lim_{\rho \to 0} \fint_{D_\rho(y)} u(x) \, dx, 
		\qquad y \in \Rn,
	\end{equation}
	where the sets $D_\rho(y)$ satisfy $B_{c_D \rho}(y) \subset D_\rho(y) \subset B_{C_D \rho}(y)$ for $c_D, C_D > 0$ only depending on $n$ and $\lambda$. 
\end{corollary}

\begin{proof}
	The first assertion follows from Lemma~\ref{lem:EL} and the definition of $M_1$. For the second, note that, fixed $y \in \Rn$, if we find the solution to $\cL \phi = -1$ in $B_1(y)$ with $\phi = 1/(2n)$ on $\partial B_1(y)$ (similarly as in Lemma~\ref{lem:harnack}, so that $\phi(x) = \abs{x-y}^2 / (2n)$ for $\cL  = -\Delta$), then $x \mapsto u(x) + M_1 \phi$ and $x \mapsto M_1 \phi$ are both $\cL $-subharmonic functions. Therefore, they can be defined at $y$ with limits of mean values as in \eqref{eq:u_pointwise} (for general $\cL $, see \cite[p.9]{C}, relying on the pointwise equivalence of any fundamental solution to that of the Laplacian, proved in \cite{LSW} and \cite[Section 6]{BH}). Subtracting both formulas, we obtain \eqref{eq:u_pointwise} at $y$, as desired. 
\end{proof}

\begin{remark}
	When $\cL  = -\Delta$, one can also take $D_\rho(y) = \partial B_r(y)$ and integrate with respect to $d\mathcal{H}^{n-1}$ instead of~$dx$, thereby recovering the classical mean value formulas.
\end{remark}

The following result (and the generalization which we will prove in Lemma~\ref{lem:harmonic_replacement}) will be key in many  arguments: we will compare $u$ with the function solving the exact PDE that $u$ should solve in $\Omega_u$ (given by Lemma~\ref{lem:EL}). It is a natural analogue in our setting of various lemmas on harmonic extensions that can be found in the literature (e.g., \cite[Lemma 3.2]{AC} and \cite[Lemma 3.8]{V}).

\begin{lemma} \label{lem:harmonic_replacement_easy}
	Assume \eqref{hip:setting}. Fix a small enough ball $B_r(x_0)$ and define $h$ as the solution of $\cL  h = F'(x, u)$ in $B_r(x_0)$, and $h = u$ in $B_r(x_0)^{\mathrm{c}}$. If  $\Vol_q(\{h>0\}) \leq m$, then $u \equiv h$.
\end{lemma}
\begin{proof} 
	We may assume $x_0 = 0$. By Lemma~\ref{lem:EL}, $\cL  (u-h) \geq 0$ in $B_r$. And clearly $u - h = 0$ on $\partial B_r$. Therefore, the maximum principle implies that $u-h \leq 0$ in $B_r$, i.e., $h \geq u \; (\geq 0)$. Then, our assumptions imply that $h \in \K_{\leq m}$, so since $u$ is a minimizer we infer $\F_0(u) \leq \F_0(h)$. With some simple computations using the definition of $h$ (recall also \eqref{hip:A_symmetric} and \eqref{hip:A_elliptic_bounded}), this means 
	\begin{align*} 
		0 
		& \geq 
		\F_0(u) - \F_0(h)
		\nonumber
		\\ & =
		\int_{B_r} \nabla (u-h) \, A \, \nabla (u-h)^{\miT} \, dx
		- 2 \int_{B_r} \nabla h \, A \, \nabla (h-u)^{\miT} \, dx
		-
		2 \int_{B_r} \big( F(x, u) - F(x, h) \big) \, dx
		\nonumber
		\\ & \geq 
		\lambda \int_{B_r} \abs{\nabla (u-h)}^2 dx
		- 2 \int_{B_r} \left( F(x, u) - F(x, h) - F'(x, u)(u-h) \right) dx
		=:
		\mathrm{I} + \mathrm{II}.
	\end{align*}
	The first term is obviously nonnegative. For the second, note that, using \eqref{hip:F'_F''} and Poincaré, 	\begin{equation*}
		\mathrm{II}
		\geq 
		-2 M_2 \int_{B_r} \abs{u-h}^2 dx
		\geq 
		- C M_2 \, r^2 \int_{B_r} \abs{\nabla (u-h)}^2 dx.
	\end{equation*}
		Thus, if $r$ is small enough (so that $CM_2r^2 \leq \lambda/2$, say), we can absorb this contribution and get $0 \geq \frac{\lambda}{2} \int_{B_r} \abs{\nabla (u-h)}^2 dx$ from the first display, which shows that $u \equiv h$.	
\end{proof}

Now we observe that the (very natural) assumption \eqref{hip:sol_not_zero} ensures that $u\equiv 0$ is not a minimizer. With this, we can prove the following:

\begin{lemma}[Saturation of the measure] \label{lem:saturation}
	Assume \eqref{hip:setting}. Then $\Vol_q(\Omega_u) = m$, i.e., $u \in \K_{=m}$.
\end{lemma}
\begin{proof} 
	With the results already obtained, the proof follows as in \cite[Proposition 11.1, step 4]{V}. 
	Assume, for the sake of a contradiction, that $\Vol_q(\Omega_u) < m$.
	Then, using \eqref{hip:q_positive}, find $r > 0$ with $\Vol_q(\Omega_u) + \Vol_q(B_r(x)) \leq m$ for any $x \in \Rn$. 
	
	Let us show that $\Omega_u$ is closed. Indeed, let $x_j \in \Omega_u$ converge to $x_0$. Then, for $j$ large, $x_j \in B_r(x_0)$. Now, Lemma~\ref{lem:harmonic_replacement_easy} implies that $u \equiv h$ (its $F'$-harmonic extension) in $B_r(x_0)$. Concretely, \eqref{hip:negative} implies that $u$ is $\cL $-superharmonic (i.e. $\cL  u \geq 0$) in $B_r(x_0)$. Thus, the strong minimum principle (and $u(x_j) > 0$) implies that $u > 0$ in $B_r(x_0)$. Concretely $u(x_0) > 0$. 
	
	Next, we show that $\Omega_u$ is open, too. Pick $x_0 \in \Omega_u$, i.e., $u(x_0) > 0$. Then, by \eqref{eq:u_pointwise}, $\int_{B_{C_Dr}(x_0)} u \, dx > 0$ if $r > 0$ is small, whence $\fint_{\partial B_r(x_0)} u \, d\mathcal{H}^{n-1} > 0$ for (a possibly different) small enough $r > 0$. That means that $u$ does not vanish identically on $\partial B_r(x_0)$. Hence, the $F'$-harmonic extension $h$ of $u$ in $B_r(x_0)$, which again coincides with $u$ in $B_r(x_0)$ by Lemma~\ref{lem:harmonic_replacement_easy}, does not vanish in $B_r(x_0)$ by the strong minimum principle for $\cL $-superharmonic functions (as in the last paragraph). This shows that $B_r(x_0) \subset \Omega_u$. 
	
	Therefore, $\Omega_u$ is both open and closed, so either $u \equiv 0$ or $\Omega_u = \Rn$. Both give the desired contradiction because of \eqref{hip:sol_not_zero} and $\Vol_q(\Omega_u) \leq m$, respectively.
\end{proof}



\section{Asymptotic minimality at small scales for a penalized energy} \label{sec:first_variation}

To follow the original blueprint of \cite{AC}, and ultimately obtain the overdetermined boundary condition, it will be convenient to use the energy functional
\[
\F_\Lambda(v):=\F_0(v)+\Lambda \Vol_q(\{v>0\})
\]
instead of $\F_0$, for some~$\Lambda>0$. This corresponds to considering a Lagrange multiplier $\Lambda$ for our constrained minimization problem. The basic idea in this section is to show that our  minimizer~$u$ of $\F_0$ under the constraint~\eqref{E.volume} is, in a way, ``almost'' an unconstrained minimizer of $\F_\Lambda$ for some~$\Lambda>0$. This introduces some more difficulties, which were already addressed in~\cite{V} in a simpler context. 

Our goal here is to adapt those results to our setting. For this, we can mostly follow \cite[Section 11]{V}, and only few things need to be modified. There is one substantial difference, though: we have not managed to extend~\cite[Subsection 11.3]{V} to our setting: in the presence of $F(x, u)$, with the positive derivative \eqref{hip:negative}, it is not clear to us how to establish the key Almgren-like monotonicity formulas as in \cite[Lemma 11.7]{V} and Caccioppoli inequalities as in the proof of \cite[Proposition 11.4]{V}. The way we circumvent this difficulty is by not showing in this section that~$\Lambda>0$ (we only obtain $\Lambda\geq0$). We will be able to fix this issue in later sections.

The first result we will establish is the following. Note that the first variation of~$\F_\Lambda$ is $$\delta \F_{\Lambda} (u)[\xi] := \delta \F_0 (u) [\xi] + \Lambda \int_{\Omega_u} \div (q\xi)  \, dx,$$ where $\xi \in C^\infty_c(\Rn; \Rn)$ and 
	\begin{multline} \label{eq:first_variation}
		\delta \F_0 (u) [\xi]
		:=
		\int_\Rn 
		\Big( 
		-2\nabla u \, \nabla  \xi \, A \, \nabla u^{\miT}
		 + \nabla u \, A \, \nabla u^{\miT} \, \div \xi 
		+ \nabla u \, (\nabla A \cdot \xi) \, \nabla u^{\miT}	 
		\\- 2 \nabla F(x, u) \cdot \xi 
		- 2 F(x, u) \div \xi 
		\Big) 
		\, dx.
	\end{multline}

\begin{lemma} \label{lem:first_variation}
	Assume \eqref{hip:setting}.
	Then, there exists $\Lambda \geq 0$ such that $u$ is stationary for $\F_{\Lambda}$, i.e.
	\begin{equation} \label{eq:stationary}
		\delta \F_{\Lambda} (u) [\xi] = 0 
		\qquad 
		\forall \; \xi \in C^\infty_c(\Rn; \Rn).
	\end{equation} 
\end{lemma}
\begin{proof} 
	First, let us deduce the expression \eqref{eq:first_variation} for the first variation of $\F_0$. We use the perturbations $u_t := u \circ \Phi_t(x)$, where $\Phi_t(x) := \Psi_t(x)^{-1}$ and $\Psi_t(x) := (I + t\xi)(x)$ for some vector field $\xi \in C^\infty_c(\Rn; \Rn)$. Note that, using the chain rule and a well-known expansion of the determinant around the identity,
	\begin{equation*}
		\nabla \Psi_t(x) = I + t\nabla \xi(x), 
		\quad 
		\nabla \Phi_t(\Psi_t(x)) = I - t\nabla \xi(x) + o(t), 
		\quad 
		\det \nabla \Psi_t (x) = 1 + t\div\xi(x) + o(t).
	\end{equation*}
	With this, the change of variables $y = \Psi_t(x)$ (equiv. $x = \Phi_t(y)$), and some Taylor expansions using \eqref{hip:A_regular} to neglect lower order terms, we get
	\begin{align*} 
		& \int_\Rn \!\!\! \nabla u_t(x) A(x) \nabla u_t(x)^{\miT} dx 
		\! = 
		\!\! \int_\Rn \!\!\! \nabla u(y) \nabla \Phi_t(\Psi_t(y)) A(\Psi_t(y)) \big( \nabla u(y) \nabla \Phi_t(\Psi_t(y)) \big)^{\miT} \! \abs{ \text{det} \nabla \Psi_t (y)} dy
		\\ & =
		\int_\Rn \!\! \nabla u(y) \Big( I - t\nabla \xi(y) + o(t)\Big) \Big(A(y) + t \nabla A(y) \cdot \xi(y) + o(t)\Big) 
		\\ & \qquad\qquad\qquad\qquad\qquad\qquad\qquad\qquad 
		\cdot \Big( I - t\nabla \xi(y) + o(t)\Big)^{\miT} \! \nabla u(y)^{\miT} \Big( 1 + t\div\xi(x) + o(t) \Big) \, dy
		\\ & =
		\int_\Rn \! \nabla u A \nabla u^{\miT} \, dy
		+ t \int_\Rn \! \Big( \! -2\nabla u \, \nabla  \xi \, A \, \nabla u^{\miT} 
		 + \nabla u \, A \, \nabla u^{\miT} \div \xi + \nabla u (\nabla A \cdot \xi) \nabla u^{\miT} \Big) \, dy 
		+ o(t),
	\end{align*}
	Similarly, using \eqref{hip:F_regular} to neglect lower order terms	
	\begin{align*}
		\int_\Rn F(x, u_t(x)) &\, dx
		=
		\int_\Rn F(\Psi_t(y), u(y)) \abs{\text{det} \nabla \Psi_t(y)} \, dy
		\\ &= 
		\int_\Rn \Big( F(y, u(y)) + \nabla F(y, u(y)) \cdot t\xi(y) + o(t) \Big) \Big( 1 + t\div \xi + o(t) \Big) \, dy
		\\ &= 
		\int_\Rn F(y, u(y)) \, dy
		+ t \int_\Rn \Big( \nabla F(y, u(y)) \cdot \xi(y) + F(y, u(y)) \div \xi \Big) \, dy
		+ o(t),
	\end{align*}
	Lastly, using \eqref{hip:q_regular} to neglect lower order terms,
	\begin{align*} 
		\Vol_q(\Omega_{u_t})
		& =
		\int_\Rn q(x) \, \mathbf{1}_{\Omega_u}(\Phi_t(x)) \, dx
		=
		\int_\Rn q(\Psi_t(y)) \, \mathbf{1}_{\Omega_u}(y) \abs{\text{det} \nabla \Psi_t(y)} \, dy
		\\ & =
		\int_{\Omega_u} \Big( q(y) + t\nabla q(y) \cdot \xi(y) + o(t) \Big) \Big( 1 + t\div \xi(y) + o(t) \Big) \, dy
		\\ & =
		\Vol_q(\Omega_u)
		+ t \int_{\Omega_u} \div (q\xi) \, dy + o(t).
	\end{align*}
	All these give our desired formulas in the statement, like \eqref{eq:first_variation}.
	We are left to show \eqref{eq:stationary}. One can essentially argue as in \cite[Section 11.2]{V}. First, note that there is a vector field $\xi \in C^\infty_c(\Rn; \Rn)$ for which $\int_{\Omega_u} \div(q\xi) \, dx = 1$ (indeed, one may simply multiply by $q^{-1}$ the vector field given by \cite[Lemma 11.3]{V}). And this allows us to invoke \cite[Proposition 11.2]{V}  (because its proof does not rely on the explicit form of $\F_0$ or $\Vol_q$) to obtain \eqref{eq:stationary}. 
\end{proof}

Let us now show that the assumption that $u$ is a critical point for $\F_{\Lambda}$ implies that $u$ is almost a minimizer of $\F_{\Lambda}$ at small scales. The following lemma will be the last time that we will use \eqref{hip:F_quadratic} explicitly, until Section~\ref{sec:proof}.


\begin{lemma} \label{lem:almost_minimality}
	Assume \eqref{hip:setting}. Let $B$ be a ball such that $0 < \abs{\Omega_u \cap B} < \abs{B}$. Then, for every $\varepsilon > 0$, there exists $r > 0$ such that for any $B_r(x_0) \subset B$ we have 
	\begin{equation} \label{eq:outwards}
		\F_{\Lambda + \varepsilon}(u) \leq \F_{\Lambda + \varepsilon}(v)
		\quad \text{ for every $v \in H^1(\Rn)$ such that }
		\begin{cases}
			v-u \in H^1_0(B_r(x_0)), \\
			\Vol_q(\Omega_u) \leq \Vol_q(\Omega_v).
		\end{cases}
	\end{equation}
	\begin{equation} \label{eq:inwards}
		\F_{\Lambda - \varepsilon}(u) \leq \F_{\Lambda - \varepsilon}(v)
		\quad \text{ for every $v \in H^1(\Rn)$ such that }
		\begin{cases}
			v-u \in H^1_0(B_r(x_0)), \\
			\Vol_q(\Omega_u) \geq \Vol_q(\Omega_v).
		\end{cases}
	\end{equation}
\end{lemma} 

This is an the analogue of \cite[Proposition 11.10]{V}. The proof, given in \cite[Sections 11.4 and 11.5]{V}, is very robust, and most of it works for fairly general energy functionals. However, there are a couple of arguments that need to be adapted our setting because of the presence of $F, A$ and $q$. We include the details in Appendix~\ref{app:proofs}.

For future reference, let us record here the following immediate consequence of Lemma~\ref{lem:almost_minimality}:

\begin{corollary}
	Assume \eqref{hip:setting}. For every $\varepsilon > 0$ there exists $r_0 > 0$ such that, for every $0 < r \leq r_0$,
	\begin{equation}\label{eq:almost_minimality}
		\F_{\Lambda} (u) \leq \F_{\Lambda} (v) + \varepsilon \abs{B_r}
	\quad  \text{ for every $v \in H^1(\Rn)$ such that }
		v-u \in H^1_0(B_r(x_0)).
	\end{equation}
\end{corollary}

\begin{proof}
	It follows directly from Lemma~\ref{lem:almost_minimality} and the fact that $\F_{\Lambda + \varepsilon}(w) = \F_{\Lambda}(w) + \varepsilon \abs{\{w>0\}}$.
\end{proof}



\section{Lipschitz continuity and non-degeneracy} \label{sec:lip}

Using that our constrained minimizer $u$ is ``almost'' a minimizer of some energy $\F_{\Lambda}$ by~\eqref{eq:almost_minimality}, our goal in this section is to show that~$u$ is Lipschitz and non-degenerate, which will be the key properties in the forthcoming analysis. To do so, we will extend the basic strategy used in~\cite{AC} and~\cite[Sections 3 and 4]{V}.

Let us start by a crucial extension of Lemma~\ref{lem:harmonic_replacement_easy}:

\begin{lemma} \label{lem:harmonic_replacement}
	Assume \eqref{hip:setting}. Fix $\varepsilon > 0$ and a small enough ball $B_r(x_0)$ as in Lemma~\ref{lem:almost_minimality}. Define $h$ as the only solution of the equation $\cL  h = F'(x, u)$ in $B_r(x_0)$ with $h = u$ in $B_r(x_0)^{\mathrm{c}}$. Then $h \geq u$ in $B_r(x_0)$ and 
	\begin{equation} \label{eq:AC_3.2_estimate}
		\int_{B_r(x_0)} \abs{\nabla (u-h)}^2 dx \leq 2 \, \overline{q} \, \frac{\Lambda + \varepsilon}{\lambda} \abs{\{u = 0\} \cap B_r(x_0)}.
	\end{equation}
\end{lemma}
\begin{proof} 
	We may assume $x_0 = 0$, and write $\Lambda' := \Lambda + \varepsilon$. By Lemma~\ref{lem:EL}, $\cL  (u-h) \leq 0$ in $B_r$; and clearly $u - h = 0$ on $\partial B_r$. Therefore, the maximum principle implies that $h \geq u$. Then, by \eqref{eq:outwards}, $\F_{\Lambda'}(u) \leq \F_{\Lambda'}(h)$. This implies, just as in the proof of Lemma~\ref{lem:harmonic_replacement_easy},
	\begin{multline*}
		\!\!\! 0 
		\geq 
		\lambda \! \int_{B_r} \!\! \abs{\nabla (u-h)}^2 dx
		\, - \, 2 \int_{B_r} \!\!\!\! \left( F(x, u) - F(x, h) - F'(x, u)(u-h) \right) dx
		\,-\, \Lambda' \Vol_q(\{u = 0\} \cap B_r)
		\\
		\geq 
		\lambda \int_{B_r} \abs{\nabla (u-h)}^2 dx
		- \frac{\lambda}{2} \int_{B_r} \abs{\nabla (u-h)}^2 dx
		- \Lambda' \,\overline{q} \abs{\{u = 0\} \cap B_r},
	\end{multline*}
	where we have used \eqref{hip:F'_F''} and Poincaré to hide (for $r$ small enough) the second term. The lemma follows.
	\end{proof}


As a consequence, we can show some thickness of $\Rn \setminus \Omega_u$, where we recall the notation $\Omega_u:=\{u>0\}$:

\begin{lemma} \label{lem:exterior_positive_measure}
	Assume \eqref{hip:setting}. If $x_0 \in \pom_u$, then $\abs{\{u=0\} \cap B_r(x_0)} > 0$ for every $r > 0$. 
\end{lemma}
\begin{proof} 
	Fix $x_0 \in \pom_u$. Suppose, on the contrary, that $\abs{\{u=0\} \cap B_r(x_0)} = 0$ for some $r > 0$. Then (possibly using a smaller $r$), by \eqref{eq:AC_3.2_estimate}, $\nabla u \equiv \nabla h$ a.e. in $B_r(x_0)$. Since $u = h$ on $\partial B_r(x_0)$, that translates into $u = h$ a.e. in $B_r$. Therefore $\cL  u = F'(x, u) \geq 0$ in the weak sense in $B_r(x_0)$ (using \eqref{hip:negative}), whence by elliptic regularity (e.g. Theorems 8.8 and 9.19 from \cite{GT}, recalling \eqref{hip:A_regular}, \eqref{hip:F_regular} and \eqref{eq:def_M_1}) $u$ is actually $C^2$ in $B_{r/2}(x_0)$, and hence $\cL $-superharmonic in a classical sense. On the other hand, since $x_0 \in \pom_u$, $u$ is not constantly zero in $B_{r/2}(x_0)$. Thus, by the classical strong minimum principle for $\cL $-superharmonic functions, $u > 0$ in $B_{r/2}(x_0)$. But this clearly contradicts $x_0 \in \pom_u$.
\end{proof}


After these preliminaries, we are ready to proceed to prove the optimal regularity of $u$.

\begin{proposition}[Lipschitz continuity] \label{prop:lipschitz}
	Assume \eqref{hip:setting}. Then $\Omega_u$ is open and  $$L := \norm{\nabla u}_{L^\infty(\Rn)} < + \infty.$$ 
\end{proposition}
\begin{proof} 
		We can mostly follow~\cite[Theorem 3.2]{AC} or~\cite[Section 3.1]{V}. Fix $\varepsilon > 0$ small, and set $\Lambda' := \Lambda + \varepsilon$ to use \eqref{eq:outwards}. 
	
	First, given $x_0 \in \Rn$ and small $r > 0$, let $h$ be as in Lemma~\ref{lem:harmonic_replacement}. Since $h \geq u$ and it is $\cL $-superharmonic in $B_r(x_0)$ by \eqref{hip:negative}, we may follow the proof of \cite[Lemma 3.7]{V} to obtain
	\begin{equation*}
		\abs{ \{u = 0\ \cap B_r(x_0)} \left( \frac1r \fint_{\partial B_r(x_0)} u \, d\mathcal{H}^{n-1} \right)^2
		\lesssim 
		\int_{B_r(x_0)} \abs{\nabla (u-h)}^2 \, dx.
	\end{equation*}
	Dominating the right hand side with \eqref{eq:AC_3.2_estimate}, and cancelling terms (recall Lemma~\ref{lem:exterior_positive_measure}), we obtain
	\begin{equation} \label{eq:integral_linear_growth}
		\fint_{\partial B_r(x_0)} u \, d\mathcal{H}^{n-1} \leq C(n, \lambda, \overline{q}) \sqrt{\Lambda'} \, r 
		\qquad  \forall \, x_0 \in \pom_u, \; r \text{ small enough}.
	\end{equation}
	Recalling \eqref{eq:u_pointwise}, this implies that $u \equiv 0$ on $\pom_u$ just letting $r \to 0$. Hence $\Omega_u$ is open. 

	It only remains to adapt~\cite[Lemma 3.5]{V}.
	Fix $x \in \Omega_u$ sufficiently close to $\pom_u$, and find $x_0 \in \pom_u$ its closest point on the boundary. Set $r := \abs{x-x_0}$, so $r > 0$ because $\Omega_u$ is open.
	Then, by Lemma~\ref{lem:EL}, $\cL  u = F'(x, u)$ in $B_r(x)$, whence $\abs{\cL  u} \leq M_1$ in $B_r(x)$ (using also \eqref{hip:negative} and Corollary~\ref{corol:Lap_bdd}). This allows us to use Lemma~\ref{lem:harnack} to get
	\begin{equation} \label{eq:gradient_estimate}
		\abs{\nabla u(x)} 
		\lesssim
		\frac{1}{r} \norm{u}_{L^\infty(\partial B_{r/2}(x))} + r,
	\end{equation}
	Now, for any fixed $y \in \partial B_{r/2}(x)$, again using Lemma~\ref{lem:harnack} and also \eqref{eq:integral_linear_growth}, we obtain
	\begin{equation} \label{eq:gradient_estimate_aux}
		u(y) 
		\lesssim 
		\fint_{B_{r/4}(y)} u \, dz + r^2
		\lesssim
		\fint_{B_{2r}(x_0)} u \, dz + r^2
		\lesssim
		\sqrt{\Lambda'} \, r + r^2,
	\end{equation}
	which back in \eqref{eq:gradient_estimate} finally yields $\abs{\nabla u(x)} \leq C$ if $x$ is sufficiently close to $\pom_u$ (i.e. when $r$ is small). Outside of a neighborhood of $\pom_u$, standard regularity for elliptic PDE yields that $u$ is actually $C^2$ there, which is sufficient to finish the proof. 
\end{proof}

Recalling Lemma~\ref{lem:EL}, we immediately obtain the following:

\begin{corollary}[Euler--Lagrange (II)] \label{corol:EL2}
	Assume \eqref{hip:setting}. Then $\cL  u = F'(x, u)$ in $\Omega_u$.   
\end{corollary}

As a consequence of the Lipschitz estimate at the boundary (Proposition~\ref{prop:lipschitz}) and the exact equation that $u$ solves in $\Omega_u$ (Corollary~\ref{corol:EL2}), we can now ensure that $\Lambda \neq 0$  using only elementary PDE techniques. It should be emphasized that, up to this moment, we never needed the strict inequality $\Lambda > 0$ in our proofs.

\begin{proposition}[Positivity of Lagrange multiplier] \label{prop:Lambda_positive}
	Assume \eqref{hip:setting}, and let $\Lambda \geq 0$ be the number obtained in Lemma~\ref{lem:first_variation}. Then actually $\Lambda > 0$.
\end{proposition}
\begin{proof} 
	Let us assume, on the contrary, that $\Lambda = 0$. Then, the estimates \eqref{eq:gradient_estimate} and \eqref{eq:gradient_estimate_aux} that we obtained in the proof of Proposition~\ref{prop:lipschitz} self-improve to $\abs{\nabla u(x)} \leq C r = C \dist(x, \pom_u)$ for $x$ close to $\pom_u$, simply using $\Lambda = \Lambda + \varepsilon = \varepsilon$ and letting $\varepsilon \to 0$. On the other hand, by elliptic regularity, we may suppose that $u$ is (qualitatively) regular in $\Omega_u$, at least $C^2$ (one may use Theorems 8.8 and 9.19 from \cite{GT} taking \eqref{hip:F_regular} and \eqref{hip:A_regular} into account). Therefore, $\nabla u$ is a continuous function in $\Omega_u$, which converges to 0 as we approach the boundary; so we can extend it by 0 in a continuous way. This means that, defining $\nabla u := 0$ on $\pom_u$, we have $u \in C^1(\overline{\Omega_u})$. 
	
	Now pick any $x \in \Omega_u$ and denote $B := B_{\dist(x, \pom_u)}(x)$. Clearly $B \subset \Omega_u$, so Corollary~\ref{corol:EL2} and $F' \geq 0$ (see \eqref{hip:negative}) imply that $u$ is $\cL $-superharmonic in $B$. Therefore, if $x_0 \in \pom_u$ is the closest point on the boundary to $x$, the classical strong minimum principle for $\cL $-superharmonic functions implies that $u(x_0) = 0 < u(y)$ for every $y \in \partial (\frac12 B)$. Moreover, recall that we already discussed in the previous paragraph that $u \in C^2(B) \cap C^1(\overline{B})$.  All these allow us to apply the classical Hopf's lemma in $B$ (because $\partial B$ is smooth) from \cite[p.330]{E} to deduce that $\partial_\nu u(x_0) < 0$, which contradicts the fact that $\nabla u = 0$ on $\pom$. 
\end{proof}

Let us finish the section with a classical non-degeneracy estimate, which shows that around any point of the free boundary $\pom_u$, $u$ must exhibit some linear growth.

\begin{proposition}[Non-degeneracy] \label{prop:non_degeneracy}
	Assume \eqref{hip:setting}. Then there exist small enough $\kappa_0, r_0 > 0$ such that, for any $x_0 \in \Rn$ and $0 < r < r_0$, it holds
	\begin{equation*}
		\fint_{\partial B_r(x_0)} u \, d\mathcal{H}^{n-1}
		\leq 
		\kappa_0 \, r
		\qquad \implies \qquad 
		u \equiv 0 \;\text{ in } B_{r/8}(x_0).
	\end{equation*}
\end{proposition}

The proof of this result uses arguments of David--Toro (see e.g.~\cite[Lemma 4.5]{V}). The key idea is that, upon a comparison of $u$ with a carefully chosen competitor that vanishes over a slightly smaller ball, one can show that the estimate $\fint_{\partial B_r(x_0)} u \, d\mathcal{H}^{n-1} \leq \kappa_0 \, r$ is preserved as one decreases $r$, and an iteration of this property ends up giving the result. It is important for the proof that we have already shown that $\Lambda > 0$ in Proposition~\ref{prop:Lambda_positive}. Full details can be found in Appendix~\ref{app:proofs}.



\section{Compact support of minimizers} \label{sec:compact_support}

The volume of the support of the minimizer $u$ is obviously bounded because $u\in \K_{\leq m}$, so
\begin{equation}\label{eq:mtilde}
|\{u>0\}|\leq m/\underline q=:\tm\,.
\end{equation}
But it is not clear a priori whether the support $\Omega_u = \{u > 0\}$ is bounded itself. 
In this section, we prove that this is actually the case: $\Omega_u$ is always bounded. Moreover, we provide uniform bounds for the number and diameter of its connected components. The fact that the support of the minimizers is not bounded a priori is one of the main differences of our work with respect to \cite{AC} and most of the subsequent works on this topic. 

We start  by partitioning $\Omega_u$ into a family of {\em enlarged connected components}\/, by which one means a union of connected components which are close to each other. More precisely, let us make the following definition, where ``c.c.'' is short for  ``connected component'':
 
\begin{definition} \label{def:ecc}
	Given a connected component $V$ of $\Omega_u$, we define the \emph{enlarged connected component}\/ to which it belongs as
	\begin{equation*}
		\ECC(V)
		:=
		\bigcup 
		\left\{\!\begin{aligned}
			&
			\exists \;r\geq0  \text{ and $V_j$ ($0 \leq j \leq r$) c.c.'s of $\Omega_u$,} \\ 
			 \text{$U$ c.c.\ of $\Omega_u$} : \;\; &
			\text{such that } V =: V_0, V_1, \ldots, V_{r-1}, V_r := U, \text{ and }\\  & \text{$\dist(\overline{V_j}, \overline{V_{j+1}}) \leq 1/4$ for every $0 \leq j \leq r-1$}
		\end{aligned}\right\}.
	\end{equation*}
\end{definition}

Let us now show that~$\Omega_u$ is bounded:

\begin{proposition}[Boundedness of supports] \label{prop:Omega_bounded} 
	Assume \eqref{hip:setting}. Then $\Omega_u$ is a bounded set.
	In fact, $\Omega_u$ has at most $N$ enlarged connected components, each of them having diameter bounded above by $D$, where $N, D > 0$ are uniform constants. 
\end{proposition} 
\begin{proof}
	This follows from the non-degeneracy estimate from Proposition~\ref{prop:non_degeneracy} and $\abs{\Omega_u} \leq \tm$ (recall \eqref{eq:mtilde}). To spell out the details, let us fix  a connected component $V$ of $\Omega_u$. 
	Denote 
	\begin{equation*}
		\mathcal{J} := \big\{ \mathbf{j} \in \Z^n : \ECC(V) \cap (\mathbf{j} + [0, 1)^n) \neq \emptyset \big\}.
	\end{equation*}
	Then, for a given $\mathbf{j} \in \mathcal{J}$, find $x_{\mathbf{j}} \in \ECC(V) \cap (\mathbf{j} + [0, 1)^n)$. 
	
	We claim (and will prove later) that there exists some uniform $\rho > 0$, and some point $z_{\mathbf{j}}$ such that $B_\rho(z_{\mathbf{j}}) \subset \Omega_u \cap B_{1/8}(x_{\mathbf{j}})$. Assuming this claim for the moment, we compute
	\begin{equation*}
		\tm
		\geq 
		|\Omega_u|
		\geq 
		\bigg\lvert \bigcup_{\mathbf{j} \in \mathcal{J}} B_\rho(z_{\mathbf{j}}) \bigg\rvert
		\geq
		C(n)
		\sum_{\mathbf{j} \in \mathcal{J}} \abs{B_\rho(z_{\mathbf{j}})}
		=
		C(n) \rho^n \abs{\mathcal{J}}, 
	\end{equation*}
	where the third inequality follows simply by bounded overlap of the $B_\rho(z_{\mathbf{j}})$, because each of them lives in $B_{1/8}(x_{\mathbf{j}})$, which are themselves balls with bounded overlap because $x_{\mathbf{j}} \in \mathbf{j} + [0, 1)^n$. This computation  shows that $\abs{\mathcal{J}} < +\infty$, and actually provides a uniform bound. Since $\bigcup_{\mathbf{j} \in \mathcal{J}} (\mathbf{j} + [0, 1]^n)$ is connected by the definition of $\ECC(V)$, and since it contains $\ECC(V)$, we conclude that $\diam(\ECC(V))$ is uniformly bounded. 
	
	Moreover, given any $x_{\mathbf{j}} \in \ECC(V) \cap (\mathbf{j} + [0, 1)^n)$, where $\mathbf{j} \in \mathcal{J}$, we have found $z_{\mathbf{j}}$ such that it holds $B_\rho(z_{\mathbf{j}}) \subset \Omega_u \cap B_{1/8}(x_{\mathbf{j}})$. Clearly, by definition of $\ECC(V)$, we have $B_\rho(z_{\mathbf{j}}) \subset \ECC(V)$. Therefore, $|\ECC(V)| \gtrsim \rho^n$. Since the enlarged connected components are disjoint, this implies that the amount of them is $\lesssim \widetilde{m} / \rho^n$.
	
	Therefore, to finish the proof of the proposition, it only remains to prove the claim. First, note that there exists a uniform $0 < \rho < 1/40$ such that for any $y \in \pom_u$, there exists $z \in \Omega_u$ with $B_\rho(z) \subset \Omega_u \cap B_{1/20}(y)$ (this is the so-called ``Corkscrew condition'', see Figure~\ref{fig:corkscrew}). This was already spotted in \cite[Lemma 3.7]{AC}: for $r$ small, by non-degeneracy (see Proposition~\ref{prop:non_degeneracy}) $\norm{u}_{L^\infty(B_r(y))} > \kappa_0 r$, so there exists $z \in B_r(y)$ with $u(z) \geq \kappa_0 r$, and then by Lipschitz continuity (see Proposition~\ref{prop:lipschitz}) we have $u \geq \kappa_0 r / 2$ in $B_{\kappa_0r/(2L)}(z)$. Thus, writing $\rho := \kappa_0 r / (2L)$, we have $B_\rho(z) \subset \Omega_u \cap B_{1/20}(y)$ (because we can always take $r, \rho < 1/40$).
	
	\begin{figure}[b!]		
		\centering 
		\resizebox{0.5\textwidth}{!}{
			\begin{tikzpicture}[scale=1]
				\filldraw[thin, fill=lightblue] (12.889, 6.2) .. controls (12.91, 4.277) and (12.597, 3.128) .. (12.196, 2.309) .. controls (11.796, 1.49) and (11.308, 1.002) .. (10.914, 0.964) .. controls (10.521, 0.926) and (10.221, 1.337) .. (10.308, 1.88) .. controls (10.395, 2.424) and (10.869, 3.1) .. (10.932, 3.782) .. controls (10.994, 4.465) and (10.646, 5.155) .. (10.28, 5.486) .. controls (9.914, 5.817) and (9.531, 5.789) .. (9.496, 5.702) .. controls (9.462, 5.615) and (9.775, 5.469) .. (10.092, 5.172) .. controls (10.409, 4.876) and (10.73, 4.43) .. (10.59, 4.249) .. controls (10.451, 4.068) and (9.852, 4.152) .. (9.298, 3.842) .. controls (8.744, 3.532) and (8.235, 2.828) .. (7.929, 2.762) .. controls (7.622, 2.696) and (7.518, 3.267) .. (7.573, 3.64) .. controls (7.629, 4.012) and (7.845, 4.187) .. (7.66, 4.486) .. controls (7.476, 4.786) and (6.891, 5.211) .. (6.455, 5.273) .. controls (6.02, 5.336) and (5.734, 5.037) .. (5.671, 4.817) .. controls (5.609, 4.598) and (5.769, 4.458) .. (6.013, 4.298) .. controls (6.257, 4.138) and (6.584, 3.957) .. (6.612, 3.904) .. controls (6.64, 3.852) and (6.368, 3.929) .. (6.26, 3.859) .. controls (6.152, 3.789) and (6.208, 3.573) .. (6.365, 3.431) .. controls (6.521, 3.288) and (6.779, 3.218) .. (6.905, 3.026) .. controls (7.03, 2.835) and (7.023, 2.521) .. (6.946, 2.365) .. controls (6.87, 2.208) and (6.723, 2.208) .. (6.563, 2.309) .. controls (6.403, 2.41) and (6.229, 2.612) .. (6.121, 2.786) .. controls (6.013, 2.96) and (5.971, 3.107) .. (5.797, 3.222) .. controls (5.623, 3.337) and (5.316, 3.42) .. (5.058, 3.445) .. controls (4.8, 3.469) and (4.591, 3.434) .. (4.334, 3.438) .. controls (4.076, 3.441) and (3.769, 3.483) .. (3.564, 3.598) .. controls (3.358, 3.713) and (3.254, 3.901) .. (3.174, 4.166) .. controls (3.093, 4.43) and (3.038, 4.772) .. (2.881, 4.988) .. controls (2.724, 5.204) and (2.466, 5.294) .. (2.212, 5.291) .. controls (1.958, 5.287) and (1.707, 5.19) .. (1.471, 5.108) .. controls (1.235, 5.025) and (1.013, 4.959) .. (0.755, 5.033) .. controls (0.497, 5.107) and (0.202, 5.322) .. (0.014, 5.782) .. controls (0, 9.602) and (0.008, 9.602) .. (0.012, 9.602) .. controls (0.016, 9.602) and (0.016, 9.602) .. (2.161, 9.605) .. controls (4.307, 9.608) and (8.597, 9.613) .. (10.742, 9.616) .. controls (12.887, 9.618) and (12.887, 9.618) .. (12.895, 9.61) .. controls (12.903, 9.602) and (12.92, 9.586) .. (12.89, 6.191);
				\draw[red, fat] 
				(6.051, 2.843) circle[radius=2.258];
				\node[circle, fill, inner sep=3pt] at (6.051, 2.843) {};
				\draw[fat] 
				(4.585, 4.56) circle[radius=0.7];
				\node[circle, fill, inner sep=3pt] at (4.561, 4.539) {};
				\node[anchor=center, font=\huge, text=black] at (1.358, 8.834) {$\{u>0\}$};
				\node[anchor=center, font=\LARGE, text=black] at (3.706, 5.391) {$B_\rho(z)$};
				\node[anchor=center, font=\LARGE] at (4.376, 4.785) {$z$};
				\node[anchor=center, font=\LARGE, text=black] at (5.742, 2.52) {$y$};
				\node[anchor=center, font=\LARGE, text=black] at (3.114, 2.024) {$B_r(y)$};
				\node[anchor=center, font=\LARGE, text=black] at (7.44, 7.606) {$B_{1/20}(y)$};
				\draw[myblue, fat] (1.175, 0) .. controls (1.088, 0.154) and (1.024, 0.279) .. (0.978, 0.372) .. controls (0.932, 0.466) and (0.904, 0.528) .. (0.872, 0.602) .. controls (0.84, 0.677) and (0.805, 0.762) .. (0.773, 0.849) .. controls (0.74, 0.935) and (0.71, 1.022) .. (0.669, 1.162) .. controls (0.628, 1.303) and (0.577, 1.497) .. (0.542, 1.643) .. controls (0.508, 1.789) and (0.49, 1.887) .. (0.474, 1.989) .. controls (0.459, 2.09) and (0.446, 2.195) .. (0.435, 2.341) .. controls (0.424, 2.487) and (0.415, 2.674) .. (0.412, 2.816) .. controls (0.409, 2.959) and (0.412, 3.057) .. (0.418, 3.159) .. controls (0.424, 3.261) and (0.433, 3.367) .. (0.451, 3.507) .. controls (0.469, 3.646) and (0.497, 3.817) .. (0.522, 3.95) .. controls (0.546, 4.083) and (0.568, 4.177) .. (0.603, 4.3) .. controls (0.637, 4.423) and (0.684, 4.575) .. (0.733, 4.716) .. controls (0.782, 4.858) and (0.834, 4.987) .. (0.893, 5.122) .. controls (0.953, 5.256) and (1.021, 5.395) .. (1.092, 5.526) .. controls (1.162, 5.657) and (1.236, 5.78) .. (1.323, 5.913) .. controls (1.411, 6.047) and (1.512, 6.189) .. (1.616, 6.322) .. controls (1.72, 6.455) and (1.826, 6.578) .. (1.932, 6.693) .. controls (2.039, 6.807) and (2.147, 6.913) .. (2.255, 7.013) .. controls (2.364, 7.112) and (2.474, 7.204) .. (2.59, 7.295) .. controls (2.707, 7.385) and (2.831, 7.474) .. (2.965, 7.56) .. controls (3.098, 7.647) and (3.242, 7.733) .. (3.416, 7.823) .. controls (3.59, 7.912) and (3.794, 8.007) .. (4.02, 8.092) .. controls (4.246, 8.178) and (4.493, 8.256) .. (4.668, 8.305) .. controls (4.842, 8.354) and (4.944, 8.376) .. (5.064, 8.396) .. controls (5.183, 8.417) and (5.322, 8.436) .. (5.501, 8.451) .. controls (5.681, 8.466) and (5.902, 8.477) .. (6.082, 8.478) .. controls (6.261, 8.48) and (6.399, 8.473) .. (6.568, 8.456) .. controls (6.737, 8.438) and (6.937, 8.41) .. (7.122, 8.375) .. controls (7.308, 8.34) and (7.481, 8.297) .. (7.643, 8.25) .. controls (7.805, 8.203) and (7.958, 8.15) .. (8.12, 8.086) .. controls (8.282, 8.021) and (8.454, 7.943) .. (8.612, 7.864) .. controls (8.769, 7.784) and (8.913, 7.702) .. (9.034, 7.628) .. controls (9.156, 7.553) and (9.256, 7.486) .. (9.361, 7.409) .. controls (9.466, 7.332) and (9.577, 7.246) .. (9.654, 7.184) .. controls (9.73, 7.122) and (9.772, 7.086) .. (9.843, 7.02) .. controls (9.913, 6.955) and (10.011, 6.86) .. (10.11, 6.757) .. controls (10.209, 6.654) and (10.308, 6.544) .. (10.385, 6.452) .. controls (10.463, 6.361) and (10.52, 6.289) .. (10.577, 6.211) .. controls (10.635, 6.133) and (10.693, 6.05) .. (10.762, 5.945) .. controls (10.83, 5.839) and (10.908, 5.711) .. (10.976, 5.591) .. controls (11.043, 5.47) and (11.101, 5.358) .. (11.144, 5.27) .. controls (11.187, 5.183) and (11.215, 5.12) .. (11.255, 5.02) .. controls (11.296, 4.919) and (11.348, 4.78) .. (11.398, 4.629) .. controls (11.448, 4.478) and (11.495, 4.316) .. (11.531, 4.171) .. controls (11.567, 4.026) and (11.593, 3.898) .. (11.611, 3.8) .. controls (11.629, 3.701) and (11.639, 3.633) .. (11.649, 3.551) .. controls (11.659, 3.469) and (11.669, 3.374) .. (11.676, 3.267) .. controls (11.684, 3.16) and (11.689, 3.042) .. (11.691, 2.917) .. controls (11.692, 2.792) and (11.689, 2.66) .. (11.683, 2.538) .. controls (11.677, 2.416) and (11.667, 2.304) .. (11.657, 2.208) .. controls (11.647, 2.111) and (11.635, 2.029) .. (11.621, 1.942) .. controls (11.607, 1.854) and (11.589, 1.76) .. (11.567, 1.661) .. controls (11.546, 1.561) and (11.52, 1.456) .. (11.497, 1.368) .. controls (11.474, 1.281) and (11.454, 1.211) .. (11.43, 1.136) .. controls (11.406, 1.061) and (11.378, 0.981) .. (11.348, 0.9) .. controls (11.318, 0.818) and (11.285, 0.734) .. (11.253, 0.657) .. controls (11.221, 0.58) and (11.189, 0.509) .. (11.15, 0.425) .. controls (11.11, 0.342) and (11.062, 0.245) .. (10.94, 0.022);
			\end{tikzpicture}
		}
		\caption{Given any free boundary point $y \in \partial \{u > 0\}$ there is always an interior ball $B_\rho(z) \subset \Omega_u \cap B_1(y)$, with $\rho > 0$ uniform.}
		\label{fig:corkscrew}
	\end{figure}
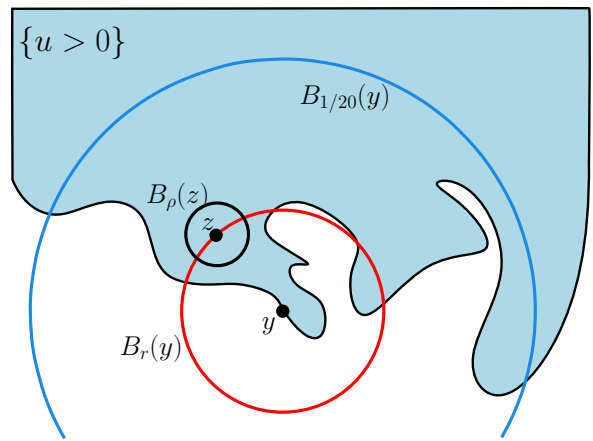
	
	Then, fix $\mathbf{j} \in \mathcal{J}$, and $\rho > 0$ as in last paragraph. If $B_\rho(x_{\mathbf{j}}) \subset \Omega_u$, we can simply take $z_{\mathbf{j}} := x_{\mathbf{j}}$ and finish. Otherwise, there exists $y_{\mathbf{j}} \in B_\rho(x_{\mathbf{j}}) \cap \pom_u$. Applying the method in the previous paragraph, we find $z_{\mathbf{j}}$ such that $B_\rho(z_\mathbf{j}) \subset \Omega_u \cap B_{1/20}(y_{\mathbf{j}}) \subset \Omega_u \cap B_{1/8}(x_{\mathbf{j}})$, where the last inclusion holds because $|y_{\mathbf{j}}-x_{\mathbf{j}}| < \rho < 1/40$. This finishes the proof of the claim, and therefore of the proposition.
\end{proof}

Note that this proposition does {\em not}\/ provide a uniform bound for the diameter of $\Omega_u$. In fact, without the periodicity assumption \eqref{hip:periodic}, which we will not use until Section~\ref{sec:proof}, uniform bounds do not exist in general, as discussed in Appendix~\ref{S.appendix}.



\section{The free boundary} \label{sec:boundary}

In Section~\ref{sec:lip}, we showed the Lipschitz continuity and the non-degeneracy estimates for our minimizer $u$, and also that it behaves {almost} like a minimizer of $\F_\Lambda$ for some $\Lambda > 0$ (more precisely, it is a critical point and it satisfies some ``asymptotic minimality'' as in \eqref{eq:almost_minimality}). Therefore, we are in the position to mostly follow the classical blueprint of~\cite{AC} or~\cite[Sections 6--10]{V} to show the regularity of the free boundary $\pom_u$, along with estimating the size of its singular set.  The arguments will be of local nature, so Hypothesis~\eqref{hip:periodic} will not be needed.


\subsection{Analysis of blow-up limits}

In this subsection, we will construct appropriate blow-up limits at free boundary points, and will prove that they are 1-homogeneous. We will need to introduce some changes in the classical strategy  to deal with the terms introduced by $F$; for instance, we will need to introduce a modified version of Weiss formula for it to become (almost) monotone in our setting. We will show that $F$ basically disappears when blowing up around free boundary points, which will morally take us back to the classical setting.

In this subsection, we will fix a free boundary point. Translating the origin, we can take this point as $0 \in \pom_u$ without any loss of generality, which makes the notation less cumbersome. With this in mind, we define, for $r > 0$,
\begin{equation*}
	u_r(x) := \frac{u(rx)}{r}, 
	\qquad 
	F_r(x, u) := F(rx, ru), 
	\qquad 
	F_0(x, u) := 0,
\end{equation*}
where the last definition is coherent because $F(\cdot, 0) \equiv 0$, see \eqref{hip:zero}. Similarly,
\begin{equation*}
	A_r(x) := A(rx), 
	\qquad 
	A_0(x) := A(0), 
	\qquad 
	q_r(x) := q(rx), 
	\qquad 
	q_0(x) := q(0).
\end{equation*} 

With all these in mind, we set, for $r \geq 0$ and $D \subset \Rn$,
\begin{equation*}
	\F_{\Lambda, r}(v, D)
	:=
	\int_D \nabla v \, A_r \, \nabla v^{\miT} \, dx - 2 \int_D F_r(x, v) \, dx + \Lambda \Vol_{q_r}(\Omega_v \cap D),
\end{equation*}
and we will omit $D$ when $D = \Rn$. Here and in what follows, $\Omega_v:=\{v>0\}$.
Thus, 
\begin{equation*}
	\F_{\Lambda, 0}(v, D)
	:=
	\int_D \nabla v(x) \, A(0) \, \nabla v(x)^{\miT} \, dx + \Lambda q(0) \abs{\Omega_v \cap D}.
\end{equation*}
Since  $u(0) = 0$ because $0 \in \pom_u$, we note that the rescalings of the minimizer~$u$ yield, for $r > 0$,
\begin{equation} \label{eq:rescaling_bounds_u}
	\norm{\nabla u_r}_\infty = \norm{\nabla u}_\infty = L, 
	\quad \text{whence} \quad 
	u_r(x) = u_r(x) - u_r(0) \leq L\abs{x}, 
	\quad x \in \Rn.
\end{equation}

We start with the following auxiliary results:

\begin{lemma} \label{lem:convergence_parameters}
	Assume \eqref{hip:setting} and $0 \in \pom_u$. Fix $R > 0$. Then, as $r \to 0$, $F_r(\cdot, u_{r}(\cdot)) = O(r)$, $\abs{A_r - A_0} = O(r)$, and $\abs{q_r - q_0} = O(r)$. Concretely, $F_r(\cdot, u_{r}(\cdot)) \to 0$, $A_r \to A_0$ and $q_r \to q_0$ in $L^\infty(B_R)$, as $r \to 0$.
\end{lemma}
\begin{proof}
	Using \eqref{hip:zero}, \eqref{hip:negative}, \eqref{eq:def_M_1} and Proposition~\ref{prop:lipschitz}, we get, for $x \in B_R$, 
	\begin{equation*}
		\abs{F_r(x, u_{r_j}(x))}
		=
		\abs{F(rx, u(rx)) - F(rx, 0)}
		\leq 
		M_1 \abs{u(rx) - u(0)}
		\leq 
		M_1 L \, r R
		\underset{r \to 0}{\longrightarrow}
		0.
	\end{equation*}
	And then, using \eqref{hip:A_regular}, we compute $\abs{A_r(x) - A_0(x)} = \abs{A(rx) - A(0)} \leq \norm{\nabla A}_\infty r R \to 0$ as $r \to 0$. And similarly for $q$, this time using \eqref{hip:q_regular}.
\end{proof}

\begin{lemma}[Convergence to blow-ups] \label{lem:convergence_blowups}
	Assume \eqref{hip:setting} and $0 \in \pom_u$. Fix $R > 0$. Let $r_j \to 0$. Then, up to taking a subsequence, there exists $u_0 \in H^1(B_R)$ such that, for every $0 < r < R$,
	\begin{enumerate}
		\item $u_{r_j} \to u_0$ in $H^1(B_r)$ and uniformly, as $j \to +\infty$,
		
		\item \label{item:conv_indicator} $\mathbf{1}_{\Omega_{u_{r_j}}} \to \mathbf{1}_{\Omega_{u_0}}$ in $L^1(B_r)$ and pointwise a.e., as $j \to +\infty$.
	\end{enumerate}
\end{lemma}
\begin{proof} 
	Fix $0 < r < R$ as in the statement.
	
	\subsubsection*{Step 1: Weak convergence.} First note that \eqref{eq:rescaling_bounds_u} easily yields that te sequence $\{u_{r_j}\}_j$ is uniformly bounded in $H^1(B_R)$. Therefore, by standard compactness results, there exists $u_0 \in H^1(B_R)$ such that (up to taking a subsequence) as $j \to \infty$, we have $u_{r_j} \longrightarrow u_0$ both in $L^2(B_R)$ and uniformly in $B_R$ (by Arzel\`a--Ascoli, because of the uniform Lipschitz constant in \eqref{eq:rescaling_bounds_u}), and $\nabla u_{r_j} \rightharpoonup \nabla u_0$ in $L^2(B_R)$. Clearly, that implies 
	\begin{equation} \label{eq:less_liminf_gradient}
		\norm{\nabla u_0}_{L^2(B_r)} 
		\leq 
		\liminf_{j \to +\infty} \norm{\nabla u_{r_j}}_{L^2(B_r)}.
	\end{equation}
	Also, similarly as in \eqref{eq:less_liminf_gradient_A_0},
	\begin{equation} \label{eq:less_liminf_gradient_A}
		\int_{B_r} \nabla u_0 \, A_0 \, \nabla u_0^{\miT} \, dx
		\leq 
		\liminf_{j \to +\infty} \int_{B_r} \nabla u_{r_j} \, A_0 \, \nabla u_{r_j}^{\miT} \, dx.
	\end{equation}

	Moreover, by the uniform pointwise convergence $u_j \to u_0$, if $x \in \Omega_{u_0}$ (that is, $u_0(x) > 0$), necessarily $u_{r_j}(x) > 0$ for sufficiently large $j$. This means 
	\begin{equation} \label{eq:less_liminf_volume}
		\mathbf{1}_{\Omega_{u_0}} 
		\leq 
		\liminf_{j \to +\infty} \mathbf{1}_{\Omega_{u_{r_j}}}, 
		\quad \text{so by Fatou} \quad 
		\Vol_q(\Omega_{u_0} \cap B_r) \leq \liminf_{j \to +\infty} \Vol_q(\Omega_{u_{r_j}} \cap B_r).
	\end{equation}

	Therefore, to obtain the desired convergences, it suffices to obtain 
	\begin{equation} \label{eq:goal_convergence}
		\norm{\nabla u_0}_{L^2(B_r)} 
		\geq 
		\liminf_{j \to +\infty} \norm{\nabla u_{r_j}}_{L^2(B_r)}
		\quad \text{and} \quad 
		\Vol_q(\Omega_{u_0} \cap B_r) \geq \liminf_{j \to +\infty} \Vol_q(\Omega_{u_{r_j}} \cap B_r).
	\end{equation}
	Indeed, the second condition and \eqref{eq:less_liminf_volume}, along with \eqref{hip:q_positive}, will imply the convergence in \ref{item:conv_indicator} similarly as in \eqref{eq:iff_convergences_q}. Moreover, for simplicity in the notation, we may just assume that the $\liminf$ are actually $\lim$ by passing to a further subsequence.
	
	\subsubsection*{Step 2: Rescaling the energy.} 
	We will prove \eqref{eq:goal_convergence} using variational methods.  Indeed, the key fact is that minimization properties of $u$ with respect to $\F_\Lambda$ translate directly into minimization properties of $u_{r_j}$ with respect to $\F_{\Lambda, r_j}$: an easy change of variables shows that for any $\rho > 0$ we have $\F_{\Lambda, r_j}(u_{r_j}, B_\rho) = r_j^{-n} \F_\Lambda (u, B_{r_j\rho})$. Therefore, if we fix $\varepsilon > 0$, rescaling \eqref{eq:almost_minimality} for $j \gg 1$ (so that $R r_j \leq r_0$) yields 
	\begin{equation} \label{eq:rescaled_minimizing}
		\F_{\Lambda, r_j} (u_{r_j}, B_R) \leq \F_{\Lambda, r_j}(v, B_R) + C(\overline{q}, R, n) \, \varepsilon, 
		\quad 
		\forall \, v \in H^1(\Rn) \text{ with } v-u_{r_j} \in H^1_0(B_R).
	\end{equation}

To be able to use \eqref{eq:rescaled_minimizing} and still obtain convergences in $B_r$ as in \eqref{eq:goal_convergence}, we define the competitor 
	\begin{equation*}
		\widetilde{u}_{r_j} 
		:=
		\eta \, u_{r_j} + (1-\eta) \, u_0,
	\end{equation*}
	where $\eta$ is a cut-off function satisfying 
	\begin{equation*}
		\eta \in C^\infty(B_R), 
		\quad 
		0 \leq \eta \leq 1, 
		\quad 
		\eta \equiv 0 \text{ in } B_r, 
		\quad 
		\eta = 1 \text{ on } \partial B_R,
	\end{equation*}
	so that $\widetilde{u}_{r_j} \equiv u_0$ in $B_r$ and $\widetilde{u}_{r_j} - u_{r_j} \in H^1_0(B_R)$. Therefore, using \eqref{eq:rescaled_minimizing} with $v = \widetilde{u}_{r_j}$, we obtain 
	\begin{equation} \label{eq:minimizing_rearranged}
		\F_{\Lambda, r_j} (u_{r_j}, B_r)
		\leq 
		\F_{\Lambda, r_j} (u_0, B_r)
		+ C \varepsilon
		+ \left( \F_{\Lambda, r_j}(\widetilde{u}_{r_j}, B_R \setminus B_r) - \F_{\Lambda, r_j}(u_{r_j}, B_R \setminus B_r) \right).
	\end{equation}
	Our goal now is to disregard the last term and take limits.
	
	\subsubsection*{Step 3: Terms with $F$ vanish in the limit.} Before that, let us show that, after taking limits, all the terms with $F$ disappear. Indeed, if $x \in B_R$, \eqref{eq:rescaling_bounds_u} yields
	\begin{equation*}
		\abs{F_{r_j}(x, \widetilde{u}_{r_j}(x))} 
		+ \abs{F_{r_j}(x, u_{r_j}(x))}
		\lesssim_{M_1} 
		\abs{r_j  \widetilde{u}_{r_j}(x)} 
		+ \abs{r_j u_{r_j}(x)}
		\leq 
		r_j (2LR + |u_0(x)|). 
	\end{equation*}
	Therefore, the fact that $u_0 \in L^2(B_R)$ and the Dominated Convergence Theorem imply that 
	\begin{equation} \label{eq:F_disappears_1}
		\lim_{j \to \infty} \int_{B_R \setminus B_r} \big(F_{r_j}(x, \widetilde{u}_{r_j}) - F_{r_j}(x, u_{r_j})\big) \, dx
		=
		0.
	\end{equation} 
	In fact, similar computations yield 
	\begin{equation} \label{eq:F_disappears_2}
		\lim_{j \to \infty} \int_{B_r} F_{r_j}(x, u_{r_j}) \, dx = 0, 
		\quad \text{ and } \quad 
		\lim_{j \to \infty} \int_{B_r} F_{r_j}(x, u_0) \, dx = 0.
	\end{equation}
	
	With this in mind, our new goal is to show that for every $\varepsilon > 0$ it holds
	\begin{multline} \label{eq:goal_convergence_new}
		\text{LHS}
		:=
		\lim_{j \to +\infty} 
		\int_{B_r} \!\! \nabla u_{r_j} \, A_0 \, \nabla u_{r_j}^{\miT} \, dx
		+ \Lambda \, \Vol_{q_0}(\Omega_{u_{r_j}} \cap B_r)
		\\ \leq 
		\int_{B_r} \!\! \nabla u_0 \, A_0 \, \nabla u_0^{\miT} \, dx
		+ \Lambda \, \Vol_{q_0}(\Omega_{u_0} \cap B_r)
		+ C\varepsilon
		=: 
		\text{RHS}.
	\end{multline}
	Indeed, assuming this momentarily, we can let $\varepsilon \to 0$. Then, noting that every term is non-negative (recall \eqref{hip:A_elliptic_bounded} and \eqref{hip:q_positive}) and that \eqref{eq:less_liminf_gradient_A} and \eqref{eq:less_liminf_volume} hold, we obtain 
	\begin{equation*}
		\lim_{j \to +\infty} \! \int_{B_r} \!\!\!\! \nabla u_{r_j} \, A_0 \, \nabla u_{r_j}^{\miT} \, dx 
		= \!\!
		\int_{B_r} \!\!\!\! \nabla u_0 \, A_0 \, \nabla u_0^{\miT} \, dx, 
		\;\; \text{and} \;\;
		\lim_{j \to +\infty} \!\!\! \Vol_{q_0}(\Omega_{u_{r_j}} \cap B_r) 
		=
		\Vol_{q_0}(\Omega_{u_0} \cap B_r).
	\end{equation*} 
	Which clearly implies \eqref{eq:goal_convergence} because 	\begin{align*}
		0
		& \leq 
		\lim_{j \to +\infty} \int_{B_r} \left( \abs{\nabla u_{r_j}}^2 - \abs{\nabla u_0}^2 \right) \, dx
		=
		\lim_{j \to +\infty} \int_{B_r} \left( \nabla u_{r_j} \nabla u_{r_j}^{\miT} - \nabla u_0 \nabla u_0^{\miT} \right) \, dx
		\\ & =
		\lim_{j \to +\infty} \int_{B_r} \nabla (u_{r_j} - u_0) \nabla (u_{r_j} - u_0)^{\miT} \, dx
		+ 2 \lim_{j \to +\infty} \int_{B_r} \nabla (u_{r_j} - u_0) \, \nabla u_0^{\miT} \, 
		\\ & =
		\lim_{j \to +\infty} \int_{B_r} \nabla (u_{r_j} - u_0) \nabla (u_{r_j} - u_0)^{\miT} \, dx
		\leq 
		\lambda^{-1} \int_{B_r} \nabla (u_{r_j} - u_0) \, A_0 \, \nabla (u_{r_j} - u_0)^{\miT} \, dx
		=
		0,
	\end{align*}
	where we have used \eqref{eq:less_liminf_gradient} and the weak convergence $\nabla u_{r_j} \rightharpoonup \nabla u_0$ in $L^2(B_r)$.
	
	\subsubsection*{Step 4: Analysis of the error terms.}
	Starting from $\text{LHS}$ in \eqref{eq:goal_convergence_new}, approximating $A_0$ by $A_{r_j}$ and $q_0$ by $q_{r_j}$, then applying the inequality \eqref{eq:minimizing_rearranged} noting that the terms with $F$ disappear by \eqref{eq:F_disappears_1} and \eqref{eq:F_disappears_2}, and finally reintroducing $A_0$ and $q_0$ by approximation from $A_{r_j}$ and $q_{r_j}$, we can compute
	\begin{align} \label{eq:long_approximation}
		\text{LHS} 
		& \leq 
		\text{RHS} 
		+ \text{Errors $A$} 
		+ \text{Errors $q$} 
		\nonumber
		\\ & \qquad 
		+ \lim_{j \to +\infty} \int_{B_R \setminus B_r} \left( \nabla \widetilde{u}_{r_j} \, A_{r_j} \, \nabla \widetilde{u}_{r_j}^{\miT} - \nabla u_{r_j} \, A_{r_j} \, \nabla u_{r_j}^{\miT} \right) \, dx 
		\\ & \qquad 
		+ \Lambda \lim_{j \to +\infty} \int_{B_R \setminus B_r} q_{r_j} \left( \mathbf{1}_{\Omega_{\widetilde{u}_{r_j}}} - \mathbf{1}_{\Omega_{u_{r_j}}} \right) \, dx.
		\nonumber
	\end{align}
	Since $\abs{A_{r_j}(x) - A_0(x)} = \abs{A(r_jx) - A(0)} \leq r_j \, r \norm{\nabla A}_\infty$ for $x \in B_r$ by \eqref{hip:A_regular}, we can bound the errors in the approximations between $A_{r_j}$ and $A_0$ as follows: 
	\begin{multline*}
		\abs{\text{Errors $A$}}
		=
		\abs{
			\lim_{j \to +\infty} \int_{B_r} \nabla u_{r_j} \, (A_{r_j} - A_0) \, \nabla u_{r_j} \, dx 
			+ \int_{B_r} \nabla u_0 \, (A_{r_j} - A_0) \, \nabla u_0 \, dx
		}
		\\  
		\leq 
		\lim_{j \to +\infty} r_j \, r \norm{\nabla A}_\infty \left( \norm{\nabla u_{r_j}}_{L^2(B_r)} + \norm{\nabla u_0}_{L^2(B_r)} \right)
		=
		0.
	\end{multline*}
	We have used that $u_{r_j}$ have a uniform bound in $H^1(B_R)$, and $u_0 \in H^1(B_R)$. Similarly, using \eqref{hip:q_regular},
	\begin{multline*}
		\abs{\text{Errors $q$}}
		=
		\Lambda \abs{ \lim_{j \to +\infty}  
		\int_{B_r} (q_{r_j} - q_0) \mathbf{1}_{\Omega_{u_{r_j}}} \, dx
		+ \int_{B_r} (q_{r_j} - q_0) \mathbf{1}_{\Omega_{u_0}} \, dx }
		\\ \leq 
		2\Lambda \lim_{j \to +\infty} r_j r \norm{\nabla q}_\infty  \abs{B_r}
		= 
		0.
	\end{multline*}
	
	Therefore, to obtain \eqref{eq:goal_convergence_new} from \eqref{eq:long_approximation}, we are left with the last two terms in \eqref{eq:long_approximation}. 
	Recalling the definition of $\widetilde{u}_{r_j}$, we easily see that $\nabla \widetilde{u}_{r_j} = (u_{r_j} - u_0) \nabla \eta + \eta \nabla u_{r_j} + (1-\eta) \nabla u_0$. Thus, expanding the products in the second-to-last term in \eqref{eq:long_approximation}, and neglecting the terms that converge to zero because $u_{r_j} \to u_0$ in $L^2(B_r)$ (for that, note that $\eta$ is smooth, $A_{r_j}$ is uniformly bounded, $u_0 \in H^1(B_r)$ and $u_{r_j}$ have a uniform bound in $H^1(B_r)$), we obtain 
	\begin{align*}
		\lim_{j \to +\infty} & \int_{B_R \setminus B_r} \left( \nabla \widetilde{u}_{r_j} \, A_{r_j} \, \nabla \widetilde{u}_{r_j}^{\miT} - \nabla u_{r_j} \, A_{r_j} \, \nabla u_{r_j}^{\miT} \right) \, dx 
		\\ & =
		\lim_{j \to +\infty} \int_{B_R \setminus B_r} \Big( 
		(\eta^2 - 1) \nabla u_{r_j} \, A_{r_j} \, \nabla u_{r_j}^{\miT} 
		+ 2\eta (1-\eta) \nabla u_{r_j} \, A_{r_j} \, \nabla u_0^{\miT}
		\\ & \qquad \qquad \qquad \qquad 
		+ (1-\eta)^2 \nabla u_0 \, A_{r_j} \, \nabla u_0^{\miT}
		\Big) \, dx
		\\ & =
		\lim_{j \to +\infty} \int_{B_R \setminus B_r} 
		(1 - \eta^2) \left( \nabla u_0 \, A_{r_j} \, \nabla u_0^{\miT} - \nabla u_{r_j} \, A_{r_j} \, \nabla u_{r_j}^{\miT} \right) \, dx
		\\ & \leq 
		\lim_{j \to +\infty} \int_{B_R \setminus B_r \cap \{\eta = 0\}} \left( \nabla u_0 \, A_{r_j} \, \nabla u_0^{\miT} - \nabla u_{r_j} \, A_{r_j} \, \nabla u_{r_j}^{\miT} \right) \, dx 
		\\ & \qquad + \lim_{j \to +\infty} \int_{\{\eta > 0\}} \nabla u_0 \, A_{r_j} \, \nabla u_0^{\miT} \, dx
		\\ & \leq 
		\lambda^{-1} \int_{\{\eta>0\}} \abs{\nabla u_0}^2 \, dx.
	\end{align*} 
	In the second equality we have used the weak convergence $\nabla u_{r_j} \rightharpoonup \nabla u_0$ in $L^2(B_r)$; in the next inequality we have used ellipticity of $A_{r_j}$ \eqref{hip:A_elliptic_bounded} to disregard the term with $u_{r_j}$ in the region where $\eta > 0$; and in the last estimate we have argued as in \eqref{eq:less_liminf_gradient_A}, and used the boundedness of $A_{r_j}$ \eqref{hip:A_elliptic_bounded}.

	In turn, the last term in \eqref{eq:long_approximation} is easier to deal with:
	\begin{multline*}
		\lim_{j \to +\infty} \int_{B_R \setminus B_r} q_{r_j} \big( \mathbf{1}_{\Omega_{\widetilde{u}_{r_j}}} - \mathbf{1}_{\Omega_{u_{r_j}}} \big) \, dx
		=
		\lim_{j \to +\infty} \int_{B_R \setminus B_r \cap \{\eta = 0\}} q_{r_j} \big( \mathbf{1}_{\Omega_{u_0}} - \mathbf{1}_{\Omega_{u_{r_j}}} \big) \, dx
		\\ + 
		\int_{\{\eta > 0\}} q_{r_j} \big( \mathbf{1}_{\Omega_{\widetilde{u}_{r_j}}} - \mathbf{1}_{\Omega_{u_{r_j}}} \big) \, dx
		\leq 
		2 \, \overline{q} \abs{\{\eta > 0\}}, 
	\end{multline*}
	simply because the first term in the second expression is non-positive by \eqref{eq:less_liminf_volume}.
	
	\subsubsection*{Step 4: Conclusion of the proof.} Plugging these estimates into \eqref{eq:long_approximation} yields 
	\begin{equation*}
		\text{LHS}
		\leq 
		\text{RHS}
		+ \lambda^{-1} \int_{\{\eta>0\}} \abs{\nabla u_0}^2 \, dx
		+ 2 \, \overline{q} \abs{\{\eta > 0\}}.
	\end{equation*}
	Now, since $\eta$ is arbitrary, this yields \eqref{eq:goal_convergence_new}, hence finishing the proof.
\end{proof}

Let us state a few interesting consequences of the convergence result we have just established:

\begin{corollary} \label{corol:blowup_minimizer}
	Assume \eqref{hip:setting} and $0 \in \pom_u$. Then, for any $R > 0$, if $u_0$ is a blow-up limit as in Lemma~\ref{lem:convergence_blowups}, then:
	\begin{enumerate}
		\item $u_0$ is stationary for $\F_{\Lambda, 0}$, i.e. $\delta \F_{\Lambda, 0}(u_0) [\xi] = 0$ for every $\xi \in C^\infty_c(\Rn; \Rn)$,
		
		\item Furthermore, $u_0$ is a local minimizer of $\F_{\Lambda, 0}$ in $B_R$, that is, for every $v \in H^1(B_R)$ so that $u_0 - v \in H^1_0(B_R)$ we have $\F_{\Lambda, 0}(u_0) \leq \F_{\Lambda, 0} (v)$.
	\end{enumerate}
\end{corollary}

The proof of these properties follows easily from the minimality of $u$ and explicit expressions for $\delta \F_{\Lambda}$ obtained in Lemma~\ref{lem:first_variation}, and all the convergences shown in Lemmas~\ref{lem:convergence_parameters} and \ref{lem:convergence_blowups}. We postpone the details to Appendix~\ref{app:proofs}.


Our next objective is to show that blow-ups are 1-homogeneous. For this, we will follow the approach in \cite[Section 9]{V}, which depends on the Weiss monotonicity formula.
In this kind of free boundary problems, the Weiss function is typically monotone in $r$ (see e.g.~\cite[Proposition 9.9]{V}). Nevertheless, in our setting, the presence of $F(x, u)$ and the coefficients $A(x)$ make it more difficult to prove this monotonicity, and in fact we will actually only prove that the derivative of the Weiss function is bounded below (but possibly negative). Fortunately, we will show this  suffices to show that blow-up limits are 1-homogeneous.

\begin{lemma} \label{lem:1-homogeneous}
	Assume \eqref{hip:setting} and $0 \in \pom_u$. Then any blow-up limit $u_0$ obtained as in Lemma~\ref{lem:convergence_blowups} is 1-homogeneous.
\end{lemma}
\begin{proof}
	As commented before, we will rely on a Weiss monotonicity formula. However, the fact that our blow-up limits $u_0$ are minimizers of energies with $A(0)$ instead of the identity, would make us need to adapt the monotonicity formula to these constant matrices. Instead of that, we prefer to change coordinates from the beginning, to find the Laplacian in the limit.	
	
	\subsubsection*{Step 1: A convenient change of variables.} The fact that $A(0)$ is symmetric and elliptic (see \eqref{hip:A_symmetric} and \eqref{hip:A_elliptic_bounded}) allows us to find a (unique) symmetric matrix $\sqrt{A(0)}$ which is still elliptic and bounded (as in \eqref{hip:A_elliptic_bounded}), which constants depending on $\lambda$.
	
	With that in mind, define $\widehat{u}_0 (x) := u_0(x \sqrt{A(0)})$. Then, Corollary~\ref{corol:blowup_minimizer} implies that $\widehat{u}_0$ is a local minimizer of the energy
	\[
	\widehat{\F}_{\Lambda, 0}(v, D) 
	:=
	\int_D \abs{\nabla v}^2 \, dx 
	+ \Lambda q(0) \abs{\Omega_v \cap D}, 
	\qquad D \subset \Rn.
	\]	
	Indeed, let $R > 0$ and $\widehat{v} \in H^1(\Rn)$ with $\widehat{v}-\widehat{u}_0 \in H^1_0(B_R)$. Then, for $v (x) := \widehat{v}(x \sqrt{A(0)}^{-1})$ (so that $\widehat{v}(x) = v(x \sqrt{A(0)})$), we have $v - u_0 \in H^1_0(B_{\lambda^{-1/2} R})$ by ellipticity of $A(0)$. Therefore, by Corollary~\ref{corol:blowup_minimizer}, $\F_{\Lambda, 0}(u_0, B_{\lambda^{-1/2} R}) \leq \F_{\Lambda, 0}(v, B_{\lambda^{-1/2} R})$, that is,
	\begin{equation*}
		\int_\Rn \nabla u_0 \, A(0) \, \nabla u_0^{\miT} \, dx
		+ \Lambda q(0) \abs{\Omega_{u_0} \cap B_{\lambda^{1/2}R}}
		\leq 
		\int_\Rn \nabla v \, A(0) \, \nabla v^{\miT} \, dx
		+ \Lambda q(0) \abs{\Omega_v \cap B_{\lambda^{1/2}R}},
	\end{equation*}
	and the change of variables $x = y \sqrt{A(0)}$ yields (after canceling out the constant Jacobian)
	\begin{equation*}
		\int_\Rn \nabla \widehat{u}_0 \, \nabla \widehat{u}_0^{\miT} \, dy
		+ \Lambda q(0) \abs{\Omega_{\widehat{u}_0} \cap B_R}
		\leq 
		\int_\Rn \nabla \widehat{v} \, \nabla \widehat{v}^{\miT} \, dx
		+ \Lambda q(0) \abs{\Omega_v \cap B_R}, 
	\end{equation*}
	i.e. $\widehat{\F}_{\Lambda, 0}(\widehat{u}_0, B_R) \leq \widehat{\F}_{\Lambda, 0}(\widehat{v}, B_R)$.
	
	Furthermore, $\widehat{u}_0$ is a stationary point for $\widehat{\F}_{\Lambda, 0}$. 
	Take any $\widehat{\xi} \in C^\infty_c(\Rn; \Rn)$, and define $\xi (x) := \sqrt{A(0)} \, \widehat{\xi} (x \sqrt{A(0)}^{-1}) \in C^\infty_c(\Rn; \Rn)$ (so that $\widehat{\xi} (x) = \sqrt{A(0)}^{-1} \xi (x \sqrt{A(0)})$). A change of variables and Corollary~\ref{corol:blowup_minimizer} show that indeed 
	\begin{multline} \label{eq:change_vars_stationary}
		\delta \widehat{\F}_{\Lambda, 0} (\widehat{u}_0)[\widehat{\xi}]
		=
		\int_\Rn \Big( -2 \nabla \widehat{u}(y) \, D\widehat{\xi}(y) \, \nabla \widehat{u}(y)^{\miT} \, dy 
		+ \nabla \widehat{u}(y) \, \nabla \widehat{u}(y)^{\miT}  \div \widehat{\xi}(y) \Big) \, dy
		\\ \qquad 
		+ \Lambda q(0) \int_\Rn \mathbf{1}_{\Omega_{\widehat{u}}}(y) \div \widehat{\xi}(y) \, dy
		=
		\abs{\text{det}(A(0))}^{-1/2} \delta \F_{\Lambda, 0} (u_0)[\xi]
		=
		0.
	\end{multline} 
	What we use in the change of variable is that
	\begin{equation*}
		\nabla \widehat{u}(y) = \nabla u(y \sqrt{A(0)}) \sqrt{A(0)}, 
		\qquad 
		\mathbf{1}_{\Omega_{\widehat{u}}}(y) = \mathbf{1}_{\Omega_u}(y \sqrt{A(0)}), 
	\end{equation*}
	\begin{equation*}
		D\widehat{\xi}(y)
		=
		\sqrt{A(0)}^{-1} (D\xi)(y \sqrt{A(0)}) \, \sqrt{A(0)}, 
		\qquad 
		\div \widehat{\xi} (y) 
		=
		(\div \xi)(y \sqrt{A(0)}).
	\end{equation*}

	To carry out the whole blow-up procedure later, we set
	\begin{equation*}
		\widehat{u}(x) := u(x \sqrt{A(0)}), 
		\qquad 
		\widehat{A}(x) := \sqrt{A(0)}^{-1} A(x\sqrt{A(0)}) \, \sqrt{A(0)}^{-1}, 
	\end{equation*}
	\begin{equation*}
		\widehat{\F}(x, u) := F(x\sqrt{A(0)}, u), 
		\qquad 
		\widehat{q}(x) := q(x\sqrt{A(0)}).
	\end{equation*}
	Then, Lemma~\ref{lem:first_variation} and a change of coordinates show that $\widehat{u}$ is stationary for 
	\[
	\widehat{\F}_{\Lambda}(v) := \int_\Rn \nabla v \, \widehat{A} \, \nabla v^{\miT} dx - 2 \int_\Rn \widehat{F}(x, v(x)) \, dx + \Lambda \Vol_{\widehat{q}}(\Omega_v).
	\]
	The change of variables is similar to the one in \eqref{eq:change_vars_stationary}, taking also into account that
	\[
	\nabla \widehat{F}(y, u)
	=
	(\nabla F)(y \sqrt{A(0)}, u) \sqrt{A(0)}, 
	\quad 
	\nabla \widehat{q}(y) 
	=
	(\nabla q)(y\sqrt{A(0)}) \sqrt{A(0)}.
	\]

	Finally, in a similar way, from Corollary~\ref{corol:EL2} we infer that $-\div(\widehat{A}\nabla \widehat{u}) = \widehat{F}'(x, \widehat{u}(x))$ in $\Omega_{\widehat{u}}$.
	
	\subsubsection*{Step 2: Equipartition of the energy.} 
	Define a modified Weiss function as
	\begin{equation*}
		W_\Lambda(v)
		:=
		\int_{B_1} \abs{\nabla v}^2 \, dx
		- \int_{\partial B_1} v^2 \, d\mathcal{H}^{n-1}
		+ \Lambda q(0) \abs{\Omega_v \cap B_1}.
	\end{equation*}
	Let $\widehat{z}_r(x) := \abs{x} \widehat{u}_r(x / \abs{x})$  be the 1-homogeneous extension of $\widehat{u}_r$ from $\partial B_1$ to $B_1$.  Our goal in this step is to show that there exists $C > 0$, uniform on \eqref{hip:setting}, such that it holds
	\begin{equation} \label{eq:equipartition}
		W_{\Lambda}(\widehat{z}_r) - W_\Lambda(\widehat{u}_r) 
		\geq 
		\frac1n \int_{\partial B_1} \abs{x \cdot \nabla \widehat{u}_r - \widehat{u}_r}^2 \, d\mathcal{H}^{n-1} 
		- Cr.
	\end{equation} 

	For this purpose, we follow \cite[Lemma 9.8]{V}, substantially extending it to introduce $A(x), q(x)$ and $F(x, u)$. First, testing the stationarity of $\widehat{u}$ with respect to $\widehat{\F}_\Lambda$ (see Lemma~\ref{lem:first_variation}) against the same vector field $\xi_\varepsilon$ as in \cite[Lemma 9.8]{V}, and letting $\varepsilon \to 0$, we obtain
	\begin{align} \label{eq:V9.9}
		0
		& =
		-2 \int_{B_r} \nabla \widehat{u} \, \widehat{A} \, \nabla \widehat{u}^{\miT} \, dx
		+ n \int_{B_r} \nabla \widehat{u} \, \widehat{A} \, \nabla \widehat{u}^{\miT} \, dx
		+ \int_{B_r} \nabla \widehat{u} \, (\nabla \widehat{A} \cdot x) \, \nabla \widehat{u}^{\miT} \, dx
		\nonumber
		\\ & \quad  
		- 2 \int_{B_r} \nabla \widehat{F}(x, \widehat{u}) \cdot x \, dx
		- 2n \int_{B_r} \widehat{F}(x, \widehat{u}) \, dx
		+ \Lambda \int_{B_r} \mathbf{1}_{\Omega_{\widehat{u}}} \nabla \widehat{q} \cdot x \, dx
		+ n \, \Lambda \int_{B_r} \mathbf{1}_{\Omega_{\widehat{u}}} \widehat{q} \, dx
		\nonumber 
		\\ & \quad 
		+ 2 r \int_{\partial B_r} (\nabla \widehat{u} \, x^{\miT})(\nabla \widehat{u} \, \widehat{A} \, x^{\miT}) \, d\mathcal{H}^{n-1}
		- r \int_{\partial B_r} \nabla \widehat{u} \, \widehat{A} \, \nabla \widehat{u}^{\miT} \, d\mathcal{H}^{n-1}
		\\ & \quad 
		+ 2r \int_{\partial B_r} \widehat{F}(x, \widehat{u}) \, dx 
		- \Lambda \, r \int_{\partial B_r} \mathbf{1}_{\Omega_{\widehat{u}}} \widehat{q} \, d\mathcal{H}^{n-1}.
		\nonumber
	\end{align}
Also, the fact that $\div(\widehat{A}\nabla \widehat{u}) = \widehat{F}'(x, \widehat{u})$ in $\Omega_{\widehat{u}}$ (see Step 1) yields
	\begin{multline*}
		\int_{B_r} \nabla \widehat{u} \, \widehat{A} \, \nabla \widehat{u}^{\miT} \, dx
		=
		\int_{B_r} \div(\widehat{u} \, \widehat{A} \, \nabla \widehat{u}) \, dx
		- \int_{B_r} \widehat{u} \, \div(\widehat{A} \, \nabla \widehat{u}) \, dx
		\\ =
		\int_{\partial B_r} \widehat{u} \, \nabla \widehat{u} \, \widehat{A} \, \frac{x^{\miT}}{\abs{x}} \, d\mathcal{H}^{n-1}
		+ 2 \int_{B_r} \widehat{u} \, \widehat{F}'(x, \widehat{u}) \, dx.
	\end{multline*}
	Inserting this computation into the first term of \eqref{eq:V9.9}, then rescaling everything to $B_1$, we obtain
	\begin{align} \label{eq:V9.9_rescaled}
		0
		& =
		2 \int_{B_1} \nabla \widehat{u}_r \, \widehat{F}'_r(x, \widehat{u}_r) \, dx
		+ n \int_{B_1} \nabla \widehat{u}_r \, \widehat{A}_r \, \nabla \widehat{u}_r^{\miT} \, dx
		+ \int_{B_1} \nabla \widehat{u}_r \, (\nabla \widehat{A}_r \cdot x) \, \nabla \widehat{u}_r^{\miT} \, dx
		\nonumber
		\\ & \quad  
		- 2 \int_{B_1} \nabla \widehat{F}_r(x, \widehat{u}_r) \cdot x \, dx
		- 2n \int_{B_1} \widehat{F}_r(x, \widehat{u}_r) \, dx
		+ \Lambda \int_{B_1} \mathbf{1}_{\Omega_{\widehat{u}_r}} \nabla \widehat{q}_r \cdot x \, dx
		+ n \Lambda \int_{B_1} \mathbf{1}_{\Omega_{\widehat{u}_r}} \widehat{q}_r \, dx
		\nonumber 
		\\ & \quad 
		- 2 \int_{\partial B_1} \widehat{u}_r \, \nabla \widehat{u}_r \, \widehat{A}_r \, x^{\miT} \, d\mathcal{H}^{n-1}
		+ 2 \int_{\partial B_1} (\nabla \widehat{u}_r \, x^{\miT}) \, (\nabla \widehat{u}_r \, \widehat{A}_r \, x^{\miT}) \, d\mathcal{H}^{n-1}
		\\ & \quad 
		- \int_{\partial B_1} \nabla \widehat{u}_r \, \widehat{A}_r \, \nabla \widehat{u}_r^{\miT} \, d\mathcal{H}^{n-1}
		+ 2 \int_{\partial B_1} \widehat{F}_r(x, \widehat{u}_r) \, dx 
		- \Lambda \int_{\partial B_1} \mathbf{1}_{\Omega_{\widehat{u}_r}} \widehat{q}_r \, d\mathcal{H}^{n-1}.
		\nonumber
	\end{align}
	
	With this, we find
		\begin{align} \label{eq:diff_Ws}
		W_{\Lambda}(\widehat{z}_r) &- W_\Lambda(\widehat{u}_r) 
		=
		\int_{B_1} \abs{\nabla \widehat{z}_r}^2 \, dx
		+ \Lambda \, q(0) \abs{\Omega_{\widehat{z}_r} \cap B_1}
		- \int_{B_1} \abs{\nabla \widehat{u}_r}^2 \, dx
		- \Lambda \, q(0) \abs{\Omega_{\widehat{u}_r} \cap B_1}
		\nonumber
		\\ & =
		\frac1n \int_{\partial B_1} \left( \abs{\nabla \widehat{u}_r}^2 - \abs{\nabla \widehat{u}_r \cdot x}^2 \right) \, d\mathcal{H}^{n-1}
		+ \frac1n \int_{\partial B_1} \widehat{u}_r^2 \, d\mathcal{H}^{n-1}
		\nonumber
		\\ & \qquad + \frac{\Lambda}{n} \, q(0) \, \mathcal{H}^{n-1}(\Omega_{\widehat{u}_r} \cap \partial B_1) 
		\nonumber
		 - \int_{B_1} \abs{\nabla \widehat{u}_r}^2 \, dx
		- \Lambda \, q(0) \abs{\Omega_{\widehat{u}_r} \cap B_1}
		\nonumber
		\\ & =
		\frac1n \int_{\partial B_1} \abs{\nabla \widehat{u}_r \cdot x - \widehat{u}_r}^2 \, d\mathcal{H}^{n-1}
		+ \frac1n \int_{\partial B_1} \left( \abs{\nabla \widehat{u}_r}^2 - 2 \abs{\nabla \widehat{u}_r \cdot x}^2 \right) \, d\mathcal{H}^{n-1}
		\nonumber
		\\ & \qquad + \frac2n \int_{\partial B_1} \widehat{u}_r \nabla \widehat{u}_r \cdot x \, d\mathcal{H}^{n-1}
		+ \frac{\Lambda}{n} \, q(0) \, \mathcal{H}^{n-1}(\Omega_{\widehat{u}_r} \cap \partial B_1) 
		\\ & \qquad 
		- \int_{B_1} \abs{\nabla \widehat{u}_r}^2 \, dx
		- \Lambda \, q(0) \abs{\Omega_{\widehat{u}_r} \cap B_1}
		\nonumber
		\\ & =
		\frac1n \int_{\partial B_1} \abs{\nabla \widehat{u}_r \cdot x - \widehat{u}_r}^2 \, d\mathcal{H}^{n-1}
		+ \sum_{j = 1}^{10} \text{Error}_j.
		\nonumber
	\end{align}
	Here we have used: in the first step, the fact that $\widehat{u}_r \equiv \widehat{z}_r$ on $\partial B_1$; in the second step, explicit expressions for the terms involving $\widehat{z}_r$ by integrations in polar coordinates (just as in \cite[(9.5)]{V}); in the third step, we have simply inserted and removed $\frac1n \int_{\partial B_1} \abs{\nabla \widehat{u}_r \cdot x - \widehat{u}_r}^2 \, d\mathcal{H}^{n-1}$. In the last step, we have introduced all the terms in \eqref{eq:V9.9_rescaled} (divided by $n$), and rearranged everything into the following error terms:
	\begin{align*}
		& \text{Error}_1 
		:=
		\frac{1}{n} \int_{\partial B_1} \left( \abs{\nabla \widehat{u}_r}^2 - \nabla \widehat{u}_r \, \widehat{A}_r \, \nabla \widehat{u}_r^{\miT} \right) \, d\mathcal{H}^{n-1}, 
		\\ & 
		\text{Error}_2
		:=
		\frac{-2}{n} \int_{\partial B_1} \left( \abs{\nabla \widehat{u}_r \cdot x}^2 - (\nabla \widehat{u}_r \, x^{\miT}) \, (\nabla \widehat{u}_r \, \widehat{A}_r \, x^{\miT}) \right) \, d\mathcal{H}^{n-1},
		\\ & 
		\text{Error}_3
		:=
		\frac{2}{n} \int_{\partial B_1} \left( \widehat{u}_r \nabla \widehat{u}_r \cdot x - \widehat{u}_r \, \nabla \widehat{u}_r \, \widehat{A}_r \, x^{\miT} \right) \, d\mathcal{H}^{n-1},
		\\ & 
		\text{Error}_4
		:=
		\frac{\Lambda}{n} \left( q(0) \, \mathcal{H}^{n-1}(\Omega_{\widehat{u}_r} \cap B_1) - \int_{\partial B_1} \mathbf{1}_{\Omega_{\widehat{u}_r}} \widehat{q}_r \, d\mathcal{H}^{n-1} \right),
		\\ & 
		\text{Error}_5 
		:=
		- \int_{B_1} \left( \abs{\nabla \widehat{u}_r}^2 - \nabla \widehat{u}_r \, \widehat{A}_r \, \nabla \widehat{u}_r^{\miT} \right) \, dx, 
		\quad
		\text{Error}_6
		:= \!
		- \Lambda \bigg( q(0) \abs{\Omega_{\widehat{u}_r} \cap B_1} - \! \int_{B_1} \!\!\! \mathbf{1}_{\Omega_{\widehat{u}_r}} \widehat{q}_r \, dx	\bigg)	,
		\\ & 
		\text{Error}_7
		:=
		\frac{1}{n} \int_{B_1} \nabla \widehat{u}_r \, (\nabla \widehat{A}_r \cdot x) \, \nabla \widehat{u}_r^{\miT} \, dx,
		\quad \qquad \;\;
		\text{Error}_8
		:=
		- \frac{2}{n} \int_{B_1} \nabla \widehat{F}_r(x, \widehat{u}_r) \cdot x \, dx,
		\\ & 
		\text{Error}_9
		:=
		- 2 \int_{B_1} \widehat{F}_r(x, \widehat{u}_r) \, dx,
		\quad \qquad \qquad \qquad \;\;
		\text{Error}_{10}
		:=
		\frac{\Lambda}{n} \int_{B_1} \mathbf{1}_{\Omega_{\widehat{u}_r}} \nabla \widehat{q}_r \cdot x \, dx.
	\end{align*}
	
	It is now easy to see that $\abs{\text{Error}_j} = O(r)$ for every $1 \leq j \leq 10$, using the estimates in \eqref{eq:rescaling_bounds_u} and Lemma~\ref{lem:convergence_parameters} (which clearly also hold for $\widehat{A}, \widehat{F}$ and $\widehat{q}$ because there has only been a linear change of variables by the matrix $\sqrt{A(0)}$, which is elliptic and bounded \eqref{hip:A_elliptic_bounded}). Moreover, $\widehat{u}_r$ are uniformly bounded in $H^1(B_1)$ and $L^\infty$ by \eqref{eq:rescaling_bounds_u}. More precisely,
	\begin{align*}
		&
		\abs{\text{Error}_1}, \; \abs{\text{Error}_2}, \; \abs{\text{Error}_3}, \; \abs{\text{Error}_5}
		\lesssim 
		\|I - \widehat{A}_r\|_\infty \, L^2
		=
		O(r), 
		\\ & 
		\abs{\text{Error}_4}, \; \abs{\text{Error}_6}
		\lesssim 
		\norm{q(0) - \widehat{q}_r}_\infty 
		= O(r),
		\\ & 
		\abs{\text{Error}_7} 
		\lesssim 
		r \, \|\nabla \widehat{A}\|_\infty \, L^2 
		=
		O(r),
		\qquad \; 
		\abs{\text{Error}_8} 
		\lesssim 
		r \, \|\nabla \widehat{F}\|_\infty 
		=
		O(r),		
		\\ & 
		\abs{\text{Error}_9}
		\lesssim 
		\int_{B_1} u(rx) \, dx
		=
		O(r),
		\qquad \;
		\abs{\text{Error}_{10}} 
		\lesssim 
		r \norm{\nabla \widehat{q}}_\infty
		=
		O(r).
	\end{align*}
	Inserting all these estimates back in \eqref{eq:diff_Ws} yields \eqref{eq:equipartition}, as claimed.

	\subsubsection*{Step 3: Blow-up procedure.} 
	Take any sequence $u_{r_j}$ (with $r_j \to 0$) converging to a blow-up limit $v := u_0$ as in Lemma~\ref{lem:convergence_blowups}. This convergence takes place in $H^1(B_1)$, where we are assuming that $R > 1$ in Lemma~\ref{lem:convergence_blowups} for the ease of notation. 
	
	It is clear that if we define $\widehat{u}(x) := u(x\sqrt{A(0)})$ and denote  its rescalings  $\widehat{u}_r$, and if we take $\widehat{v}(x) := u_0(x \sqrt{A(0)}) = \widehat{u}_0(x)$ as in Step 1, the same convergences take place (because we have only made a linear change of variables).

	We claim that 
	\begin{equation} \label{claim:limit_W}
		\ell := \lim_{r \to 0} W_\Lambda( \widehat{u}_r) \in \R
	\end{equation} 
	exists.
	Indeed, $W_\Lambda(\widehat{u}_r) \geq 
		- \int_{\partial B_1} \widehat{u}_r(x)^2 \, d\mathcal{H}^{n-1}
		\gtrsim 
		- L^2$
	by \eqref{eq:rescaling_bounds_u}. Here, note that \cite[Lemmas 9.1 and 9.2]{V} show that $W_\Lambda(\widehat{u}_r)$ is continuous and a.e.\ differentiable.	
	Moreover, \cite[Lemma 9.2]{V} also gives an expression for its derivative, in terms of $\widehat{z}_r$ from Step 2, namely
	\begin{equation} \label{eq:derivative_Weiss}
		\frac{\partial}{\partial r} W_\Lambda(\widehat{u}_r)
		=
		\frac{n}{r} \Big(W_\Lambda(\widehat{z}_r) - W_\Lambda(\widehat{u}_r)\Big) 
		+ \frac1r \int_{\partial B_1} \abs{x \cdot \nabla \widehat{u}_r - \widehat{u}_r}^2 \, d\mathcal{H}^{n-1},
	\end{equation}
	so that using \eqref{eq:equipartition} we obtain 
	\begin{equation*}
		\frac{\partial}{\partial r} W_\Lambda(\widehat{u}_r)
		\geq 
		- C.
	\end{equation*}
	
	It is elementary that a continuous and a.e.\ differentiable function which is bounded below in $(0, 1)$ and whose derivative is also bounded below in $(0, 1)$, like $W_\Lambda(\widehat{u}_r)$, indeed has a finite limit at 0. (Roughly speaking, if the limit did not exist, the function should oscillate ``infinitely rapidly'', which is prevented the lower bound on the derivative.) Thus we obtain~\eqref{claim:limit_W}, as claimed. 

	From the existence of the limit of $W_\Lambda(\widehat{u}_r)$ at $r=0$, the 1-homogeneity of blow-up limits is  standard (see e.g. \cite[Lemma 9.10]{V}). Indeed, we have the  rescaling property
	\begin{equation} \label{claim:limit_W_rescaled}
		\lim_{j \to \infty} W_{\Lambda} (\widehat{u}_{sr_j})
		=
		W_\Lambda(\widehat{v}_s),
		\qquad 
		\text{ for every $s \in (0, 1)$}
	\end{equation}
	because, as $\widehat{u}_{r_j} \to \widehat{v}$ in $H^1(B_1)$,
		\begin{equation*}
		\int_{B_1} \abs{\nabla \widehat{u}_{sr_j}(x)}^2 dx 
		=
		\int_{B_1} \abs{\nabla \left(\frac{\widehat{u}_{r_j}(sx)}{s} \right)}^2 dx 
		\underset{j \to \infty}{\longrightarrow}
		\int_{B_1} \abs{\nabla \left(\frac{\widehat{v}(sx)}{s}\right)}^2 dx
		=
		\int_{B_1} \abs{\nabla \widehat{v}_s}^2 dx.
	\end{equation*}	
	The terms leading to $\int_{\partial B_1} \widehat{v}_s^2 \, d\mathcal{H}^{n-1}$ and $\abs{\Omega_{\widehat{v}_s} \cap B_1}$ can be handled similarly, using the different modes of convergence in Lemma~\ref{lem:convergence_blowups}.
	
	Then, \eqref{claim:limit_W} and \eqref{claim:limit_W_rescaled} yield
	\begin{equation*}
		\frac{\partial}{\partial s} W_\Lambda(\widehat{v}_s) = 0, 
		\qquad 
		s \in (0, 1).
	\end{equation*}
	Inserting this back in our formula \eqref{eq:derivative_Weiss} for the derivative of $W_\Lambda$, and noting that, as discussed in Step 1, $\widehat{v}$ is a minimizer of $\widehat{\F}_{\Lambda, 0}$ (which coincides with $W_\Lambda$ up to the term supported on $\partial B_1$), we infer that 
	\begin{equation*} 
		\int_{\partial B_1} \abs{x \cdot \nabla \widehat{v}_s - \widehat{v}_s}^2 \, d \mathcal{H}^{n-1} = 0,
		\qquad 
		s \in (0, 1).
	\end{equation*}
	It is well known that this implies the 1-homogeneity of $\widehat{v}$ (see e.g.~\cite[Lemma 9.3]{V}), which clearly implies the 1-homogeneity of $v := u_0$.
\end{proof}


\subsection{Overdetermined boundary condition}

To establish the Neumann boundary condition on~$\partial\Omega_u$, let us first establish some more properties of blow-ups. Essentially, the idea is again that $F$ somehow disappears in the limit, and that blow-ups inherit many properties from the minimizer $u$ because of the convergences shown in Lemma~\ref{lem:convergence_blowups}.

\begin{lemma}[PDE for blow-up limits] \label{lem:blowups_harmonic}
	Assume \eqref{hip:setting}. Let $u_0$ be a blow-up limit for $u$ at $0 \in \pom_u$. Then $\Omega_{u_0}$ is open and $-\div (A_0 \nabla u_0) = 0$ in $\Omega_{u_0}$.
\end{lemma}

\begin{lemma}[Non-degeneracy of blow-up limits] \label{lem:blowups_not_zero}
	Assume \eqref{hip:setting}. Let $u_0$ be a blow-up limit for $u$ at $0 \in \pom_u$. Then $u_0 \not \equiv 0$.
\end{lemma}

\begin{lemma}[Upper densities] \label{lem:density_upper_bound}
	Assume \eqref{hip:setting}. Then $\displaystyle \lim_{r \to 0} \dfrac{\abs{\Omega_u \cap B_r(x_0)}}{\abs{B_r}} < 1$ for any $x_0 \in \pom_u$.
\end{lemma}

We defer the complete proofs to Appendix~\ref{app:proofs}, but let us quickly comment that the proof of Lemma~\ref{lem:blowups_harmonic} follows as in Corollary~\ref{corol:EL2} from the minimality obtained in Corollary~\ref{corol:blowup_minimizer}, Lemma~\ref{lem:blowups_not_zero} is a consequence of the non-degeneracy property from Proposition~\ref{prop:non_degeneracy}, and Lemma~\ref{lem:density_upper_bound} combines this non-degeneracy property with the fact that $u$ is Lipschitz (from Proposition~\ref{prop:lipschitz}).

We are ready to finish, obtaining the overdetermined condition as a consequence of a classification of the possible blow-up limits. In general, it will only hold in the viscosity sense:

\begin{definition}[Viscosity sense] \label{def:viscosity}
	Given $B: \Rn \to \R$, we say that $\nabla u(x) \, A(x) \, \nabla u(x)^{\miT} \leq B(x)$ in the viscosity sense at $x_0 \in \pom_u$ if whenever $\varphi \in C^\infty(\Rn)$ touches $u$ from below at $x_0$ (i.e. $\varphi(x_0) = u(x_0)$ and $\varphi \leq u$ in a neighborhood of $x_0$), we have $\nabla \varphi(x_0) \, A(x_0) \, \nabla \varphi(x_0)^{\miT} \leq B(x_0)$. In turn, $\nabla u(x) \, A(x) \, \nabla u(x)^{\miT} \geq B(x)$ is defined similarly. One says $\nabla u(x) \, A(x) \, \nabla u(x)^{\miT} = B(x)$ in the {\em viscosity sense}\/ at a point~$x_0\in\partial\Omega_u$ when both conditions hold.
\end{definition}

\begin{proposition}[Overdetermined boundary condition] \label{prop:overdetermined}
	Assume \eqref{hip:setting}. Then $$\nabla u(x) \, A(x) \, \nabla u(x)^{\miT} = \Lambda \, q(x)$$ in the viscosity sense for all  $x\in \pom_u$. 
\end{proposition}

To establish the result at, say, $0 \in \pom_u$, one may again just change coordinates using $\sqrt{A(0)}$, and then follow the usual strategy as for constant coefficient operators, which can be seen e.g. in \cite[Proposition 9.18]{V}. The idea is simple: any test function $\phi$ as in Definition~\ref{def:viscosity} becomes linear after blowing up it (because it is chosen to be smooth), and then one can use this linear function and the fact that $\phi$ is above/below $u$ to categorize the possible blow-ups $u_0$, using that they are 1-homogeneous by Lemma~\ref{lem:1-homogeneous} and harmonic by Lemma~\ref{lem:blowups_harmonic}. For full details, see Appendix~\ref{app:proofs}.


\subsection{Decomposition and properties of the free boundary}

As is customary in elliptic free boundary problems, the boundary $\pom_u$ will consist of a part which is regular, and a singular part that is small in some sense (in our setting, in the sense of Hausdorff dimension). 

To make this precise, we will only need minor modifications to the classical approach, as presented  in \cite[Sections 6--10]{V}. This is because the improvement of flatness for rather general PDE (the essential step when showing regularity of the free boundary) was already developed by De Silva, Ferrari and Salsa in \cite{DFS2}, extending the strategy of De Silva in~\cite{DS}. 

We start by the definition of the regular and singular parts of the boundary. We have to make a modification: as already seen in Proposition~\ref{prop:overdetermined}, we expect the blow-ups limits after the change of variables to have the shape $\widehat{u}_0(x) = (x \cdot \widehat \nu)_+$ for some unit vector $\widehat \nu \in \partial B_1$. Therefore, before the change of variables, these blow-up limits should be 
\begin{equation*}
	u_0(x) = \widehat{u}_0 \big(x\sqrt{A(0)}^{-1}\big) = \big(x \cdot \sqrt{A(0)}^{-1}\widehat \nu\big)_+ =: (x \cdot \nu)_+,
\end{equation*}
where now $\abs{\nu} \in [\lambda^{1/2}, \lambda^{-1/2}]$ by \eqref{hip:A_elliptic_bounded}. Since the distortion is always uniformly controlled (depending on the constant $\lambda$), the subsequent arguments will work without major modifications.

\begin{definition}[Regular and singular parts of the boundary]
	We define the {\em regular part}\/ of the boundary, $\Reg(\pom_u)$, as the set of points $x_0 \in \pom_u$ such that there exists some blow-up limit $u_0$ at $x_0$ taking the form  $$u_0(x) = h_{\nu, x_0}(x) := \sqrt{\Lambda q(x_0)} \, ((x-x_0) \cdot \nu)_+$$ for some vector $\nu \in \Rn$ with $\abs{\nu} \in [\lambda^{1/2}, \lambda^{-1/2}]$. The {\em singular part}\/ of the boundary is $\Sing(\pom_u) := \pom_u \setminus \Reg(\pom_u)$.
\end{definition}

The smoothness of the regular part is proven using robust techniques (see e.g.~\cite[Corollary 8.2]{V}) that need essentially no modifications in our setting:

\begin{proposition}[Regularity of $\Reg(\pom_u)$] \label{prop:regular}
	Assume \eqref{hip:setting}. Then, there exists $r_0 > 0$ so that, for every $x_0 \in \Reg(\pom_u)$, $\pom_u \cap B_{r_0}(x_0)$ is a $C^{1, \alpha}$ regular manifold, for some $\alpha > 0$. Furthermore, if $F, A$ and $q$ are smooth, so is $\pom_u$ in a (possibly smaller) ball around $x_0$.
\end{proposition}
\begin{proof} 
	Up to a translation and dilation (with constants depending on $\Lambda$ and $\underline{q}, \overline{q}$), we may assume $x_0 = 0$ and $u_0(x) = (x \cdot \nu)_+$ for some $\nu \in B_{\lambda^{-1/2}} \setminus B_{\lambda^{1/2}}$. Then, given $\delta > 0$, by Lemma~\ref{lem:convergence_blowups}, it holds $\| u_{r_j} - u_0\|_{L^\infty (B_1)} < \delta$ for $j$ large enough, from where it easily follows that $\{u_{r_j} = 0\} \cap B_1 \subset \{x \cdot \nu \leq \delta\} \cap B_1$.
	
	To obtain, $\{x \cdot \nu \leq -\delta\} \cap B_1 \subset \{u_{r_j} = 0\} \cap B_1$, we first use \cite[Lemma 6.5]{V} to obtain that $\overline{\Omega_{u_{r_j}}} \to \overline{\Omega_{u_0}} = \{x \cdot \nu \geq 0\}$ locally Hausdorff in $B_2$. This can be applied here because of the convergence in Lemma~\ref{lem:convergence_blowups}, the Lipschitz continuity in Proposition~\ref{prop:lipschitz}, and the non-degeneracy in Proposition~\ref{prop:non_degeneracy} and Lemma~\ref{lem:blowups_not_zero}. Then, pick $x \in \{x \cdot \nu \leq -\delta\} \cap B_1$, which implies that $\dist(x, \overline{\Omega_{u_0}}) \geq \delta \abs{\nu} \geq \delta \lambda^{1/2}$. Combining this with the fact that, for $j$ large enough, $\dist(\overline{\Omega_{u_{r_j}}}, \overline{\Omega_{u_0}}) \leq \delta \lambda^{1/2}/2$ by the Hausdorff convergence, we get $\dist(x, \overline{\Omega_{u_{r_j}}}) \geq \delta \lambda^{1/2}/2$, whence $u_{r_j}(x) = 0$. 
	
	Overall, we have obtained that, for any $\delta > 0$, if we consider $j \gg 1$,
	\begin{equation*}
		\{ x \cdot \nu \leq -\delta \} \cap B_1 
		\; \subset \;  
		\{u_{r_j} = 0\} \cap B_1 
		\; \subset \;  
		\{ x \cdot \nu \leq \delta \} \cap B_1.
	\end{equation*}
	A simple rescaling yields
	\begin{equation*}
		\{ x \cdot \nu \leq -\delta r_j \} \cap B_{r_j}
		\; \subset \; 
		\{u = 0\} \cap B_{r_j}
		\; \subset \; 
		\{ x \cdot \nu \leq \delta r_j \} \cap B_{r_j}.
	\end{equation*}
	Hence, denoting $\widehat{\nu} := \nu / \abs{\nu} \in \partial B_1$, and recalling that $\abs{\nu} \in [\lambda^{1/2}, \lambda^{-1/2}]$, this yields
	\begin{equation*}
		\{ x \cdot \widehat \nu \leq -\delta \lambda^{-1/2} r_j \} \cap B_{r_j}
		\; \subset \; 
		\{u = 0\} \cap B_{r_j}
		\; \subset \; 
		\{ x \cdot \widehat \nu \leq \delta \lambda^{-1/2}  r_j \} \cap B_{r_j}.
	\end{equation*}
	
	Since $\cL  u = F'(x, u(x)) \in L^\infty(\Omega_u\cap B_1)$ (because it is controlled by $M_1$) and we have also shown the overdetermined boundary condition $|\nabla u \sqrt{A}|^2 = \Lambda q$ in a viscosity sense in Lemma~\ref{prop:overdetermined}, applying \cite[Theorem 1.4]{DFS2}, we obtain the desired $C^{1, \alpha}$ regularity in $B_{r_j/2}$ for $j \gg 1$. 
	
	More precisely, one should divide by $\Lambda q(x)$ to obtain the equations 
	\begin{equation*} 
		\begin{cases}
			- \div \left( \dfrac{A(x)}{\Lambda q(x)} \nabla u(x) \right) 
			=
			\dfrac{F'(x, u(x))}{\Lambda q(x)} - \dfrac{\nabla u(x) A(x) \nabla q(x)^{\miT}}{\Lambda q(x)^2} \quad \text{ in $\Omega_u \cap B_1$,} \\ 
			\nabla u(x) \dfrac{A(x)}{\Lambda q(x)} \nabla u(x)^{\miT} 
			=
			1 
			\quad \text{ on $\partial \Omega_u \cap B_1$},
		\end{cases}
	\end{equation*}
	where the right-hand-side of the first equation is uniformly bounded by \eqref{hip:negative}, \eqref{eq:def_M_1}, \eqref{hip:q_regular}, \eqref{hip:q_positive}, \eqref{hip:A_elliptic_bounded} and Proposition~\ref{prop:lipschitz}; and where the coefficients $A(x) / (\Lambda q(x))$ are Lipschitz, symmetric, elliptic and bounded because of \eqref{hip:A_regular}, \eqref{hip:A_symmetric}, \eqref{hip:A_elliptic_bounded}, \eqref{hip:q_regular} and \eqref{hip:q_positive}. \cite[Theorem 1.4]{DFS2} applies directly to this problem. 
	
	Lastly, let us derive higher regularity of $\Reg(\pom_u)$ whenever $F, A$ and $q$ are smooth. We would like to simply invoke the classical Kinderlehrer--Nirenberg self-improvement result \cite[Theorem 2]{KN}, but that requires $u \in C^2(\overline{\Omega_u} \cap B_{r})$ for some small $r$, and we have not yet established that. Therefore, we prefer to follow another approach. 
	
	Start by using elliptic regularity to derive that $u \in C^{1, \alpha}(\overline{\Omega_u} \cap B_{r_0/2})$ (see e.g. \cite[Corollary 8.36]{GT}). This allows us to follow the first part of the proof of \cite[Theorem 1.31]{LZ} to show that $u \in C^{2, \widetilde{\alpha}}(\overline{\Omega_u} \cap B_r)$ and $\pom_u \cap B_r \in C^{2, \widetilde{\alpha}}$ for some $\widetilde{\alpha} > 0$, if $r > 0$ is small enough. Indeed, adapting the proof of \cite[Theorem 1.31]{LZ} to our setting is possible because it is essentially a repeated application of \cite[Theorem 1.24]{LZ}, which is very versatile and adapts to our PDE, as explained in \cite[Remark 1.30]{LZ}. Note also that the oblique derivative condition over the boundary in \cite[Theorem 1.24]{LZ} is just our overdetermined boundary condition.
	
	Now that $\pom_u \cap B_r$ is $C^{2, \widetilde{\alpha}}$, classical elliptic regularity yields $u \in C^{2, \widetilde{\alpha}}(\overline{\Omega_u} \cap B_{r/2})$ (see e.g. \cite[Theorem 6.19]{GT}). And then we can invoke \cite[Theorem 2]{KN} to infer that $\pom_u$ is actually smooth in a neighborhood of the origin, hence finishing the proof. (Alternatively, one could have continued adapting the proof in \cite[Theorem 1.31]{LZ} by iterating \cite[Theorem 1.29]{LZ}, which is again possible in our setting thanks to \cite[Remark 1.30]{LZ}).
\end{proof}

We are now ready to derive the classical bounds for the Hausdorff dimension of the singular set of the free boundary:

\begin{proposition}[Dimension of the singular set] \label{prop:singular}
	Assume \eqref{hip:setting}. There exists $n^* \in \{5, 6, 7\}$ such that
	\begin{itemize}
		\item if $n < n^*$, then $\Sing(\pom_u)$ is empty,
		\item if $n = n^*$, then $\Sing(\pom_u)$ is a locally finite set,
		\item if $n > n^*$, then $\dim_{\mathcal{H}} (\Sing(\pom_u)) \leq n - n^*$.
	\end{itemize}
\end{proposition}

The result follows from very well-established techniques in the area, namely the dimension reduction procedure, originally introduced by Federer and later exploited by Weiss in the context of free boundary problems, to reduce to the case of 1-homogeneous global minimizers. In our setting, one may just invoke, upon very slight modifications, the abstract result from \cite[Proposition 10.13]{V}; see Appendix~\ref{app:proofs} for more details.


\section{Proof of Theorem~\ref{th:main}} \label{sec:proof}

We are now ready to prove our main result, Theorem~\ref{th:main}. Fix $m > 0$ and admissible $F(x, u), A(x)$ and $q(x)$, satisfying the periodicity hypothesis~\eqref{hip:periodic}  (see Definition~\ref{def:admissible}). 

Let us fix some $R \gg 1$ to be determined later. Now take a cutoff $\varphi_R \in C^\infty_c(\Rn)$ satisfying 
\begin{equation*}
	0 \leq \varphi_R \leq 1 \text{ in $\Rn$}, 
	\quad 
	\varphi_R \equiv 1 \text{ in $B_R$}, 
	\quad 
	\varphi_R \equiv 0 \text{ in $\Rn \setminus B_{2R}$}.
\end{equation*}
We then define the smoothly truncated energy
\begin{equation*}
	\F_{0, R} (v, D) 
	:=
	\int_D \abs{\nabla v}^2 dx
	- 2 \int_D F_R(x, v) \, dx, 
	\quad D \subset \Rn,
\end{equation*}
with $F_R(x, v) := F(x, v) \, \varphi_R(x)$.
As usual, we will omit $D$ when $D = \Rn$.
Since $\varphi_R$ only depends on $x$, is bounded between 0 and 1, and is smooth, it is clear that $F_R$ is admissible\footnote{
To verify \eqref{hip:sol_not_zero} with $F_R$, note the following: since \eqref{hip:sol_not_zero} holds for $F$, by density, there exists $u_0 \in C^\infty_c(\Rn)$ with $\F_0 (u_0) < 0$. Then, taking $R$ large enough so that $\supp u_0 \subset B_{R/10}$, \eqref{hip:sol_not_zero} is satisfied with $F_R$.
}, and in fact fulfills \eqref{hip:F_quadratic} and \eqref{hip:F'_F''} with the same constants as $F$.
In what follows, we shall make sure that no constant depends on $R$, but only on the constants appearing in the admissibility of $F$ (equivalently, $F_R$), $A$ and $q$, and of course $m$ and $n$.

\subsubsection*{Step 1. Existence of minimizers of $F_{0, R}$ in $\K_{\leq m}$.} 
First let us note that Faber-Krahn's inequality provides a Poincaré inequality in $\K_{\leq m}$. Indeed, given any $u \in \K_{\leq m}$, the symmetrization $u^*$ of $u$ satisfies $u^* \in H^1_0(B^{\widetilde{m}})$ (see \eqref{eq:mtilde}), so that we obtain
\begin{equation} \label{eq:poincare}
	\int_\Rn u^2 \, dx 
	=
	\int_{B^\tm} (u^*)^2 \, dx 
	\leq 
	\lambda_1(B^\tm)^{-1} \int_{B^\tm} \abs{\nabla u^*}^2 \, dx 
	\leq 
	\lambda_1(B^\tm)^{-1} \int_\Rn \abs{\nabla u}^2 \, dx.
\end{equation}

Having obtained this Poincaré inequality, the existence of minimizers of $\F_{0, R}$ in $\K_{\leq m}$ follows from standard arguments. 
First, the fact that $u \in \K_{\leq m}$ (and \eqref{eq:mtilde}), \eqref{hip:zero}, \eqref{hip:F_quadratic}, \eqref{hip:A_elliptic_bounded}, and \eqref{eq:poincare} imply
\begin{align}\notag
	\F_{0, R}(u)
	&\geq 
	\lambda \int_\Rn \!\! \abs{\nabla u}^2 \, dx 
	\,-\, 2 N \abs{\{u > 0\}} 
	\,-\, 2 b \int_\Rn u^2 \, dx\\
	&\geq
	\left(\lambda- \frac{2b}{\lambda_1(B^m)}\right) \int_\Rn \!\! \abs{\nabla u}^2 dx 
	\,-\, 2 N \frac{m}{\underline{q}},\label{E.mystar}
\end{align}
so that the smallness of $b$ in \eqref{hip:F_quadratic} gives $\F_{0, R}(u) \geq -2Nm/\underline{q}$ for every $u \in \K_{\leq m}$. That is, $\F_{0, R}$ is bounded below uniformly, with constants independent of $R$. 

This allows us to take a minimizing sequence $u_j \in \K_{\leq m}$ (i.e. $\F_{0, R}(u_j) \longrightarrow \inf_{\K_{\leq m}} \F_{0, R}$ as $j \to +\infty$). Concretely, testing with the constant zero function, we obtain that $\F_{0, R}(u_j) \leq 1$ for every $j$ (up to taking a subsequence). This and~\eqref{E.mystar} ensure that as long as $b < \lambda \lambda_1(B^m)/2$ we have the uniform bound
\begin{equation} \label{eq:bound_norm_gradient}
	\int_\Rn \abs{\nabla u_j}^2 dx \leq \left( 1 + 2N \frac{m}{\underline{q}} \right) \left(\lambda - \frac{2b}{\lambda_1(B^m)}\right)^{-1}.
\end{equation}

Thus, by the Poincaré inequality in \eqref{eq:poincare}, $u_j$ are also  bounded in $L^2(\Rn)$, with uniform constants. Therefore, by the Banach--Alaoglu compactness theorem, there exists $u_0 \in H^1(\Rn)$ such that $u_j \rightharpoonup u_0$ in $H^1(\Rn)$ (up to a subsequence). Therefore, by the Rellich--Kondrachov theorem in the bounded set $B_{2R}$, we get (up to a subsequence) $u_j \to u_0$ in $L^2(B_{2R})$, so (up to a further subsequence) we can also get pointwise a.e. convergence $u_j \to u_0$ in $B_{2R}$. This convergence, along with the smoothness of $F_R$ (recall \eqref{hip:F_regular}) and the fact that $\supp F_R \subset B_{2R}$, lower semi-continuity of weak convergence (recall \eqref{hip:A_regular}) and the Dominated Convergence Theorem
, imply 
\begin{multline} \label{eq:lsc}
	\F_{0, R}(u_0)
	=
	\int_\Rn \abs{\nabla u_0}^2 \, dx 
	- 2 \int_{B_{2R}} F_R(x, u_0) \, dx 
	\\ \leq 
	\liminf_{j \to +\infty} \int_\Rn \abs{\nabla u_j}^2 \, dx 
	- \lim_{j \to +\infty} \; 2 \int_{B_{2R}} F_R(x, u_j) \, dx 
	=
	\liminf_{j \to +\infty} \F_{0, R} (u_j),
\end{multline}
whence $u_0$ minimizes $\F_{0, R}$. Moreover, by Mazur's theorem on weak convergences and convexity,  $u_0 \geq 0$ in $\Rn$. 

Lastly, to see that $\Vol_q(\Omega_{u_0}) \leq m$, it suffices (by the Monotone Convergence Theorem) to check that $\Vol_q(\Omega_{u_0} \cap B_r) \leq m$ for every $r > 0$. Given one such $r > 0$, using standard compactness results as above we obtain that $u_j \to u_0$ a.e. in $B_r$ (up to a subsequence). This clearly implies that $\mathbf{1}_{\{u_0 > 0\}}(x) \leq \liminf_{j \to +\infty} \mathbf{1}_{\{u_j > 0\}}(x)$ a.e. in $B_r$. Integrating, this yields $\Vol_q(\Omega_{u_0} \cap B_r) \leq m$ because $u_j \in \K_{\leq m}$. This finishes the proof that $u_0 \in \K_{\leq m}$, so $u_0$ is a minimizer of $\F_{0, R}$ in $\K_{\leq m}$.

\subsubsection*{Step 2. Properties of the minimizer of $\F_{0, R}$.} 
The fact that $u_0$ has the properties listed in the conclusions of Theorem~\ref{th:main} is  a consequence of all the previous sections, because \eqref{hip:setting} is because (using $u_0, F_R$ and $\F_{0, R}$ in place of $u, F$ and $\F_0$) with constants that do not depend on $R$, as discussed right before Step 1.

Indeed, in Lemma~\ref{lem:saturation} we proved that $u_0 \in \K_{=m}$, in Proposition~\ref{prop:lipschitz} we showed that $u_0$ is Lipschitz and $\{u_0 > 0\}$ is open, and in Proposition~\ref{prop:Omega_bounded} we obtained that $\{u_0 > 0\}$ is also bounded. We also proved the equations~\eqref{eq:bvp} in Corollary~\ref{corol:EL2} (in the first equation we have $F_R$ on the right-hand-side instead of $F$) and Proposition~\ref{prop:overdetermined}. And lastly, the properties of the regular and singular parts of the boundary follow from Propositions~\ref{prop:regular} and~\ref{prop:singular}.

To finish, we need to show that $u_0$ solves the equation \eqref{eq:bvp} with $F$ in the right-hand-side instead of $F_R$, and also that the support of $u_0$ is uniformly bounded. We shall next show that this follows using the periodicity assumption~\eqref{hip:periodic}. 

\subsubsection*{Step 3. Boundedness of supports, and transference to the original $F$.} 
Recall that we showed in Proposition~\ref{prop:Omega_bounded} that $\Omega_{u_0}$ has a finite number $N'\leq N$ of enlarged connected components (see Definition~\ref{def:ecc}), and each of these enlarged connected components has diameter controlled by $D$. The important point here is that $N$ and $D$ only depend on $m, n$ and the admissibility constants of $F_R, A$ and $q$. And since the admissibility constants of $F_R$ and $F$ are the same, $N$ and $D$ are independent of $R$. 

Now label all the enlarged connected components of $\Omega_{u_0}$ as $\ECC(V_1), \ldots, \ECC(V_{N'})$. Using the $x$-periodicity of $F$, $A$ and~$q$, we will translate them towards the origin, without increasing the energy or changing the $q$-volume of the support. Writing $T$ for the length of the period in \eqref{hip:periodic}, we may assume that $D \geq 10T$. Concretely, given $W_j := \ECC(V_j)$, consider a translation of $W_j$ by a whole number of periods in each direction (say by the vector $Tv_j$ with $v_j \in \Z^n$) so that $\widetilde{W}_j := W_j - Tv_j \subset (2D j, 0, \ldots, 0) + [0, 2D]^n$ (we can do this because $D \geq 10T$, see Figure~\ref{fig:periodic}). This way, it is clear that $\widetilde{W}_j \subset [0, 2DN']^n \subset B_R$ if we choose $R$ large enough. 

After translating, we have a function $\widetilde{u}_0\in H^1(\Rn)$ which satisfies $\widetilde{u}_0(x) := u_0(x + Tv_j)$ on each set $\widetilde{W}_j$. The point is that $\F_{0, R}(\widetilde{u}_0, \widetilde{W}_j) \leq \F_{0, R}(u_0, W_j)$. To see this, note that, of course, the Dirichlet energies are the same because $A$ is periodic;  for the other terms, using the periodicity of $F$,
\begin{multline*}
	\int_{\widetilde{W}_j} F_R(x, \widetilde{u}_0(x)) \, dx 
	=
	\int_{W_j} F_R(x-Tv_j, u_0(x)) \, dx
	=
	\int_{W_j} F(x-Tv_j, u_0(x)) \, \varphi_R(x-Tv_j) \, dx
	\\ =
	\int_{W_j} F(x, u_0(x)) \, \varphi_R(x-Tv_j) \, dx
	\geq 
	\int_{W_j} F(x, u_0(x)) \, \varphi_R(x) \, dx.
\end{multline*}
In the last step we have used that $F \geq 0$ (see \eqref{hip:zero} and \eqref{hip:negative}) and $\varphi_R(x - Tv_j) = 1 \geq \varphi_R(x)$ for $x \in W_j$ (indeed, in such case, $x - Tv_j \in \widetilde{W}_j \subset B_R$). The analysis of the $q$-volume is analogous, because $q$ is periodic.

\begin{figure}[t!]
	\centering 
	\resizebox{0.5\textwidth}{!}{
		\begin{tikzpicture}[scale=1]
			\draw[thin, lightgray] (0, 13.547) rectangle (13.547, 13.547);
			\draw[thin, lightgray] (0, 12.982) rectangle (13.547, 12.982);
			\draw[thin, lightgray] (0, 12.418) rectangle (13.547, 12.418);
			\draw[thin, lightgray] (0, 11.853) rectangle (13.547, 11.853);
			\draw[thin, lightgray] (0, 11.289) rectangle (13.547, 11.289);
			\draw[thin, lightgray] (0, 10.724) rectangle (13.547, 10.724);
			\draw[thin, lightgray] (0, 10.16) rectangle (13.547, 10.16);
			\draw[thin, lightgray] (0, 9.596) rectangle (13.547, 9.596);
			\draw[thin, lightgray] (0, 9.031) rectangle (13.547, 9.031);
			\draw[thin, lightgray] (0, 8.467) rectangle (13.547, 8.467);
			\draw[thin, lightgray] (0, 7.902) rectangle (13.547, 7.902);
			\draw[thin, lightgray] (0, 7.338) rectangle (13.547, 7.338);
			\draw[thin, lightgray] (0, 6.773) rectangle (13.547, 6.773);
			\draw[thin, lightgray] (0, 6.209) rectangle (13.547, 6.209);
			\draw[thin, lightgray] (0, 5.644) rectangle (13.547, 5.644);
			\draw[thin, lightgray] (0, 5.08) rectangle (13.547, 5.08);
			\draw[thin, lightgray] (0, 4.516) rectangle (13.547, 4.516);
			\draw[thin, lightgray] (0, 3.951) rectangle (13.547, 3.951);
			\draw[thin, lightgray] (0, 3.387) rectangle (13.547, 3.387);
			\draw[thin, lightgray] (0, 2.822) rectangle (13.547, 2.822);
			\draw[thin, lightgray] (0, 2.258) rectangle (13.547, 2.258);
			\draw[thin, lightgray] (0, 1.693) rectangle (13.547, 1.693);
			\draw[thin, lightgray] (0, 1.129) rectangle (13.547, 1.129);
			\draw[thin, lightgray] (0, 0.564) rectangle (13.547, 0.564);
			\draw[thin, lightgray] (0, 0) rectangle (13.547, 0);
			\draw[thin, lightgray] (0, 13.547) rectangle (0, 0);
			\draw[thin, lightgray] (0.564, 13.547) rectangle (0.564, 0);
			\draw[thin, lightgray] (1.129, 13.547) rectangle (1.129, 0);
			\draw[thin, lightgray] (1.693, 13.547) rectangle (1.693, 0);
			\draw[thin, lightgray] (2.258, 13.547) rectangle (2.258, 0);
			\draw[thin, lightgray] (2.822, 13.547) rectangle (2.822, 0);
			\draw[thin, lightgray] (3.387, 13.547) rectangle (3.387, 0);
			\draw[thin, lightgray] (3.951, 13.547) rectangle (3.951, 0);
			\draw[thin, lightgray] (4.516, 13.547) rectangle (4.516, 0);
			\draw[thin, lightgray] (5.08, 13.547) rectangle (5.08, 0);
			\draw[thin, lightgray] (5.644, 13.547) rectangle (5.644, 0);
			\draw[thin, lightgray] (6.209, 13.547) rectangle (6.209, 0);
			\draw[thin, lightgray] (6.773, 13.547) rectangle (6.773, 0);
			\draw[thin, lightgray] (7.338, 13.547) rectangle (7.338, 0);
			\draw[thin, lightgray] (7.902, 13.547) rectangle (7.902, 0);
			\draw[thin, lightgray] (8.467, 13.547) rectangle (8.467, 0);
			\draw[thin, lightgray] (9.031, 13.547) rectangle (9.031, 0);
			\draw[thin, lightgray] (9.596, 13.547) rectangle (9.596, 0);
			\draw[thin, lightgray] (10.16, 13.547) rectangle (10.16, 0);
			\draw[thin, lightgray] (10.724, 13.547) rectangle (10.724, 0);
			\draw[thin, lightgray] (11.289, 13.547) rectangle (11.289, 0);
			\draw[thin, lightgray] (11.853, 13.547) rectangle (11.853, 0);
			\draw[thin, lightgray] (12.418, 13.547) rectangle (12.418, 0);
			\draw[thin, lightgray] (12.982, 13.547) rectangle (12.982, 0);
			\draw[thin, lightgray] (13.547, 13.547) rectangle (13.547, 0);
			\draw[thin] (4.516, 5.08) rectangle (6.773, 7.338);
			\draw[thin] (6.773, 5.08) rectangle (9.031, 7.338);
			\draw[thin] (9.031, 5.08) rectangle (11.289, 7.338);
			\filldraw[thin, fill=lightblue] 
			(1.445, 11.507)
			.. controls (1.35, 11.607) and (1.25, 11.724) .. (1.268, 11.86)
			.. controls (1.286, 11.996) and (1.422, 12.15) .. (1.562, 12.199)
			.. controls (1.703, 12.249) and (1.848, 12.195) .. (1.952, 12.245)
			.. controls (2.056, 12.294) and (2.119, 12.448) .. (2.192, 12.58)
			.. controls (2.264, 12.711) and (2.345, 12.819) .. (2.422, 12.788)
			.. controls (2.499, 12.756) and (2.572, 12.584) .. (2.572, 12.426)
			.. controls (2.572, 12.267) and (2.499, 12.122) .. (2.413, 12.055)
			.. controls (2.327, 11.987) and (2.228, 11.996) .. (2.192, 11.919)
			.. controls (2.155, 11.842) and (2.183, 11.679) .. (2.291, 11.629)
			.. controls (2.4, 11.579) and (2.59, 11.643) .. (2.739, 11.629)
			.. controls (2.889, 11.616) and (2.997, 11.525) .. (2.938, 11.43)
			.. controls (2.88, 11.335) and (2.653, 11.235) .. (2.468, 11.181)
			.. controls (2.282, 11.127) and (2.137, 11.118) .. (2.02, 11.136)
			.. controls (1.902, 11.154) and (1.811, 11.199) .. (1.721, 11.263)
			.. controls (1.63, 11.326) and (1.54, 11.407) .. cycle;
			\filldraw[thin, fill=lightblue] 
			(4.831, 5.863)
			.. controls (4.736, 5.962) and (4.637, 6.08) .. (4.655, 6.216)
			.. controls (4.673, 6.351) and (4.809, 6.505) .. (4.949, 6.555)
			.. controls (5.089, 6.605) and (5.234, 6.55) .. (5.338, 6.6)
			.. controls (5.442, 6.65) and (5.506, 6.804) .. (5.578, 6.935)
			.. controls (5.651, 7.066) and (5.732, 7.175) .. (5.809, 7.143)
			.. controls (5.886, 7.112) and (5.958, 6.94) .. (5.958, 6.781)
			.. controls (5.958, 6.623) and (5.886, 6.478) .. (5.8, 6.41)
			.. controls (5.714, 6.342) and (5.614, 6.351) .. (5.578, 6.274)
			.. controls (5.542, 6.197) and (5.569, 6.034) .. (5.678, 5.985)
			.. controls (5.786, 5.935) and (5.977, 5.998) .. (6.126, 5.985)
			.. controls (6.275, 5.971) and (6.384, 5.881) .. (6.325, 5.786)
			.. controls (6.266, 5.691) and (6.04, 5.591) .. (5.854, 5.537)
			.. controls (5.669, 5.482) and (5.524, 5.473) .. (5.406, 5.491)
			.. controls (5.289, 5.509) and (5.198, 5.555) .. (5.108, 5.618)
			.. controls (5.017, 5.681) and (4.927, 5.763) .. cycle;
			\filldraw[thin, fill=lightblue] 
			(6.852, 1.832)
			.. controls (6.68, 1.601) and (6.493, 1.236) .. (6.538, 0.984)
			.. controls (6.584, 0.732) and (6.863, 0.592) .. (7.115, 0.67)
			.. controls (7.367, 0.748) and (7.592, 1.043) .. (7.713, 1.386)
			.. controls (7.833, 1.729) and (7.849, 2.121) .. (7.769, 2.284)
			.. controls (7.689, 2.448) and (7.512, 2.384) .. (7.424, 2.341)
			.. controls (7.336, 2.299) and (7.337, 2.278) .. (7.259, 2.219)
			.. controls (7.181, 2.16) and (7.025, 2.063) .. cycle;
			\filldraw[thin, fill=lightblue] 
			(7.417, 6.347)
			.. controls (7.245, 6.116) and (7.057, 5.752) .. (7.103, 5.499)
			.. controls (7.148, 5.247) and (7.427, 5.108) .. (7.679, 5.186)
			.. controls (7.931, 5.263) and (8.156, 5.558) .. (8.277, 5.902)
			.. controls (8.398, 6.245) and (8.414, 6.636) .. (8.333, 6.8)
			.. controls (8.253, 6.964) and (8.076, 6.899) .. (7.988, 6.857)
			.. controls (7.9, 6.814) and (7.901, 6.794) .. (7.823, 6.735)
			.. controls (7.746, 6.676) and (7.589, 6.578) .. cycle;
			\filldraw[thin, fill=lightblue] 
			(12.165, 10.3)
			.. controls (12.068, 10.329) and (11.977, 10.324) .. (11.848, 10.276)
			.. controls (11.72, 10.228) and (11.553, 10.136) .. (11.486, 10.016)
			.. controls (11.419, 9.895) and (11.452, 9.745) .. (11.583, 9.739)
			.. controls (11.714, 9.734) and (11.945, 9.874) .. (12.06, 9.868)
			.. controls (12.176, 9.863) and (12.176, 9.713) .. (12.103, 9.605)
			.. controls (12.031, 9.498) and (11.886, 9.434) .. (11.867, 9.348)
			.. controls (11.848, 9.262) and (11.956, 9.155) .. (12.074, 9.15)
			.. controls (12.192, 9.144) and (12.32, 9.241) .. (12.361, 9.404)
			.. controls (12.401, 9.568) and (12.353, 9.798) .. (12.39, 9.852)
			.. controls (12.428, 9.906) and (12.551, 9.782) .. (12.648, 9.675)
			.. controls (12.744, 9.568) and (12.814, 9.477) .. (12.859, 9.493)
			.. controls (12.905, 9.509) and (12.926, 9.632) .. (12.878, 9.734)
			.. controls (12.83, 9.836) and (12.712, 9.916) .. (12.701, 9.997)
			.. controls (12.69, 10.077) and (12.787, 10.158) .. (12.867, 10.241)
			.. controls (12.948, 10.324) and (13.012, 10.41) .. (12.988, 10.464)
			.. controls (12.964, 10.517) and (12.851, 10.539) .. (12.768, 10.434)
			.. controls (12.685, 10.329) and (12.631, 10.099) .. (12.543, 10.016)
			.. controls (12.454, 9.933) and (12.331, 9.997) .. (12.245, 10.037)
			.. controls (12.16, 10.077) and (12.111, 10.093) .. (12.168, 10.102)
			.. controls (12.224, 10.11) and (12.385, 10.11) .. (12.457, 10.112)
			.. controls (12.53, 10.115) and (12.513, 10.12) .. (12.5, 10.126)
			.. controls (12.487, 10.131) and (12.476, 10.136) .. (12.42, 10.171)
			.. controls (12.363, 10.206) and (12.261, 10.27) .. cycle;
			\filldraw[thin, fill=lightblue] 
			(9.907, 6.349)
			.. controls (9.811, 6.378) and (9.719, 6.373) .. (9.591, 6.325)
			.. controls (9.462, 6.276) and (9.296, 6.185) .. (9.229, 6.065)
			.. controls (9.162, 5.944) and (9.194, 5.794) .. (9.325, 5.788)
			.. controls (9.457, 5.783) and (9.687, 5.922) .. (9.803, 5.917)
			.. controls (9.918, 5.912) and (9.918, 5.762) .. (9.845, 5.654)
			.. controls (9.773, 5.547) and (9.628, 5.483) .. (9.609, 5.397)
			.. controls (9.591, 5.311) and (9.698, 5.204) .. (9.816, 5.198)
			.. controls (9.934, 5.193) and (10.063, 5.29) .. (10.103, 5.453)
			.. controls (10.143, 5.617) and (10.095, 5.847) .. (10.132, 5.901)
			.. controls (10.17, 5.955) and (10.293, 5.831) .. (10.39, 5.724)
			.. controls (10.486, 5.617) and (10.556, 5.526) .. (10.602, 5.542)
			.. controls (10.647, 5.558) and (10.669, 5.681) .. (10.62, 5.783)
			.. controls (10.572, 5.885) and (10.454, 5.965) .. (10.443, 6.046)
			.. controls (10.433, 6.126) and (10.529, 6.207) .. (10.61, 6.29)
			.. controls (10.69, 6.373) and (10.754, 6.459) .. (10.73, 6.512)
			.. controls (10.706, 6.566) and (10.594, 6.587) .. (10.51, 6.483)
			.. controls (10.427, 6.378) and (10.374, 6.148) .. (10.285, 6.065)
			.. controls (10.197, 5.981) and (10.073, 6.046) .. (9.988, 6.086)
			.. controls (9.902, 6.126) and (9.853, 6.142) .. (9.91, 6.15)
			.. controls (9.966, 6.158) and (10.127, 6.158) .. (10.199, 6.161)
			.. controls (10.272, 6.164) and (10.256, 6.169) .. (10.242, 6.175)
			.. controls (10.229, 6.18) and (10.218, 6.185) .. (10.162, 6.22)
			.. controls (10.106, 6.255) and (10.004, 6.319) .. cycle;
			\filldraw[draw=red, fat, ->, fill=violet] (2.474, 11.027) -- (5.241, 6.683);
			\filldraw[draw=red, fat, ->, fill=violet] (11.806, 9.546) -- (9.987, 6.489);
			\filldraw[draw=red, fat, ->, fill=violet] (7.542, 2.499) -- (7.735, 5.074);
			\filldraw[thin, <->, fill=violet] (4.516, 4.516) -- (6.773, 4.516);
			\filldraw[thin, <->, fill=violet] (6.773, 4.516) -- (9.031, 4.516);
			\filldraw[thin, <->, fill=violet] (9.031, 4.516) -- (11.289, 4.516);
			\node[anchor=center, font=\LARGE] at (5.75, 4.2) {$2D$};
			\node[anchor=center, font=\LARGE] at (8.245, 4.2) {$2D$};
			\node[anchor=center, font=\LARGE] at (10.238, 4.2) {$2D$};
			\filldraw[thin, fill=lightblue] 
			(2.547, 11.72)
			.. controls (2.5, 11.75) and (2.476, 11.81) .. (2.5, 11.858)
			.. controls (2.525, 11.905) and (2.599, 11.942) .. (2.676, 11.953)
			.. controls (2.753, 11.965) and (2.833, 11.952) .. (2.865, 11.973)
			.. controls (2.897, 11.995) and (2.882, 12.05) .. (2.903, 12.083)
			.. controls (2.923, 12.115) and (2.98, 12.124) .. (3.022, 12.092)
			.. controls (3.063, 12.061) and (3.089, 11.988) .. (3.093, 11.899)
			.. controls (3.097, 11.81) and (3.079, 11.704) .. (3.032, 11.662)
			.. controls (2.985, 11.619) and (2.91, 11.641) .. (2.855, 11.657)
			.. controls (2.8, 11.673) and (2.765, 11.682) .. (2.714, 11.686)
			.. controls (2.662, 11.691) and (2.593, 11.691) .. cycle;
			\filldraw[thin, fill=lightblue] 
			(5.963, 6.072)
			.. controls (5.916, 6.102) and (5.891, 6.161) .. (5.916, 6.209)
			.. controls (5.941, 6.257) and (6.014, 6.293) .. (6.091, 6.305)
			.. controls (6.168, 6.316) and (6.248, 6.303) .. (6.281, 6.325)
			.. controls (6.313, 6.346) and (6.298, 6.402) .. (6.318, 6.434)
			.. controls (6.339, 6.466) and (6.396, 6.475) .. (6.437, 6.444)
			.. controls (6.479, 6.412) and (6.505, 6.34) .. (6.509, 6.25)
			.. controls (6.512, 6.161) and (6.494, 6.055) .. (6.448, 6.013)
			.. controls (6.401, 5.971) and (6.326, 5.993) .. (6.271, 6.008)
			.. controls (6.216, 6.024) and (6.181, 6.033) .. (6.129, 6.037)
			.. controls (6.078, 6.042) and (6.009, 6.042) .. cycle;
			\draw[thin, dashed] (9.953, 0.015) .. controls (10.405, 0.547) and (10.524, 0.711) .. (10.606, 0.825) .. controls (10.687, 0.939) and (10.731, 1.005) .. (10.768, 1.063) .. controls (10.806, 1.121) and (10.839, 1.173) .. (10.876, 1.235) .. controls (10.913, 1.297) and (10.956, 1.37) .. (11.015, 1.48) .. controls (11.075, 1.589) and (11.151, 1.736) .. (11.228, 1.9) .. controls (11.305, 2.064) and (11.383, 2.246) .. (11.444, 2.399) .. controls (11.505, 2.551) and (11.548, 2.675) .. (11.597, 2.832) .. controls (11.645, 2.989) and (11.698, 3.18) .. (11.738, 3.339) .. controls (11.778, 3.498) and (11.805, 3.626) .. (11.831, 3.782) .. controls (11.858, 3.938) and (11.885, 4.123) .. (11.904, 4.311) .. controls (11.923, 4.499) and (11.936, 4.689) .. (11.941, 4.883) .. controls (11.946, 5.077) and (11.943, 5.276) .. (11.938, 5.428) .. controls (11.933, 5.579) and (11.925, 5.682) .. (11.912, 5.808) .. controls (11.898, 5.935) and (11.879, 6.086) .. (11.85, 6.259) .. controls (11.821, 6.431) and (11.782, 6.626) .. (11.736, 6.819) .. controls (11.689, 7.012) and (11.635, 7.203) .. (11.569, 7.4) .. controls (11.504, 7.597) and (11.427, 7.799) .. (11.342, 7.999) .. controls (11.256, 8.198) and (11.162, 8.394) .. (11.071, 8.566) .. controls (10.981, 8.737) and (10.896, 8.884) .. (10.838, 8.981) .. controls (10.78, 9.078) and (10.749, 9.125) .. (10.66, 9.25) .. controls (10.57, 9.374) and (10.422, 9.575) .. (10.252, 9.78) .. controls (10.082, 9.984) and (9.891, 10.192) .. (9.753, 10.335) .. controls (9.615, 10.479) and (9.53, 10.557) .. (9.368, 10.69) .. controls (9.205, 10.824) and (8.964, 11.011) .. (8.812, 11.126) .. controls (8.661, 11.24) and (8.597, 11.282) .. (8.518, 11.332) .. controls (8.438, 11.383) and (8.341, 11.441) .. (8.225, 11.507) .. controls (8.11, 11.573) and (7.975, 11.645) .. (7.791, 11.732) .. controls (7.607, 11.819) and (7.373, 11.92) .. (7.191, 11.994) .. controls (7.009, 12.067) and (6.878, 12.112) .. (6.739, 12.155) .. controls (6.599, 12.198) and (6.45, 12.24) .. (6.29, 12.278) .. controls (6.13, 12.317) and (5.959, 12.352) .. (5.787, 12.382) .. controls (5.615, 12.411) and (5.441, 12.435) .. (5.235, 12.452) .. controls (5.029, 12.47) and (4.79, 12.482) .. (4.561, 12.484) .. controls (4.332, 12.485) and (4.113, 12.476) .. (3.918, 12.461) .. controls (3.724, 12.446) and (3.553, 12.425) .. (3.358, 12.392) .. controls (3.163, 12.359) and (2.943, 12.314) .. (2.733, 12.262) .. controls (2.522, 12.209) and (2.32, 12.149) .. (2.146, 12.091) .. controls (1.971, 12.033) and (1.823, 11.977) .. (1.677, 11.916) .. controls (1.532, 11.855) and (1.387, 11.789) .. (1.279, 11.737) .. controls (1.171, 11.686) and (1.1, 11.649) .. (1.017, 11.604) .. controls (0.933, 11.559) and (0.839, 11.505) .. (0.709, 11.426) .. controls (0.578, 11.347) and (0.413, 11.243) .. (0, 10.946);
			\node[anchor=center, font=\huge] at (10.057, 10.963) {$B_R$};
		\end{tikzpicture}
	}
	\caption{Example of how one can move connected components of $\Omega_u$ to fixed cubes of length $K$ in a region around the origin, under the periodicity assumption \eqref{hip:periodic}.}
	\label{fig:periodic}
\end{figure}
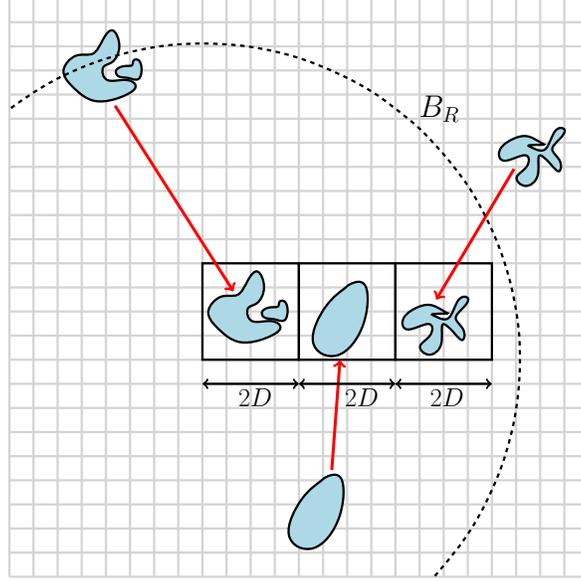

Therefore, after translating (one by one) all the $N'$ enlarged connected components of $\Omega_{u_0}$, we infer that $\Omega_{\widetilde{u}_0} \subset B_R$ if $R$ is large enough. Moreover, recalling that $u_0$ is a minimizer of $\F_{0, R}$, so is $\widetilde{u}_0$. With large~$R$, we have that $F_R \equiv F$ in $\Omega_{\widetilde{u}_0}$, whence $\widetilde{u}_0$ actually solves the PDE in \eqref{eq:bvp} with~$F$ (because it solves it with $F_R$, as we discussed in Step 2). Therefore, $u:=\widetilde{u}_0$ has all the properties listed in Theorem~\ref{th:main}, and the theorem follows.


\section{Overdetermined problems on manifolds}
\label{S.manifolds}

We shall next sketch why the approach developed in this paper works equally well for semilinear equations on compact manifolds, with only minor modifications.

Throughout this section, let $M$ be a compact connected manifold of class~$C^\infty$, endowed with a smooth Riemannian metric~$g$. We consider the following analog of~\eqref{eq:fbp}:
\begin{equation} \label{E.manifoldeq}
	\begin{cases}
		-\Delta_g v = f(x, v) \quad &\text{ in } \Omega\subset M, \\
		v = 0 \quad \text{and}\quad  |\nabla_g u|^2_g = c\,{q(x)} &\text{ on } \pom.\end{cases}
\end{equation}
Here $\Delta_g$ and~$\nabla_g$ respectively denote the Laplace--Beltrami operator and the covariant derivative on the Riemannian manifold $(M,g)$. To control the size of~$\Omega$, we will impose the condition
\begin{equation}\label{E.volumeg}
	\Vol_{q,g}(\Omega) := \int_\Omega q\, d\mu_g=m\,,
\end{equation}
where $d\mu_g$ denotes the Riemannian measure. 

Our main result for this overdetermined problem is the following. Here the definition of ``admissible'' is essentially as in the case of overdetermined problems on~$\R^n$, with some slight modifications in the constants involved that we will address later.

\begin{theorem}
	Let $M$ be a smooth, compact and connected Riemannian manifold of dimension $n\leq 4$. Then, for any constant $
	0<m<\Vol_{q,g}(M)$ and for any pair of smooth admissible functions~$q:M\to(0,\infty)$ and~$f:M\times[0,\infty)\to\R$, there exists an open set $\Omega\subset M$ with smooth boundary satisfying the $q$-volume constraint~\eqref{E.volumeg} and a constant $c>0$ for which the overdetermined boundary value problem~\eqref{E.manifoldeq} admits a positive solution $v\in C^\infty(\overline{\Omega})$.
\end{theorem}

As before, this result follows from analogous result about minimizers of the functional
\begin{equation*}
	\F_0(u)
	:=
	\int_M |\nabla_g u|_g^2  \, d\mu_g - 2 \int_M F(x, u) \, d\mu_g
	\end{equation*}
	on the space
	\begin{equation*}
	\K_{\leq m} := \big\{ u \in H^1(M) : \; u \geq 0, \;\; \Vol_{q,q}(\{u>0\}) \leq m \big\}.
\end{equation*}
We still use the notation $F(x,v):=\int_0^t f(x,t)\, dt$.
Specifically, we prove the following result:

\begin{theorem} \label{th:manifolds}
	Let $M$ be a smooth, compact and connected Riemannian manifold of dimension $n$.
	Let $0 < m < \Vol_{q,g}(M)$ and assume that $F(x, u)$ and $q(x)$ are admissible. Then there exists a Lipschitz continuous minimizer $u$ of $\F_0$ in $\K_{\leq m}$,  which actually belongs to $\K_{=m}$. Moreover, $\{u > 0\}$ is open and 
	\begin{equation*}
		\begin{cases}
			-\Delta_g u = f(x, u) \quad & \text{in } \{u > 0\}, \\
			 |\nabla_g u|_g^2 = c\, {q} \quad & \text{on } \partial \{u > 0\} \text{ in the viscosity sense}
		\end{cases}
	\end{equation*}
	for some constant $c > 0$.
	Furthermore, the free boundary can be decomposed as a disjoint union 
	$$\partial \{u > 0\} = \Reg(\{u > 0\}) \cup \Sing(\{u > 0\}),$$
	where:
	\begin{enumerate}
		\item $\Reg(\{u > 0\})$ is a $C^{1, \alpha}$ submanifold of dimension $n-1$ for some $\alpha > 0$, open within $\partial \{u > 0\}$, and $|\nabla_g u|_g^2 = c\, q$ holds pointwise on this set. If $f$ and~$q$ are smooth, $\Reg(\{u > 0\})$ is smooth as well.
	\item $\Sing(\{u > 0\})$ is a closed set within $\partial \{u >0\}$ of Hausdorff dimension at most $n-5$. Moreover $\Sing(\{u > 0\})$ is empty if $n \leq 4$, and consists at most of a countable set of points if $n = 5$.
	\end{enumerate}
\end{theorem}

Let us now sketch the few modifications that are needed to establish this result. 

First, note that the results in Sections~\ref{sec:assumptions}, \ref{sec:basic}, \ref{sec:lip} and \ref{sec:boundary}, except for Proposition~\ref{prop:bounded}, are all of local nature. As the manifold is compact, we can cover it with finitely many charts of (for instance) local normal coordinates. In each chart, which covers a small enough geodesic ball, Equation~\eqref{E.manifoldeq} reads as
\[
-\sum_{i,j}\partial_{x_i} (\sqrt{\abs{g}} g^{ij} \partial_{x_j} u)=\sqrt{\abs{g}}\, f(x,v)\,,
\]
and the Riemannian measure is $d\mu_g=\sqrt{\abs{g}}\, dx$. Here $g^{ij}$ is the inverse matrix of the Riemannian metric, and $\abs{g}$ denotes the determinant of the metric, in these coordinates. So locally, this equation fits within the framework of~\eqref{eq:fbp} with the admissible matrix $A^{ij}:=\sqrt{\abs{g}} g^{ij}$. Therefore, all the results of local nature still hold in manifolds, on small enough geodesic balls.

It is also not hard to check that, although the arguments in Section~\ref{sec:compact_support} are not of local nature, they also hold on a compact manifold. Similarly, the proofs in Section~\ref{sec:first_variation} carry over verbatim to compact manifolds, essentially because they only use general ideas of the calculus of variations. The only point which we want to comment on is the assumption that the exterior of $B$ is connected in Lemma~\ref{lem:almost_minimality}, as can be seen in \cite[Proposition 11.10]{V}: it is not hard to see that manifolds do not introduce any complications here because, just as in the case of~$\R^n$, the complement of a small enough ball in a manifold is always connected (it is diffeomorphic to~$M$ minus a point). 

Regarding Proposition~\ref{prop:bounded}, one should simply work in small balls. We can actually provide an alternative proof that shows boundedness in any small ball, which by compactness of $M$ gives the result we want. Let us sketch the argument, again  considering balls $B$ that are smaller than the injectivity radius. Then, we can work in local coordinates, considering an admissible coefficient matrix $A(x)$. Since we know $u \in H^1(M)$, necessarily $u \in L^2(B)$. As in \eqref{eq:f_u}, we infer $f(x) := F'(x, u(x)) \in L^2(B)$. If we define $w_{-\Delta} (x) := \int_B \Gamma_{-\Delta}(x, y) f(y) \, dy$, we obtain (if $n \geq 3$; otherwise we can simply follow Case 2 in Proposition~\ref{prop:bounded}) $w_{-\Delta} \in L^{2^*}(B)$ by \cite[Lemma 7.12]{GT}. By the comparability of fundamental solutions discussed in Proposition~\ref{prop:bounded}, it also holds $w (x) := \int_B \Gamma_{\cL}(x, y) f(y) \, dy \in L^{2^*}(B)$. Then, by the estimate in the display following \eqref{eq:ptwise_uw}, we infer $u \in L^{2^*}(B)$ (this makes use of Lemma~\ref{lem:harnack}, which still holds in this setting because it is clearly of local nature). Repeating this self-improvement process we will eventually reach sufficiently large exponents to use \cite[Lemma 7.12]{GT} to obtain $w_{-\Delta} \in L^\infty(B)$, which implies $u \in L^\infty(B)$ as before. This completes the proof because, as mentioned above, $M$ is covered by finitely many balls of this kind.

Lastly, to adapt Section~\ref{sec:proof}, we first note that we do not need to introduce the truncation function $\varphi_R$ to ensure that no mass is lost in the limit  in \eqref{eq:lsc}: this follows because, as $M$ is compact,  one can simply use Rellich--Kondrachov on the whole manifold. Therefore, we may work directly with $F$ the whole time, and Step 3 is not needed. 

The only change is in the derivation of \eqref{eq:poincare}. Instead of using symmetrization, on the compact Riemannian manifold, we need to use the existence of sets which minimize the first eigenvalue $\lambda_1$ among those which are quasi-open and have prescribed volume $m$, which is ensured by \cite[Theorem 1.1]{LS} (their proof works without any modification with our modified volume $\Vol_{q,g}$ under the hypotheses \eqref{hip:q_regular} and \eqref{hip:q_positive}). To put it differently, for each~$m\in(0,\Vol_{q,g}(M))$ there exists some $\tilde\lambda(m)>0$ such that
\[
\int_{D}|\nabla_gw|_g^2\, d\mu_g\geq \tilde\lambda(m)\, \int_D w^2\, d\mu_g
\]
for any quasi-open set~$D$ with $\Vol_{q,g}(D)=m$ and any $w\in H^1_0(D)$.

With this in hand, since for any $u \in H^1$ and we have that $\{u > 0\}$ is quasi-open (more precisely, there is a representative in the equivalence class of $u$ in $H^1$ for which this holds, by \cite[Theorem 4.4]{HKM}), we may apply \cite[Theorem 1.1]{LS} to derive \eqref{eq:poincare}. In the subsequent arguments, the only change we need on the admissibility assumptions is that the constant~$b$ in \eqref{hip:F_quadratic} must be smaller than $\tilde\lambda(m)/2$ times the ellipticity constant of the coefficient matrix induced by the metric $g$ in the normal coordinates, as explained above.
The other arguments carry over directly, completing the proof of Theorem~\ref{th:manifolds}.



\section*{Acknowledgments}

P.H.-P.\ thanks Joaquín Domínguez-de-Tena for very helpful discussions. The authors thank Fausto Ferrari for pointing out the reference \cite{DFS2}. The authors are also indebted to Xavier Tolsa for his comments and corrections regarding the proof of Proposition~\ref{prop:bounded}. This work has received funding from the European Research Council (ERC) under the European Union's Horizon 2020 research and innovation programme through the grant agreement~862342 (A.E.). It is also partially supported by the MCIN/AEI grants CEX2023-001347-S, RED2022-134301-T, PID2022-136795NB-I00 and PID2025-168410NB-I00 (A.E.). P.H.-P. was supported by the grant CEX2019-000904-S-20-3 funded by MCIN/AEI/10.13039/501100011033, and PCI2024-155066-2 funded by MICIU/AEI/10.13039/501100011033/UE and co-funded by the European Union. X.R.-O. was supported by the European Union under the ERC Consolidator Grant No 101123223 (SSNSD), the AEI projects PID2021-125021NAI00 and PID2024-156429NB-I00 funded by MICIU/AEI/10.13039/501100011033 and by FEDER (Spain), the AEI-DFG project PCI2024-155066-2, the AEI Grant RED2024-153842-T (Spain), and the AEI Maria de Maeztu Program for Centers and Units of Excellence in R\&D CEX2020-001084-M.


\appendix

\section{Non-periodic admissible nonlinearities do not have minimizers with uniformly bounded supports}	
\label{S.appendix}

In this Appendix, we show that some key estimates can fail when the   admissibility and periodicity conditions do not hold. 

Concerning admissibility, the following elementary calculation shows that quadratic growth of $F$ on $u$, as in \eqref{hip:F_quadratic}, is indeed the threshold to be able to find minimizers of $\F_0$. A minor variation of this argument applies to the case $F(x, u) = c u^p$ for any $p > 2$ and any $c>0$. So from this point of view, \eqref{hip:F_quadratic} is sharp.

\begin{proposition}
	Let $F(x, u) := bu^2$ with some constant $b > \lambda_1(B^m)/2$. Suppose that $A:=I$ and $q:=1$. Then 
	\[
	\inf_{\varphi\in\K_{\leq m}}\F_0(\varphi)=-\infty.
	\]
\end{proposition}	

\begin{proof}
Let $\varphi_1\in H^1_0(B^m)\subset H^1(\Rn)$ the first Dirichlet eigenfunction of the Laplacian in the ball $B^m$. Then, it is easy to compute that $\F_0(\varphi_1) = (\lambda_1(B^m) - 2b) \int_{B^m} \varphi_1^2 \, dx < 0$. Therefore, for any $\tau>0$, $$\F_0(\tau \varphi_1) = \tau^2 \F_0(\varphi_1) ,$$ 
which tends to $-\infty$ as $\tau \to +\infty$. 
\end{proof}

The following proposition, which is less elementary, shows that if~$F$ is not periodic, there is no hope for minimizers to have uniformly bounded supports:

\begin{proposition}
	There are admissible $A, F, q$, and $m > 0$, for which $\diam(\Omega_u)$ is arbitrarily large. More specifically, $\diam(\Omega_u)$ cannot be controlled uniformly with respect to $m, n$ and the admissibility constants of $A, F, q$.
\end{proposition} 
\begin{proof}
	Let us set $A:=I$ and $q:=1$, and assume $F$ is only non-zero in two balls $B^m(x_1), B^m(x_2)$, which are separated by a distance $R \gg 1$ still unspecified. Concretely, set $F(x, u) := \varphi(x) \, u $ for  
	\begin{equation*}
		\varphi(x) :=
		\begin{cases}
			\phi(x - x_1) \quad &\text{for $x \in B^m(x_1)$}, \\ 
			\phi(x - x_2) \quad & \text{for $x \in B^m(x_2)$}, \\ 
			0 & \text{elsewhere},
		\end{cases}
		\qquad \text{where } 
		\begin{cases}
			\phi \in C^\infty_c(B^m), \text{ radially decreasing,}  \\
			\phi \equiv 1 \text{ in $B^{m/2}$}, \\
			0 < \phi < 1 \text{ in $B^m \setminus B^{m/2}$}.
		\end{cases}  
	\end{equation*}
	The reader should actually think that $\phi$ is very close to 0 in $B^m \setminus B^{m/2}$.
	 
	\subsubsection*{Step 1: $\Omega_u$ consists of one or two separated balls.}	 
	Let us consider the associated minimizer $u$ of $\F_0$. Given any connected component $V$ of $\Omega_u$ intersecting $B^m(x_1)$, if we define the set $\mathcal{J}_V := \big\{ \mathbf{j} \in \Z^n : V \cap (\mathbf{j} + [0, 1)^n) \neq \emptyset \big\}$, following the proof of Proposition~\ref{prop:Omega_bounded} we get $\abs{\mathcal{J}_V} \leq C$ for some uniform $C$. Therefore, the connectedness of $V$ implies that $\diam(V) \leq C$, i.e. $V \subset B_C(x_1)$. 
	
	Of course, the same happens with any other connected component of $\Omega_u$ intersecting $B^m(x_1)$, and with those intersecting $B^m(x_2)$. It is also easy to check that every connected component of $\Omega_u$ must intersect $B^m(x_1)$ or $B^m(x_2)$, because  otherwise it would only increase $\F_0(u)$ (since $F$ is supported in $B^m(x_1) \cup B^m(x_2)$), and $u$ is a minimizer of $\F_0$. Therefore, we have shown that $\Omega_u \subset B_C(x_1) \cup B_C(x_2)$. Thus, if we choose $R \gg C$, these two balls are disjoint and the support of $u$ is separated clearly in two parts. By an easy symmetrization argument inside each of these two balls (recall that $\phi$ is radially decreasing), we deduce that (see  Figure~\ref{fig:separated})
	\[
	\Omega_u = B^a(x_1) \cup B^{m-a}(x_2) \quad 
	\text{for some $0 \leq a \leq m$.}
	\]
	Our goal is to show that it must be $0 < a < m$, so that $\diam(\Omega_u) \geq R$, whence the arbitrariness of $R$ would give the result.
	
	\begin{figure}[t!]
		\centering 
		\resizebox{0.9\textwidth}{!}{
			\begin{tikzpicture}[scale=1]
				\draw[blue, thin] 
				(2.117, 2.399) circle[radius=2.117];
				\draw[blue, thin] 
				(13.264, 2.399) circle[radius=2.117];
				\filldraw[thin, fill=lightblue, fill opacity=0.5] 
				(2.117, 2.399) circle[radius=0.847];
				\filldraw[thin, fill=lightblue, fill opacity=0.5] 
				(13.264, 2.399) circle[radius=0.847];
				\draw[thin] 
				(2.117, 2.399) circle[radius=0.423];
				\fill[red, opacity=0.7] 
				(2.117, 2.399) circle[radius=0.32];
				\draw[thin, <->] (2.117, 0) -- (13.264, 0);
				\node[circle, fill, inner sep=1pt] at (2.117, 2.399) {};
				\node[anchor=center,font=\small] at (2.17, 2.528) {$x_1$};
				\node[anchor=center] at (5.535, 1.592) {$B\textsuperscript{$m$/2}(x_1)$};
				\node[anchor=center] at (2.121, 3.447) {$B\textsuperscript{$m$}(x_1)$};
				\node[anchor=center] at (4.476, 3.785) {$B_C(x_1)$};
				\node[anchor=center] at (5.364, 1.03) {$B\textsuperscript{$a$}(x_1)$};
				\draw[->] (4.755, 1.605) -- (2.576, 2.258);
				\draw[->] (4.777, 1.054) -- (2.284, 2.205);
				\node[anchor=center, font=\Large] at (7.686, 0.308) {$R$};
				\node[anchor=center] at (10.794, 3.753) {$B_C(x_2)$};
				\node[anchor=center] at (13.278, 3.506) {$B\textsuperscript{$m$}(x_2)$};
				\node[anchor=center] at (9.721, 1.583) {$B\textsuperscript{$m$/2}(x_2)$};
				\node[anchor=center] at (9.721, 1.022) {$B\textsuperscript{$m$--$a$}(x_2)$};
				\draw[thin, ->] (10.521, 1.016) -- (12.836, 2.019);
				\fill[red, opacity=0.7] 
				(13.264, 2.399) circle[radius=0.564];
				\node[anchor=center] at (13.307, 2.579) {$x_2$};
				\node[circle, fill, inner sep=1pt] at (13.264, 2.399) {};
				\draw[thin] 
				(13.264, 2.399) circle[radius=0.423];
				\draw[thin, ->] (10.521, 1.573) -- (12.851, 2.276);
			\end{tikzpicture}
		}
		\caption{We can force $\Omega_u$ to accumulate in two regions which are far apart.}
		\label{fig:separated}
	\end{figure}
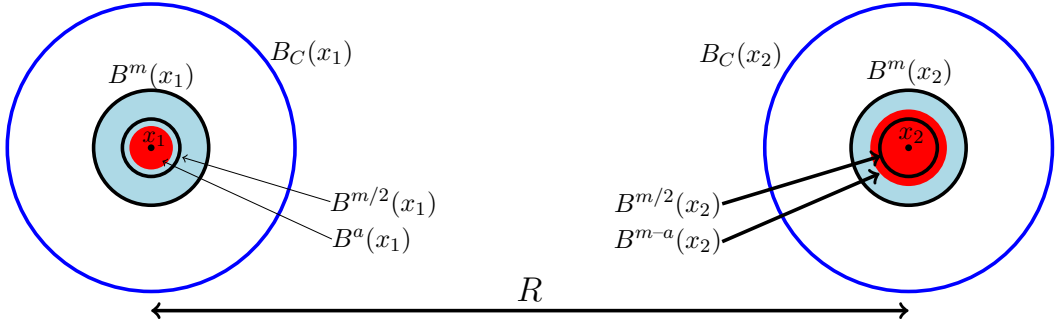
	
	\subsubsection*{Step 2: Energy if there is only one ball.}	
	For the sake of contradiction, let us suppose that $a = 0$. For simplicity of the notation along this step, set $x_2 = 0$. By Corollary~\ref{corol:EL2}, $-\Delta u = \phi$ in $B^m$. And the maximum principle informs us that $u \leq w$, where $-\Delta w = 1$ in $B^m$ and $\{w>0\} = B^m$. Therefore, testing the equation of $u$ against $u$ itself, we obtain 
	\[
	\F_0(u)
	=
	- \int_{B^m} \phi u \, dx
	\geq 
	- \int_{B^m} \phi w \, dx
	=
	- \int_{B^{m/2}} w \, dx - \int_{B^m \setminus B^{m/2}} \phi w \, dx.
	\]
	Since the last term can be made arbitrarily small using appropriate choices of $\phi$, let us focus on the first one. In fact, it is easy to check that $w(x) = -\frac{\abs{x}^2}{2n} + \frac{\rho^2}{2n}$, where $\rho > 0$ is defined so that $\abs{B_{\rho}} = \abs{B^m}$. Similarly, find $r > 0$ so that $\abs{B_r} = \abs{B^{m/2}}$. Then, by an easy integration,
	\begin{equation*}
		\int_{B^{m/2}} w \, dx
		=
		- \frac{\mathcal{H}^{n-1}(\partial B_1)}{2n} \frac{r^{n+2}}{n+2} 
		+ \frac{\rho}{2n} \frac{m}{2}.
	\end{equation*}
	
	\subsubsection*{Step 3: A better candidate.}
	On the other hand, if we concentrate the mass around both $x_1$ and $x_2$, we can obtain smaller energies. Indeed, take $v\in H^1(\Rn)$ satisfying $\Delta v = -1 $ in $B^{m/2}(x_1) \cup B^{m/2}(x_2)$, with $\{v>0\} = B^{m/2}(x_1) \cup B^{m/2}(x_2)$. Then, by a similar calculation as above,
	\begin{equation*}
		\F_0(v) 
		=
		2 \left( - \frac{\mathcal{H}^{n-1}(\partial B_1)}{2n} \frac{r^{n+2}}{n+2} 
		+ \frac{r}{2n} \frac{m}{2} \right).
	\end{equation*}
	
	Noting that $\rho = 2^{1/n} r$ and that $m \approx r^n$, it is easy to check that if $r$ is large enough (equivalently, if $m$ is large enough) it holds
	\begin{equation*}
		\F_0(v) 
		=
		- \frac{\mathcal{H}^{n-1}(\partial B_1)}{2n(n+2)} 2 r^{n+2}
		+ \frac{m}{4n} 2r
		< 
		- \frac{\mathcal{H}^{n-1}(\partial B_1)}{2n(n+2)} r^{n+2}
		+ \frac{m}{4n} 2^{1/n} r
		=
		\int_{B^{m/2}} w \, dx.
	\end{equation*}
	Since these estimates are independent on the choice of $\phi$ and we can make this choice so that $\F_0(u)$ is as close to $\int_{B^{m/2}} w \, dx$ as we wish, we obtain the desired contradiction because for these $\phi$, we have $\F_0(v) < \F_0(u)$.
\end{proof}


\section{Deferred proofs} \label{app:proofs}

This appendix collects the proofs of several results stated in the main text. Since these arguments are routine adaptations of well-established proofs in the literature, basically to take our non-linearities and non-homogeneities ($f, A$ and $q$) into account, they have been deferred to this appendix in order to streamline the exposition. In any case, we have preferred to include the proofs for completeness of the whole argument, and in order to make the paper more accessible.

\begin{proof}[\textbf{Proof of Lemma~\ref{lem:almost_minimality}}] 
	As already said in the main body of the text, the proof follows closely \cite[Sections 11.4 and 11.5]{V}, but some modifications are needed to account for $F, A$ and $q$. Let us discuss the precise modifications that are needed:
	
	\begin{enumerate}
		
		\item First, at the beginning of the proof of \cite[Lemma 11.9]{V}, it is deduced that (adopting the notation from \cite{V}) $\F_0(u_j, B) \leq C$ implies that $u_j$ are uniformly bounded in $H^1(B)$. Since $u_j$ are constructed as minimizers of the problem \cite[(11.18)]{V}, posed in the bounded domain $B$, the property above follows by standard arguments using the direct method of Calculus of Variations. In any case, let us make the details more explicit, re-using some computations from Step 1 in Section~\ref{sec:proof}, where we obtain minimizers to a more difficult problem, because it is not posed in a bounded domain.
		
		Note that $\Vol_q(\Omega_{u_j}) = m_j$ for some sequence $m_j \to m$. Indeed, $\Vol_q(B) > m$, so for $j$ large enough we have $\Vol_q(B^{m_j}) < \Vol_q(B)$ and we can prove a Poincaré inequality for $u_j$ as in \eqref{eq:poincare}, with constant $\lambda_1(B^{\widetilde{m_j}})^{-1}$, where, following \eqref{eq:mtilde}, we denote $\widetilde{m_j} := m_j / \underline{q}$. And we can continue the proof of Step 1 in Section~\ref{sec:proof} noting that
		\begin{equation*}
			\lambda - \frac{2b}{\lambda_1(B^{\widetilde{m_j}})}
			=
			\lambda - \frac{2b}{\lambda_1(B^\tm)} \frac{\lambda_1(B^\tm)}{\lambda_1(B^{\widetilde{m_j}})}
			\geq 
			c 
			>
			0
		\end{equation*}
		for some uniform $c$, as soon as we choose $j$ large enough, because ${2b}/{\lambda_1(B^\tm)} < \lambda$ by \eqref{hip:F_quadratic}, and $\lambda_1(B^{\widetilde{m_j}}) \longrightarrow \lambda_1(B^\tm)$ because $m_j \to m$. This allows us to re-run the proof in Step 1 of Section~\ref{sec:proof} for the problem \cite[(11.18)]{V} and deduce the existence of $u_j$, and get a uniform bound for $\norm{\nabla u_j}_{L^2(B)}$ as in \eqref{eq:bound_norm_gradient}, using the uniform constant $c$ in the previous display. With that, since $u_j - u \in H^1_0(B)$ for every $j$, we obtain
		\begin{equation*}
			\norm{u_j}_{L^2(B)}
			\leq 
			\norm{u_j - u}_{L^2(B)} + \norm{u}_2
			\lesssim_B 
			\norm{\nabla (u_j - u)}_2 + \norm{u}_2
			\leq 
			\norm{\nabla u_j}_2 + \norm{\nabla u}_2 + \norm{u}_2,
		\end{equation*}
		by using Poincaré in $B$. Since this bound does not depend on $j$, this finally shows that $u_j$ are uniformly bounded in $H^1(B)$, after passing to a subsequence.\smallskip
		
		\item In the middle of the proof of property (iii) in \cite[Lemma 11.9]{V}, it is asserted that the weak convergence of $u_j$ to $u_\infty$ in $H^1(B)$ implies $\F_0 (u_\infty, B) \leq \liminf_{j \to +\infty} \F_0(u_j, B)$. In our setting, there appears a new problematic term to handle: we need to show
		\begin{equation} \label{eq:less_liminf_gradient_A_0}
			\int_{B} \nabla u_\infty \, A \, \nabla u_\infty^{\miT} \, dx
			\leq 
			\liminf_{j \to +\infty} \int_{B} \nabla u_j \, A \, \nabla u_j^{\miT} \, dx.
		\end{equation}
		But it is not hard to see that this estimate holds.
		Indeed, one can compute 
		\begin{align*} 
			\int_{B} \big( \nabla u_j \, & A \, \nabla u_{j}^{\miT} - \nabla u_\infty \, A \, \nabla u_\infty^{\miT} \big) \, dx
			\\ & =
			\int_{B} \nabla (u_{j} - u_\infty) \, A \, \nabla (u_{j} - u_\infty)^{\miT} \, dx
			+ 2 \int_{B} \nabla (u_{j} - u_\infty) \, A \, \nabla u_\infty^{\miT} \, dx, 
		\end{align*}
		and the first term is non-negative by ellipticity \eqref{hip:A_elliptic_bounded}, whereas the second converges to 0 as $j \to +\infty$ because $\nabla u_j \rightharpoonup \nabla u_\infty$ in $L^2(B)$ (just test against $A \nabla u_\infty^{\miT} \in L^2(B)$ because $u_\infty \in H^1(B_r)$ and $A$ is bounded, see \eqref{hip:A_elliptic_bounded}).	\smallskip
		
		\item Next, in the proof of property (iv) in \cite[Lemma 11.9]{V}, we should use $q\, dx$ instead of the Lebesgue measure. The proof can be readily adapted because, by \eqref{hip:q_positive}, convergence in $L^2(dx)$ is obviously equivalent to convergence in $L^2(q\, dx)$: 
		\begin{equation} \label{eq:iff_convergences_q}
			\int_\Rn \abs{f_j - f}^2 \,dx \underset{j \to \infty}{\longrightarrow} 0 
			\quad \iff \quad 
			\int_\Rn \abs{f_j - f}^2 q \,dx \underset{j \to \infty}{\longrightarrow} 0.
		\end{equation}
		The case of weak convergence is analogous.\smallskip 
		
		\item Finally, in Step 1 in the proof of \cite[Proposition 11.16]{V}, the author uses $\F_0(v_j, B) \geq 0$, which is not true for us because of the presence of $F$. To fix this, we note that actually it suffices that $\F_0(u, B) - \F_0(v_j, B)$ is bounded above uniformly on $j$, which does hold in our setting. Indeed, by Poincaré (because $u-v_j \in H^1_0(B_r)$)
		\begin{align*}
			\int_{B_r} (F(x, u) - F(x, v_j)) dx
			\lesssim_{M_1}
			\int_{B_r} \abs{u-v_j} dx &
			\\ \lesssim_r 
			\left( \int_{B_r} \abs{\nabla (u-v_j)}^2 dx \right)^{1/2}
			& \lesssim
			\norm{\nabla u}_{L^2(B_r)} + \norm{\nabla v_j}_{L^2(B_r)}.
		\end{align*}
		And using this, we conclude
		\begin{equation*}
			\F_0(u, B) - \F_0(v_j, B)
			\leq 
			C(u) - \norm{\nabla v_j}_{L^2(B_r)}^2 + C \norm{\nabla v_j}_{L^2(B_r)}
			\leq 
			C(u) + C.
		\end{equation*}
		
	\end{enumerate}
	
	With these modifications, the proofs in \cite[Sections 11.4 and 11.5]{V} carry over to the present setting.
\end{proof}


\begin{proof}[\textbf{Proof of Proposition~\ref{prop:non_degeneracy}}] 
	The proof is based on ideas of David--Toro, as explained in \cite[Lemma 4.5]{V}.
	First, using $h$ from Lemma~\ref{lem:harmonic_replacement} (adapted to $B_r(x_0)$), Lemma~\ref{lem:harnack} (because $\cL  h \leq M_1$) implies, for any $y \in B_{r/2}(x_0)$,
	\begin{equation*} 
		u(y) 
		\leq 
		h(y)
		\leq 
		C \fint_{\partial B_r(x_0)} h \, d\mathcal{H}^{n-1}
		+ C r^2
		=
		C \fint_{\partial B_r(x_0)} u \, d\mathcal{H}^{n-1}
		+ C r^2
		\leq 
		C (\kappa_0
		+ r_0) r
		=:
		\kappa_1 r.
	\end{equation*}
	
	Then, consider the cut-off function 
	\[
	\phi \in C^\infty_c(B_r(x_0)), 
	\quad 
	0 \leq \phi \leq 1, 
	\quad 
	\phi \equiv 1 \text{ on } B_{r/2}(x_0), 
	\quad 
	\abs{\nabla \phi} \leq 3r^{-1}.
	\]
	Defining the competitor $v := (u-\kappa_1r\phi)_+$, the computation above implies that $v \equiv 0$ in $B_{r/2}(x_0)$. Therefore, noting that $v \leq u$, we get 
	\begin{multline*}
		\Vol_q(\Omega_u \cap B_{r/2}(x_0)) + \Vol_q(\Omega_v \cap B_r(x_0))
		\\ =
		\Vol_q(\Omega_u \cap B_{r/2}(x_0)) + \Vol_q(\Omega_v \cap B_r(x_0) \setminus B_{r/2}(x_0))
		\leq 
		\Vol_q (\Omega_u \cap B_r(x_0)).
	\end{multline*}
	This, the optimality condition \eqref{eq:inwards} (since $\Lambda > 0$ from Proposition~\ref{prop:Lambda_positive}, we can use $\varepsilon := \Lambda/2 > 0$), and the fact that $\abs{F'} \leq M_1$ (from \eqref{eq:def_M_1} and \eqref{hip:negative}), yield 
	\begin{align*}
		\frac{\Lambda}{2} & \, \underline{q} \abs{\Omega_u \cap B_{r/2}(x_0)}
		\leq 
		\frac{\Lambda}{2} \, \Vol_q(\Omega_u \cap B_{r/2}(x_0))
		\leq 
		\frac{\Lambda}{2} \, \Vol_q(\Omega_u \cap B_r(x_0)) - \frac{\Lambda}{2} \, \Vol_q(\Omega_v \cap B_r(x_0))
		\\ & \leq 
		\int_{B_r(x_0)} \nabla v \, A \, \nabla v^{\miT} \, dx 
		- \int_{B_r(x_0)} \nabla u \, A \, \nabla u^{\miT} \, dx
		- 2 \int_{B_r(x_0)} \big( F(x, v) - F(x, u) \big) \, dx
		\\ & \leq 
		\int_{B_r(x_0)} \nabla (u-\kappa_1 r \phi) \, A \, \nabla (u-\kappa_1 r \phi)^{\miT} \, dx 
		- \int_{B_r(x_0)} \nabla u \, A \, \nabla u^{\miT} \, dx
		+ 2 M_1 \abs{B_r} \norm{u - v}_\infty
		\\ & \leq
		2 \kappa_1 r \norm{A}_\infty \int_{B_r(x_0)} \abs{\nabla u} \abs{\nabla \phi} \, dx
		+ \kappa_1^2 r^2 \norm{A}_\infty \int_{B_r(x_0)} \abs{\nabla \phi}^2 \, dx 
		+ 2 M_1 \abs{B_r} \kappa_1 r
		\\ & \leq 
		(6 \kappa_1 \lambda^{-1} L + 9 \kappa_1^2 \lambda^{-1} + 2M_1\kappa_1 r_0) \abs{B_r},
	\end{align*}
	which implies that 
	\[
	\abs{\Omega_u \cap B_{r/2}(x_0)}
	\leq 
	\kappa_2 \abs{B_r}, 
	\quad \text{where } 
	\kappa_2 
	:=
	2 \kappa_1 \frac{6 \lambda^{-1} L + 9 \kappa_1 \lambda^{-1} + 2M_1 r_0}{\Lambda \underline{q}}.
	\]
	
	To finish the proof, we may just follow the remaining part of \cite[Lemma 4.5]{V} after choosing $\kappa_0$ and $r_0$ small enough. Let us include the details for completeness. Using the estimates already obtained, we compute 
	\begin{equation*}
		\int_{B_{r/2}(x_0)} u \, dx 
		\leq 
		\norm{u}_{L^\infty(B_{r/2}(x_0))} \abs{\Omega_u \cap B_{r/2}(x_0)}
		\leq 
		\kappa_1 \kappa_2 r \abs{B_r}.
	\end{equation*}
	Now, given any $y_0 \in B_{r/8}(x_0)$, it clearly holds $B_{r/4}(y_0) \subset B_{r/2}(x_0)$. This, and an easy integration in polar coordinates, allow us to find some $\rho = \rho(y_0) \in [r/8, r/4]$ for which 
	\begin{equation*}
		\int_{\partial B_\rho(y_0)} u \, d\mathcal{H}^{n-1}
		\leq 
		\fint_{r/8}^{r/4} \int_{\partial B_s(y_0)} u \, d\mathcal{H}^{n-1} \, ds 
		\leq 
		\frac{1}{r/8} \int_{B_{r/2}(x_0)} u \, dx
		\leq 
		8 \, \kappa_1 \kappa_2 \abs{B_r},
	\end{equation*}
	which trivially implies that (recalling that $\rho \geq r/8$ and that $n \geq 1$)
	\begin{equation*}
		\fint_{B_\rho(y_0)} u \, d\mathcal{H}^{n-1} 
		\leq 
		\kappa_3 \, \rho, 
		\qquad \text{where }
		\kappa_3 
		:=
		8^{n+1} \kappa_1 \kappa_2.
	\end{equation*} 
	Recalling the definitions of $\kappa_1$ and $\kappa_2$, we can choose $\kappa_0$ and $r_0$ small enough (depending on $n, \lambda, \Lambda, L, M_1$ and $\underline{q}$) so that $\kappa_3 \leq \kappa_0$. And this means that the property in the assumption of the proposition translates to smaller scales.
	
	Let us iterate this. Indeed, fix $y_0 \in B_{r/8}(x_0)$. Starting at $y_0, \rho_0 := \rho \in [r/8, r/4]$ from last paragraph, use the method above (iteratively) to find a sequence of $\rho_j$ (for $j \geq 1$) satisfying 
	\begin{equation*}
		\frac{\rho_j}{8} \leq \rho_{j+1} \leq \frac{\rho_j}{4}, 
		\quad \text{ and } \quad 
		\fint_{\partial B_{\rho_j} (y_0)} u \, d\mathcal{H}^{n-1}
		\leq 
		\kappa_0 \rho_j, 
		\qquad 
		\text{for every } j \geq 0.
	\end{equation*}
	Since $u$ is Lipschitz (Proposition~\ref{prop:lipschitz}) and $\rho_j \to 0$, we finally obtain the desired result:
	\begin{equation*}
		0
		\leq 
		u(y_0) 
		=
		\lim_{j \to +\infty} \fint_{\partial B_{\rho_j}(y_0)} u \, d\mathcal{H}^{n-1}
		\leq 
		\lim_{j \to +\infty} \kappa_0 \rho_j
		= 
		0, 
		\qquad 
		y_0 \in B_{r/8}(x_0).
	\end{equation*}
	The proposition then follows.
\end{proof}


\begin{proof}[\textbf{Proof of Corollary~\ref{corol:blowup_minimizer}}] 
	Let us prove the stationary condition first. Recall that, as defined at the beginning of this section, $\F_{\Lambda, 0}(u) 
	=
	\int_\Rn \nabla u \, A_0 \, \nabla u^{\miT} \, dx 
	+ \Lambda \Vol_{q_0} (\Omega_u),$
	where we note that $\Vol_{q_0}(\Omega_u)$ is just $q(0) \abs{\Omega_u}$. Using the expressions obtained in Lemma~\ref{lem:first_variation}, noting that $A_0$ and $q_0$ are constant functions, we obtain that for $\xi \in C^\infty_c(\Rn; \Rn)$,
	\begin{equation*}
		\delta \F_{\Lambda, 0}(u)[\xi]
		=
		\int_\Rn \Big( -2\nabla u \, D\xi \, A(0) \, \nabla u^{\miT}
		+ \nabla u \, A(0) \, \nabla u^{\miT} \div \xi \Big) \, dx
		+ \Lambda q(0) \int_{\Omega_u} \div \xi \, dx.
	\end{equation*}
	
	Now, recall that in Lemma~\ref{lem:first_variation} we obtained that $\delta \F_\Lambda(u)[\xi] = 0$ for every $\xi \in C^\infty_c(\Rn; \Rn)$. Rescaling that, we obtain $\delta \F_{\Lambda, r_j}(u_{r_j})[\xi] = 0$ for every $\xi \in C^\infty_c(\Rn; \Rn)$ ($\delta \F_{\Lambda, r_j}$ has an explicit expression as in Lemma \ref{lem:first_variation}, but with $A_{r_j}, F_{r_j}$ and $q_{r_j}$ replacing $A, F$ and $q$). Now, using all the convergences in Lemmas~\ref{lem:convergence_parameters} and \ref{lem:convergence_blowups}, we can take limits in each term of the explicit expression of $\delta \F_{\Lambda, r_j}(u_{r_j})[\xi] = 0$ to obtain $\delta \F_{\Lambda, 0} (u_0)[\xi] = 0$, for each $\xi \in C^\infty_c (\Rn; \Rn)$. 
	
	To prove the local minimality of $u_0$, we can now use Lemmas~\ref{lem:convergence_parameters} and \ref{lem:convergence_blowups} and simply follow the proof in the second half of \cite[Lemma 6.3]{V}:  everything now works directly in our setting. 
\end{proof}


\begin{proof}[\textbf{Proof of Lemma~\ref{lem:blowups_harmonic}}]
	Since $u_{r_j}$ are Lipschitz with constant $L$ (Proposition~\ref{prop:lipschitz}), by Arzelà-Ascoli and the uniform convergence in Lemma~\ref{lem:convergence_blowups}, $u_0$ is Lipschitz, too. Concretely, continuity implies that $\Omega_{u_0}$ is open. 
	In turn, the PDE follows as in Corollary~\ref{corol:EL2} from the fact that $u_0$ is a local minimizer of $\F_{\Lambda, 0}$ in $B_1$ (see Corollary~\ref{corol:blowup_minimizer}).
\end{proof}

\begin{proof}[\textbf{Proof of Lemma~\ref{lem:blowups_not_zero}}] 
	Since $0 \in \pom_u$, the non-degeneracy property (Proposition~\ref{prop:non_degeneracy}) 
	implies, for $r$ small enough, say $0 < r < r_0$, that $\fint_{\partial B_r} u \, d\mathcal{H}^{n-1} \geq \kappa_0 r$. It is easy to see, using a change of variables, that this implies (if $r_j$ is small enough) that $\fint_{\partial B_r} u_{r_j} \, d \mathcal{H}^{n-1} \geq \kappa_0 r$ for $r \in (0, 1)$. (Note that the constant does not change after the rescaling). Using also \eqref{eq:rescaling_bounds_u}, we can estimate 
	\begin{equation*}
		L \, |\Omega_{u_{r_j}} \cap B_1|
		\geq 
		\int_{B_1} u_{r_j} \,dx 
		=
		\int_0^1 \int_{\partial B_r} u_{r_j} \, d\mathcal{H}^{n-1} \, dr
		\gtrsim 
		\int_0^1 \kappa_0 r^n \, dr 
		\gtrsim 
		\kappa_0, 
	\end{equation*}
	so $|\Omega_{u_{r_j}} \cap B_1| \gtrsim 1$ uniformly on $j$. By the convergence of indicator functions in Lemma~\ref{lem:convergence_blowups}, we obtain $\abs{\Omega_{u_0} \cap B_1} \gtrsim 1$. The result follows.
\end{proof}

\begin{proof}[\textbf{Proof of Lemma~\ref{lem:density_upper_bound}}] 
	Using the function~$h$ from Lemma~\ref{lem:harmonic_replacement}, which is $\cL $-superharmonic (because $F' \geq 0$, see \eqref{hip:negative}), we can compute, for small $r$,
	\begin{equation*}
		h(x)
		\gtrsim 
		\fint_{\partial B_r(x_0)} h \, d\mathcal{H}^{n-1}
		\geq 
		\kappa_0 r, 
		\qquad x \in B_{r/2}(x_0).
	\end{equation*}
	Here we have used Lemma~\ref{lem:harnack} (applied to the function $-h(x) + \norm{u}_\infty$, which is non-negative and $\cL $-subharmonic) and Proposition~\ref{prop:non_degeneracy}. This estimate and \eqref{eq:AC_3.2_estimate} allow us to follow the proof of \cite[Lemma 5.1]{V}. Indeed, since by \eqref{eq:rescaling_bounds_u} it holds $u \leq L \varepsilon r$ in $B_{\varepsilon r}(x_0)$, we can compute (using $\varepsilon$ small enough so that $\varepsilon L \ll \kappa_0$)
	\begin{multline*}
		\abs{B_{\varepsilon r}}
		\lesssim 
		\frac1r \int_{B_{\varepsilon r}(x_0)} (h-u) \, dx
		\leq 
		\frac1r \int_{B_r(x_0)} (h-u) \, dx
		\\ \lesssim 
		\abs{B_r}^{1/2} \left( \int_{B_r(x_0)} \abs{\nabla (h-u)}^2 \, dx \right)^{1/2}
		\lesssim 
		\abs{B_r}^{1/2} \Big(\lambda^{-1} \overline{q} \, \Lambda \abs{\{u = 0\} \cap B_r(x_0)} \Big)^{1/2},
	\end{multline*} 
	where we have used Poincaré's inequality and Lemma~\ref{lem:harmonic_replacement}. The result then follows.
\end{proof}


\begin{proof}[\textbf{Proof of Proposition~\ref{prop:overdetermined}}] 
	With the auxiliary results we have established in Section~\ref{sec:boundary}, the proof parallels \cite[Proposition 9.18]{V}. Let us see the differences.  To classify blow-up limits in our context, it is more convenient to run the proof in different coordinates, so that  in the limit we encounter the Laplacian (instead of the constant coefficient PDE in Lemma~\ref{lem:blowups_harmonic}). 
	
	\subsubsection*{Step 1: Upper bound.} Assume that $0 \in \pom_u$ (as discussed along this section, we can do this without loss of generality). As in the proof of Lemma~\ref{lem:1-homogeneous}, define 
	\[
	\widehat{u}(x) := u(x \sqrt{A(0)}), 
	\qquad 
	\widehat{u}_0(x) := u_0(x \sqrt{A(0)}),
	\]
	so that the rescalings $\widehat{u}_{r_j}$ converge to $\widehat{u}_0$. It is easy to see that $-\Delta \widehat{u}_0 = 0$ in $\Omega_{\widehat{u}_0}$ (changing variables in Lemma~\ref{lem:blowups_harmonic}, or directly because $\widehat{u}_0$ is a local minimizer of $\widehat{\F}_{\Lambda, 0}$, see Step 1 of Lemma~\ref{lem:1-homogeneous}).	
	
	Let $\varphi \in C^\infty(\Rn)$ touch $u$ from below at $0$. Then, $\widehat{\varphi}(x) := \varphi(x \sqrt{A(0)})$ touches $\widehat u$ from below at $0$.
	If we rescale $\widehat \varphi_{r_j}(x) := r_j^{-1} \widehat \varphi(r_j x)$, it is clear that $\widehat \varphi_{r_j} \longrightarrow \widehat \varphi_0$ uniformly in $B_1$, as $j \to \infty$, perhaps up to a subsequence, because $\widehat \varphi$ is smooth. The smoothness of $\widehat \varphi$ implies that $\widehat \varphi_0(x) = \nabla \widehat \varphi(0) \cdot x$. Without loss of generality, using a rotation we may write $\widehat \varphi_0(x) = B \, x_n$ for $B := \abs{\nabla \widehat \varphi(0)}$. 
	Moreover, we can assume that $B > 0$, for otherwise $\abs{\nabla \varphi(0)} \leq \sqrt{\Lambda}$ is trivially satisfied.	 
	Since $\widehat{u}_0 \geq \widehat \varphi_0$ because $\widehat \varphi$ touches $\widehat{u}$ from below, this implies that $\widehat{u}_0 > 0$ in $\{x_n > 0\}$. 
	
	Thus $\widehat{u}_0$ is a 1-homogeneous (see Lemma~\ref{lem:1-homogeneous}) harmonic function in $\{x_n > 0\} \subset \{\widehat{u}_0 > 0\}$. Therefore, by \cite[Lemma 9.16]{V}, 
	\[
	\text{either } \quad  \widehat{u}_0(x) = \alpha x_n^+, 
	\quad \text{ or } \quad \widehat{u}_0(x) = \alpha x_n^+ + \beta x_n^-, 
	\qquad 
	\text{for some $\alpha, \beta > 0$}.
	\]
	But the latter possibility contradicts Lemma~\ref{lem:density_upper_bound}, so $\widehat{u}_0(x) = \alpha x_n^+$ for some $\alpha > 0$.
	Hence testing the stationarity of $\widehat{u}_0$ for $\widehat{F}_{\Lambda, 0}$ (see Step 1 in Lemma~\ref{lem:1-homogeneous}) against $\xi(x) := (0, \ldots, 0, \xi_n(x))$, we find:
	\begin{equation*}
		0
		=
		\delta \widehat\F_{\Lambda, 0} (\widehat u_0)[\xi]
		=
		\int_{\R^n_+} \left( -2\alpha^2 \partial_{x_n} \xi_n + \alpha^2 \div \xi + \Lambda q(0) \div \xi \right) dx
		= 
		\int_{\R^n_+} (-\alpha^2 + \Lambda q(0)) \partial_{x_n} \xi_n \, dx.
	\end{equation*}
	Since $\xi$ is arbitrary, this yields $\alpha = \sqrt{\Lambda q(0)}$. 
	
	Thus we have shown that $$\widehat{u}_0(x) = \sqrt{\Lambda q(0)} \left(\dfrac{\nabla \widehat \varphi(0)}{\abs{\nabla \widehat \varphi (0)}} \cdot x\right)_+,$$ which implies that $$u_0(x) = \dfrac{\sqrt{\Lambda q(0)}}{(\nabla \varphi(0) \, A(0) \, \nabla \varphi(0)^{\miT})^{1/2}} \, (\nabla \varphi(0) \cdot x)_+.$$
	Note that since $u_0 \geq \varphi_0$ and $ u_0(0) = 0 = \varphi_0(0)$, we have $$\abs{\nabla \varphi(0)} = \abs{\nabla \varphi_0(0)} \leq \abs{\nabla u_0(0)} = \sqrt{\Lambda q(0)} \dfrac{\abs{\nabla \varphi(0)}}{(\nabla \varphi(0) \, A(0) \, \nabla \varphi(0)^{\miT})^{1/2}},$$
	which finally gives $\nabla \varphi(0) \, A(0) \, \nabla \varphi(0)^{\miT} \leq \Lambda q(0)$.
	
	\subsubsection*{Step 2: Lower bound.} Similarly, let $\varphi \in C^\infty(\Rn)$ touch $u$ from above at $0 \in \pom_u$, and change coordinates again so that $\widehat \varphi$ touches $\widehat u$ from above at $0$. Doing again the blow-up procedure, we obtain blow-up limits $\widehat u_0$ and $\widehat \varphi_0$. By Lemma~\ref{lem:blowups_not_zero}, $\widehat u_0 \not \equiv 0$, whence $B := \abs{\nabla \widehat \varphi(0)} > 0$ because $\widehat \varphi$ touches $\widehat u$ from above. Therefore, since $\widehat u_0 \leq \widehat \varphi_0$, it holds 
	\begin{equation*}
		\{\widehat u_0 > 0\} \subset \{\widehat \varphi_0 > 0\} \subset \big\{ x \in \Rn : \nabla \varphi(0) \cdot x > 0 \big\},
	\end{equation*} 
	where we note that the last set is simply a (possibly rotated) half-space. Thus, since $\widehat u_0$ is 1-homogeneous and harmonic, \cite[Lemma 9.15]{V} implies that $\widehat{u}_0(x) = \alpha \, \Big(\dfrac{\nabla \widehat \varphi(0)}{\abs{\nabla \widehat \varphi(0)}} \cdot x\Big)_+$ for some $\alpha > 0$. As in Step 1, we obtain $\alpha = \sqrt{\Lambda q(0)}$, and later $\nabla \varphi(0) \, A(0) \, \nabla \varphi(0)^{\miT} \geq \Lambda \, q(0)$.	
\end{proof}


\begin{proof}[\textbf{Proof of Proposition~\ref{prop:singular}}] 
	The proposition follows from the abstract result in \cite[Proposition 10.13]{V} after we verify that the hypotheses hold in our setting. 
	
	The non-degeneracy was shown in Proposition~\ref{prop:non_degeneracy}. Then, we proved the convergence of blow-ups in Lemma~\ref{lem:convergence_blowups}. Moreover, the homogeneity and minimality of blow-up limits (with respect to $\F_{\Lambda, 0}$, exactly the same energy functional as in \cite{V} because $F$ terms disappeared in the limit) was obtained in Lemmas~\ref{lem:convergence_blowups} and \ref{lem:1-homogeneous}. Lastly, we have verified the uniform $\varepsilon$-regularity property in the proof of Proposition~\ref{prop:regular}. 
	
	The definition of the uniform $\varepsilon$-regularity property must be  modified slightly with respect to that in \cite{V}: the shape of our blow-up limits is given by $h_{\nu, x_0}(x) = \sqrt{\Lambda q(x_0)} ((x-x_0) \cdot \nu)_+$ for $\nu \in B_{\lambda^{1/2}, \lambda^{-1/2}}$; whereas in \cite{V}, the vector should lie exactly on $\partial B_1$, and $q(x_0)$ is not present. In any case, since these modifications are uniformly controlled (with constants depending on $\lambda, \underline{q}$ and $\overline{q}$), the program in \cite[Proposition 10.13]{V} carries over to the present case almost verbatim. Indeed, our new definition of uniform $\varepsilon$-regularity is sufficient to run the proof of \cite[Lemma 10.7]{V}, which is the main ingredient that needs to be adapted in \cite[Proposition 10.13]{V} to work in our setting.
	
	Also, the fact that \cite[Proposition 10.13]{V} is only stated for bounded sets is not a problem, since we could just truncate to $B_j$ and then consider the union with $j \in \N$. Recall that we showed that $\Omega_u$ is bounded in Proposition~\ref{prop:Omega_bounded}.
\end{proof}


\end{document}